\documentclass[journal,10pt,twocolumn,final,twoside]{IEEEtran}
\usepackage{colortbl,amsmath,amssymb,graphicx,algorithmic,algorithm,cases,bm,dsfont,units,array,color,cite,cancel,multirow,enumerate,subcaption}
\usepackage[bookmarks=false]{hyperref}

\def\w{{\boldsymbol w}}
\def\d{{\boldsymbol d}}
\def\x{{\boldsymbol x}}
\def\h{{\boldsymbol h}} 
\def\y{{\boldsymbol y}}
\def\v{{\boldsymbol v}}
\def\g{{\boldsymbol g}}
\def\z{{\boldsymbol z}}

\def\B{{\mathcal B}}
\def\D{\mathcal{H}}

\def\G{{\mathcal{G}}}
\def\R{{\mathcal{R}}}
\def\bpsi{{\boldsymbol \psi}}
\def\bSigma{{\boldsymbol \Sigma}}
\def\H{{\boldsymbol H}}
\def\T{{\mathsf{T}}}
\def\A{{\mathcal A}}

\def\E{{\mathbb{E}}}
\def\Tr{{\mathsf{Tr}}}

\def\AT{{\textrm{AT}}}

\def\hyph{-\penalty0\hskip0pt\relax}
\newcommand{\qed}{\blacksquare}
\newtheorem{theorem}{Theorem}

\newtheorem{assumption}{Assumption}
\graphicspath{{Figures/}}
\renewcommand{\qed}{\hfill$\blacksquare$}
\newtheorem{lemma}{Lemma}

\newtheorem{corollary}{Corollary}
\definecolor{BLUE}{rgb}{0,0,1}
\definecolor{orange}{RGB}{255,127,0}
\definecolor{lightblue}{RGB}{140,215,255}
\definecolor{green}{RGB}{0,100,0}
\newcommand{\psw}{{\scriptstyle{\mathcal{W}}}}
\newcommand{\sw}{\boldsymbol{{\scriptstyle{\mathcal{W}}}}}

\definecolor{r0}{RGB}{255,76,55}
\definecolor{r0}{RGB}{255,255,255}
\definecolor{r1}{RGB}{255,255,255}
\definecolor{r2}{RGB}{255,255,255}
\definecolor{r3}{RGB}{255,255,255}
\definecolor{r4}{RGB}{255,255,255}
\definecolor{r5}{RGB}{255,255,255}
\definecolor{r6}{RGB}{255,255,255}
\definecolor{r7}{RGB}{255,255,255}
\definecolor{r8}{RGB}{255,255,255}
\definecolor{r9}{RGB}{255,255,255}

\begin{document}

\title{Excess-Risk of Distributed Stochastic Learners}

\author{Zaid~J.~Towfic,~\IEEEmembership{Member,~IEEE,}
        Jianshu~Chen,~\IEEEmembership{Member,~IEEE,}
        and~Ali~H.~Sayed,~\IEEEmembership{Fellow,~IEEE}\vspace{-1\baselineskip}% <-this % stops a space
\thanks{
Copyright (c) 2016 IEEE. Personal use of this material is permitted. However, permission to use this material for any other purposes must be obtained from the IEEE by sending a request to pubs-permissions@ieee.org.
}
\thanks{This work was supported in part by NSF grants CCF-1011918, ECCS-1407712, and CCF-1524250. A short version of this work appears in the conference presentations \cite{MLSP2,towfic2013distributed}.}
\thanks{Z. J. Towfic and J. Chen were with Department of Electrical Engineering, University of
California, Los Angeles. Z. J. Towfic is currently with MIT Lincoln Laboratory, Lexington, MA 02421. J. Chen is currently with Microsoft Research, Redmond, WA 98052. This work was performed while they were PhD students at UCLA. A. H. Sayed is with Department of Electrical Engineering, University of
California, Los Angeles, CA 90095. (e-mail: \{ztowfic,cjs09,sayed\}@ucla.edu).}% <-this % stops a space
} 

\markboth{IEEE Transactions on Information Theory, VOL.~XX, NO.~XX, MONTH~2016}{Towfic, Chen, Sayed: 
Excess-Risk of Distributed Stochastic Learners}

\maketitle

\begin{abstract}%
%\boldmath
This work studies the learning ability of consensus and diffusion distributed learners from continuous streams of data arising from different but related statistical distributions. Four distinctive features for diffusion learners are revealed in relation to other decentralized schemes even under left-stochastic combination policies. First, closed-form expressions for the evolution of their excess-risk are derived for strongly-convex risk functions under a diminishing step-size rule. Second, using these results, it is shown that the diffusion strategy improves the asymptotic convergence rate of the excess-risk relative to non-cooperative schemes. Third, it is shown that when the in-network cooperation rules are designed optimally, the performance of the diffusion implementation can outperform that of naive centralized processing. Finally, the arguments further show that diffusion outperforms consensus strategies asymptotically, and that the asymptotic excess-risk expression is invariant to the particular network topology. The framework adopted in this work studies convergence in the stronger mean-square-error sense, rather than in distribution, and develops tools that enable a close examination of the differences between distributed strategies in terms of asymptotic behavior, as well as in terms of convergence rates. 

\end{abstract}



% \fi
%\IEEEpeerreviewmaketitle

\allowdisplaybreaks
\section{Introduction} 
\label{sec:intro} 
Machine learning applications rely on the premise that it is possible to benefit from leveraging information collected from different users. The range of benefits, and the computational cost necessary to analyze the data, depend on how the information is mined. It is sometimes advantageous to aggregate  the information from all users at a central location for processing and analysis. Many current implementations rely on this centralized approach. However, the rapid increase in the number of users, coupled with privacy and communication constraints  related to transmitting,  storing, and analyzing huge amounts of data at remote central locations, have been serving as strong motivation for the development of decentralized solutions to learning and data mining \cite{zinkevich2010parallelized,yan,Ram,Paolo, theodoridis2011adaptive, zargham2011accelerated, yildiz2009distributed, shamir2014distributed, schizas2008consensus}.
 
In this work, we study the \emph{distributed} real-time prediction problem over a network of $N$ learners. We assume the network is connected, meaning that any two arbitrary agents are either connected directly or by means of a path passing through other agents. We do not expect the agents to share their data sets but only a parameter vector (or a statistic) that is representative of their local information. Such networks serve as useful models for peer-to-peer networks and social networks. The objective of the learning process is for all nodes to minimize some objective function, termed the \emph{risk} function, in a distributed manner. We shall compare the performance of cooperative and non-cooperative solutions by examining the gap between the risk achieved by the distributed implementations and the risk achieved by an oracle solution with access to the true distribution of the input data; we shall refer to this gap as the \emph{excess-risk}.

Among other contributions, this work studies stochastic gradient-based distributed strategies that are shown here to converge in the mean-square-error sense when a decaying step-size sequence is used, and that are also shown to outperform other implementations, even under left-stochastic 
combination rules \cite{aysal2009broadcast,MorralBianchi,jianshu_part1,jianshu_part2}. Specifically, we will show that the strategies under study achieve a better convergence rate than non-cooperative algorithms, and we will also explain why diffusion strategies outperform other distributed solutions such as those relying on consensus constructions or on doubly-stochastic combination policies \cite{Bianchi,yan,KarMoura2,nedic2,kar2012distributed}, as well as na\"{i}ve centralized algorithms \cite{zinkevich2010parallelized}. It was previously shown that the diffusion strategies outperform their consensus-based counterparts in the constant step-size scenario \cite{shine_diffusion_consensus}. We analytically show that the same conclusion holds in the diminishing step-size scenario even as the step-size decays. We also illustrate in the simulations that while diffusion and consensus-based algorithms have the same computational complexity, it turns out that diffusion algorithms reduce the overshoot during the transient phase. In comparison to the useful work \cite{Bianchi}, our formulation  studies convergence in the stronger mean-square-error sense, and develops analysis tools that do not depend on using the central limit theorem or on studying convergence in a weaker distributional sense. In addition, unlike the works \cite{Agarwal_NIPS,shamir2014distributed,zinkevich2010parallelized,Bach_NIPS}, we are not solely interested in bounding the excess-risk. Instead, we are interested in obtaining a closed-form expression for the asymptotic excess-risk of distributed and non-distributed strategies in order to compare and optimize their absolute asymptotic performance. 

Recently, there has also been interest in primal-dual approaches for distributed optimization \cite{PrimalDual,ouyang2013stochastic,boyd2011distributed,Paolo,schizas2009distributed}. Generally, these approaches are studied in the deterministic optimization context where the iterates are not prone to noise or when the risk function is non-differentiable. It was demonstrated in \cite{PrimalDual} that the primal diffusion strategy studied in this manuscript also outperforms augmented Lagrangian and Arrow-Hurwicz primal-dual algorithms in the stochastic constant-step-size setting in both stability and steady-state performance. It is possible to carry out similar comparisons in the diminishing step-size scenario, but this manuscript is focused on the study of primal approaches. As the extended analysis and derivation in later sections and appendices show, this case is already demanding enough to warrant separate consideration in this work.

The techniques developed will allow us to examine analytically and closely the differences between distributed strategies in terms of asymptotic behavior, as well as in terms of rates of convergence by exploiting properties of Gamma functions and the convergence properties of products of infinitely many scaling coefficients. For instance, when the noise profile is uniform across all agents, one of our conclusions will be to show that the convergence rate of diffusion strategies is in the order of  $\Theta(1/N i)$, where the notation $f(n) = \Theta(g(n))$ means that the sequence $f(n)$ decays at the same rate as $g(n)$ for sufficiently large $n$, i.e., there exist positive constants $c_1$ and $c_2$ and an integer $n_o$ such that $c_1g(n)\leq f(n)\leq c_2g(n)$ for all $n > n_o$. This rate is consistent with the result established for consensus implementations under doubly-stochastic 
combination policies in \cite{KarMoura2,kar2012distributed,Bianchi} albeit under a weaker convergence in distribution sense, where it was argued that the estimation error approaches a Gaussian distribution whose covariance matrix scales as $1/Ni$. On the other hand, when the noise profile is non-uniform across the agents, the analysis will show  that diffusion methods can surpass this rate. These and other useful conclusions will follow from the detailed mean-square and convergence analyses that are carried out in the sequel. The theoretical findings are illustrated by simulations in the last section.

For ease of reference, we summarize here the main conclusions in the manuscript: 
\begin{itemize}
  \item We derive a closed-form expression (and not only a bound) for the asymptotic excess-risk curve of the distributed strategies. 
  \item We analyze the derived expression to conclude that the asymptotic performance depends on the network topology solely through the Perron vector of the combination matrix used in the strategy. In this way, different topology structures with the same Perron vector are shown to attain the same asymptotic performance. That is, the full eigen-structure of the topology are become irrelevant in the asymptotic regime.
  \item We show that once the Perron vector is optimized to minimize the asymptotic excess-risk, it is possible to construct a combination matrix with that Perron vector in order to attain the optimal performance in a fully distributed manner.
  \item We compare the asymptotic excess-risk performance of the diffusion strategy to centralized and non-cooperative strategies to conclude that the diffusion strategy can attain the performance of a weighted centralized strategy asymptotically.
  \item We compare the asymptotic excess-risk performance of the diffusion strategy to consensus distributed strategies to conclude that the asymptotic excess-risk curve of the consensus strategy will be worse than that of the diffusion strategy.
  \item We verify our conclusions through simulations.
\end{itemize}

\textbf{Notation}: Random quantities are denoted in boldface. Throughout the manuscript, all vectors are column vectors. Matrices are denoted in capital letters, while vectors and scalars are denoted in lowercase letters. Network variables that aggregate variables across the network are denoted in calligraphic letters. Unless otherwise noted, the notation $\|\cdot\|$ refers to the Euclidean norm for a vector and to the matrix norm that is induced by Euclidean norm for vectors. Furthermore, the notation $\otimes$ denotes the Kronecker product operation \cite[p.~139]{laub}. The notation $\mathds{1}_N$ denotes a vector of dimension $N\times 1$ with all its elements equal to one.

\section{Problem Formulation and Algorithms}
\label{sec:prob_formulation}
Consider a network of $N$ learners. Each learner $k$ is subject to a streaming sequence of independent data samples $\x_{k,i}$, for $i=1,2,\ldots$, arising from some fixed distribution $\mathcal{X}_k$. The goal of each agent is to learn the vector $w^o\in\mathbb{R}^{M\times 1}$ that optimizes the average of some \emph{loss function}, say, $\E Q(w,\x_{k,i})$, where the expectation is over the distribution of the data $\x_{k,i}$ and $w\in\mathbb{R}^{M\times 1}$ is the vector variable of optimization. For example, in order to learn the hyper-plane that best separates feature data $\h_{k,i}\in\mathbb{R}^{M\times 1}$ belonging to one of two classes $\y_{k}(i) \in \{+1,-1\}$, a regularized logistic-regression (RLR) algorithm would minimize the expected value of the following loss function over $w$ (with the expectation computed over the distribution of the data $\x_{k,i} \triangleq \{\h_{k,i},\y_{k,i}\} \sim \mathcal{X}_k$) \cite{Bach_NIPS}:
\begin{align}
	Q^\textrm{RLR}(w,\h_{k,i},\y_{k}(i)) \triangleq \frac{\rho}{2} \|w\|^2 \!+\! \log(1\!+\!e^{-\y_{k}(i) \h_{k,i}^\T \!w})
	\label{eq:RLR}
\end{align}
while a mean-square-error algorithm would minimize the expected value of the quadratic loss \cite{widrow1960adaptive} (also referred to as the ``delta rule'' \cite{bottou-mlss-2004}):
\begin{align}
	Q^\textrm{LMS}(w,\h_{k,i},\y_{k}(i)) \triangleq \frac{\rho}{2} \|w\|^2 + (\y_{k}(i) - \h_{k,i}^\T w)^2 
	\label{eq:LMS}
\end{align}
The expectation of the loss function over the distribution of the data is referred to as the \emph{risk function} \cite[p.~16]{vapnik}:
\begin{align}
	J_k(w) \triangleq \E\{Q(w,\x_{k,i})\}
	\quad\quad\quad\textrm{[risk function]}
	\label{eq:J_w} 
\end{align} 
and we denote the optimizer of \eqref{eq:J_w} by $w^o$:
\begin{align}
	w^o &\triangleq \underset{w}{\arg \min\ } J_k(w)
	\label{eq:w_o}
\end{align}
where $w^o$ is unique when $J_k(w)$ is strongly-convex, which we shall assume for the remainder of the manuscript. The assumption of strong-convexity of the risk function is important in practice since the convergence rate of most stochastic approximation strategies will be significantly reduced when the condition does not hold \cite{Bach_NIPS}. This is not a limitation in most problems arising in the context of adaptation and learning since regularization (such as $\ell_2$) is often used and it helps ensure strong convexity. The risk function can be viewed as a measure of the ``prediction-error'' of a classifier or regression method since it evaluates the performance of the method against samples taken from the distribution of the input data that have not yet been observed by the classifier/regressor \cite[p.~16]{vapnik}. The risk serves as a measure about how well an estimate $w$ will perform on a new sample $\x_{k,i}\sim\mathcal{X}_k$ on average. For this reason, the risk is also referred to as the generalization ability of the classifier.

We will assume for the remainder of this exposition that the optimizer $w^o$ is the same for all nodes $k=1,\ldots,N$. This case is common in both machine learning (where, for example, $\mathcal{X}_k = \mathcal{X}$ for all $k$) and distributed inference applications where the distributions $\mathcal{X}_k$ are dependent on a common parameter vector $w^o$ to be optimized (see Sec. \ref{sec:sim} further ahead). In order to measure the performance of each learner, we define the excess-risk (ER) at node $k$ as:
\begin{align}
	\textrm{ER}_k(i) \triangleq \E\{J_k(\w_{k,i-1}) - J_k(w^o)\}
	\label{eq:ER_k}
\end{align}
where $\w_{k,i-1}$ denotes the estimator of $w^o$ that is computed by node $k$ at time $i-1$ (i.e., it is the estimator that is generated observing past data within the neighborhood of node $k$). The excess-risk is non-negative because $J_k(w)$ is strongly-convex and, therefore, $J_k(w') > J_k(w^o)$ for all $w' \neq w^o$. 
The expectation in \eqref{eq:ER_k} is over the data since $\w_{k,i-1}$ is a random quantity that depends on all the data samples up to time $i-1$ (i.e., $\{\x_{\ell,j}\}_{j \le i-1, 1\leq \ell\leq N}$). The dependence on the data from the other agents arises from the network topology.
	Our interest in this work is to characterize the convergence rate, to zero, of the excess-risk for various distributed and non-distributed strategies of learning $\w_{k,i-1}$ for a given loss function. We also derive closed-form expressions for the asymptotic excess-risk and compare the absolute-value of the excess-risk curves for algorithms that converge at the same rate.

\begin{figure*}[t!]
	\centering
    \begin{subfigure}[t]{0.2\textwidth}
        \includegraphics[width=\textwidth]{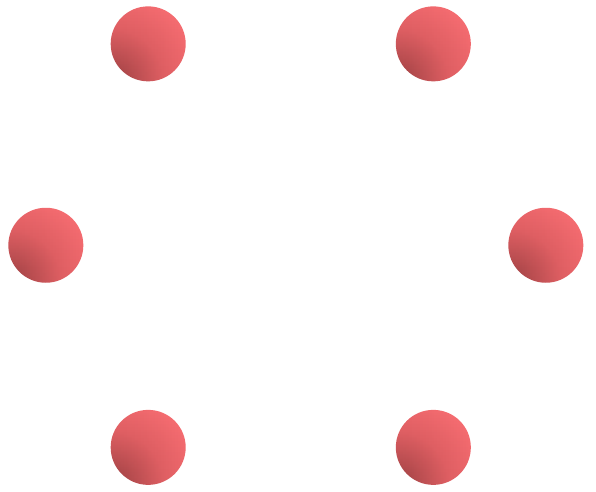}
        \caption{Non-cooperative}
    \end{subfigure}%
    \quad\quad
    \begin{subfigure}[t]{0.2\textwidth}
        \includegraphics[width=\textwidth]{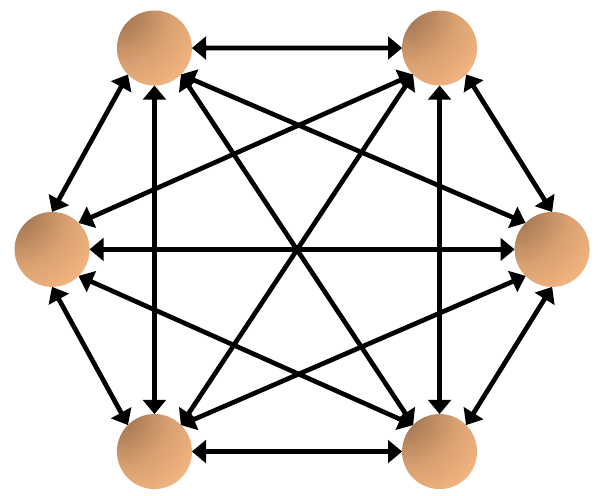}
        \caption{Fully-connected}
    \end{subfigure}%
    \quad\quad
    \begin{subfigure}[t]{0.2\textwidth}
        \centering
        \includegraphics[width=\textwidth]{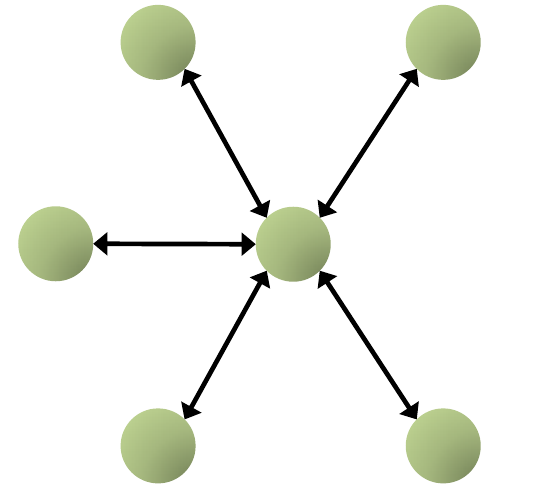}
        \caption{Centralized}
    \end{subfigure}%
    \quad\quad
    \begin{subfigure}[t]{0.2\textwidth} 
         \centering
        \includegraphics[width=\textwidth]{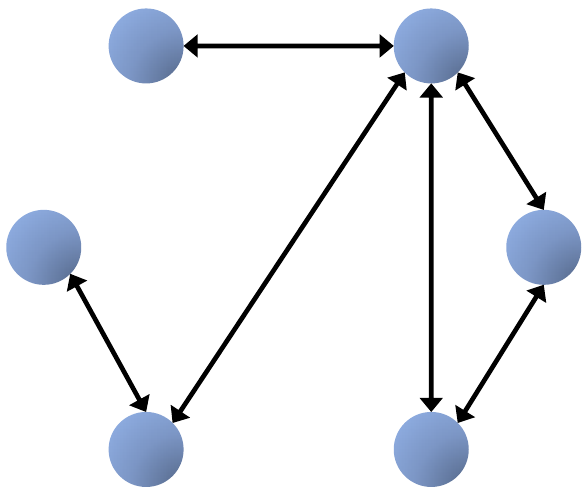}
        \caption{Sparsely-connected}
    \end{subfigure}
    \caption{Different network topologies for distributed learning: (a) a non-cooperative topology where agents do not interact with one another, (b) a fully-connected topology where every agent is connected to every other agent in the network, (c) a centralized topology that is equivalent in performance to the fully-connected topology but requires less communication overhead, and (d) a sparsely connected network topology in which each user is connected to a subset of the users in the network.}
    \label{fig:topologies}
\end{figure*}
There are various approaches for optimizing \eqref{eq:w_o}. We concentrate on \emph{fully-distributed} strategies that operate over sparsely-connected networks. The concept of a fully distributed strategy is used here to mean the following:
\begin{itemize}
		\item[(i)]
		There is no central node that is coordinating the communication and computation during the learning process.
		\item[(ii)]
		A node does not need to be connected to all other nodes. Indeed, as long as the network is connected (and it can be sparsely connected), the algorithm is able to approach the solution to the global learning problem.
		\item[(iii)]
		Only one-hop communication is allowed during the learning process. That is, we do not allow the routing of a data packet over the network. Instead, each agent/node is only allowed to be directly communicating with its intermediate neighbors.
	\end{itemize}
	Figure \ref{fig:topologies} illustrates the types of topologies examined in this manuscript. It is important to notice that the centralized and fully-connected topologies are theoretically equivalent, but practically different as the centralized topology greatly reduces the amount of information that is communicated throughout the network. The centralized topology, however, is not robust to node failure since the entire solution breaks down if the central node fails. Throughout the remainder of the manuscript, we will refer to the centralized and fully-connected approaches interchangeably since they have identical excess-risk performance.

\subsection{Non-Cooperative Strategy}
First, we examine the non-cooperative strategy for optimizing \eqref{eq:w_o}, which is for each node $k$ to run independently a stochastic gradient algorithm of the following form for $i\geq 1$ \cite{Polyak,NOW_ML,tsitsiklis1986distributed,SayedProcIEEE}:
\begin{align}
	\w_{k,i} = \w_{k,i-1}\!-\!\mu(i) \nabla_{\!w} Q(\!\w_{k,i-1},\!\x_{k,i}\!) \quad\quad\textrm{[no cooperation]}
	\label{eq:non_coop}
\end{align}
where $\nabla_{w} Q(\cdot)$ denotes the gradient vector of the loss function, and $\mu(i)\geq 0$ is a step-size sequence. The gradient vector employed in \eqref{eq:non_coop} is an approximation for the actual gradient vector, $\nabla_{w} J_k(\cdot)$, of the risk function. The difference between the true gradient vector and its approximation used in \eqref{eq:non_coop} is called \emph{gradient noise}. Due to the presence of the gradient noise, the estimate $\w_{k,i}$ generated by \eqref{eq:non_coop} becomes a random quantity; we use boldface letters to refer to random variables throughout our manuscript, which is already reflected in our notation in \eqref{eq:non_coop}.

It is shown in \cite{Polyak, Nemirovski} that for strongly-convex risk functions $J_k(w)$, the non-cooperative scheme \eqref{eq:non_coop} achieves 
an asymptotic convergence rate in the order of $O(1/i)$ under some conditions on the gradient noise and the step-size sequence $\mu(i)$, where the notation $f(n) = O(g(n))$ means that the sequence $f(n)$ decays at a rate that is at most the rate of decay of $g(n)$ for sufficiently large $n$---i.e., there exist positive constant $c_1$ and an integer $n_0$ such that $f(n) \leq c_1 g(n)$ for all $n> n_0$. In this way, in order to achieve an excess-risk accuracy on the order of $O(\epsilon)$, the non-cooperative algorithm \eqref{eq:non_coop} would require $O(1/\epsilon)$ samples. It is further shown in \cite{Polyak, Agrawal_Information_Theoretic} that no algorithm can improve upon this rate under the same conditions. This implies that if no cooperation is to take place between the nodes, then the best asymptotic rate each learner would hope to achieve is on the order of $O(1/i)$.

\subsection{Centralized Strategy}
Now, in place of the non-cooperative strategy, let us assume that the $N$ nodes transmit their samples to a \emph{central} processor, which executes the following  algorithm:
\begin{align}
	\w_{i} = \w_{i-1} \!-\! \frac{\mu(i)}{N} \!\sum_{k=1}^N \nabla_w Q(\w_{i-1},\x_{k,i}) \quad\! \textrm{[unweighted\ cent.]}
	\label{eq:full_gradient}
\end{align}
It can be shown that this implementation will have an asymptotic convergence rate in the order of $O(1/ N i)$ for step-size sequences of the form $\mu(i)=\mu/i$ (see Corollary \ref{cor:cent}).  In other words, the centralized implementation \eqref{eq:full_gradient} provides an $N-$fold increase in convergence rate relative to the non-cooperative solution \eqref{eq:non_coop}. One of the questions we wish to answer in this work is whether it is possible to derive a {\em fully} distributed algorithm that allows \emph{every} node in the network to converge (in the mean-square-error sense) at the same rate as the centralized solution, i.e., $O(1/N i)$, with only communication between neighboring nodes and for general ad-hoc networks. We show that this task is indeed possible. We will additionally show that the distributed strategy can outperform the na\"{i}ve centralized implementation  \eqref{eq:full_gradient} when the gradient noise profile across the agents is non-uniform, but that it will match the performance of a weighted version of \eqref{eq:full_gradient}, namely, the following {\em weighted} centralized strategy:

\begin{align}
	\w_{i} = \w_{i-1} \!-\! \mu(i) \!\sum_{k=1}^N \pi_k \nabla_w Q(\w_{i-1},\x_{k,i}) \quad\! \textrm{[weighted\ cent.]}
	\label{eq:weighted_full_gradient}
\end{align}
where the $\{\pi_k\}$ are convex combination weights that satisfy:
\begin{align}
	\pi_k > 0, \quad\quad \sum_{k=1}^N \pi_k = 1
\end{align}
and are meant to discount gradients with higher noise power compared to others. We next describe two popular fully-distributed strategies.  

\subsection{Diffusion Strategy} 
Following the approach of \cite{jianshu_part1,NOW_ML,SayedProcIEEE}, the diffusion strategy for evaluating distributed estimates $\w_{k,i}$ for $w^o$ in \eqref{eq:w_o} takes the following form:
\begin{alignat}{2}
			\boldsymbol{\psi}_{k,i} \!&=\! \w_{k,i-1} \!-\! \mu(i) \nabla_w Q(\w_{k,i-1},\x_{k,i}) &\ \ &\textrm{[adaptation]} \label{eq:A}\\
			\w_{k,i}   \!&=\! \sum_{\ell=1}^N a_{\ell k} \boldsymbol{\psi}_{\ell,i} & &\textrm{[aggregation]} \label{eq:C2}
\end{alignat}
where $\x_{k,i}$ denotes the current data sample available at node $k$. Each node $k$ begins with an estimate $\w_{k,0}$ and employs a diminishing positive step-size sequence $\mu(i)$. The non-negative coefficients $\{a_{\ell k}\}$, which form a left-stochastic $N \times N$ combination matrix $A=[a_{\ell k}]$, are used to scale information arriving at node $k$ from its neighbors. These coefficients are chosen to satisfy:
\begin{align}
	\underset{\ell = 1}{\overset{N}{\sum}} a_{\ell k}\! =\! 1, \ a_{\ell k}= 0 \textrm{ if nodes } (\ell,\! k) \textrm{ are not connected}
\end{align} 
We emphasize that we are only requiring $A$ to be left-stochastic, meaning that only each of its columns should add up to one rather than each of its columns and rows. The neighborhood $\mathcal{N}_k$ for each node $k$ is defined as the set of nodes $\{\ell\}$ for which $a_{\ell k} \neq 0$. The neighborhood $\mathcal{N}_k$ is typically known to agent $k$. The main difference between the above algorithm and the original adapt-then-combine (ATC) diffusion strategy studied in \cite{jianshu_part1,SayedProcIEEE,NOW_ML} is that we are employing a diminishing step-size sequence $\mu(i)$ as opposed to a constant step-size. Constant step-sizes have the distinct advantage that they allow nodes to continue adapting their estimates in response to drifts in the underlying data distribution \cite{zhao2012diffusion}. In this work, we are interested in examining the generalization ability of distributed learners asymptotically when the underlying distribution, ${\cal X}_k$, remains stationary, in which case the use of decaying step-sizes sequences is justified. If the statistical distribution of the data were subject to drifts, then constant step-sizes become a necessity, and this scenario is already studied in some detail in \cite{jianshu_part1,jianshu_part2,SayedProcIEEE,NOW_ML}.

\subsection{Consensus Strategy}
In addition to \eqref{eq:A}--\eqref{eq:C2}, we shall examine the following consensus-based implementation \cite{yan,KarMoura2,nedic2,tsitsiklis1986distributed} for solving the same problem \eqref{eq:w_o}:  
\begin{equation}
	\w_{k,i}   = \sum_{\ell=1}^N \!a_{\ell k} \w_{\ell,i-1} \!-\! \mu(i) \nabla_w Q(\w_{k,i-1},\x_{k,i}) \quad\textrm{[consensus]}
	\label{eq:consensus}
\end{equation}
The diffusion and consensus strategies \eqref{eq:A}-\eqref{eq:consensus} have exactly the same computational complexity, except that the computations are performed in a different order. We will see in Sec. \ref{sec:consensus} that this difference enhances the performance of diffusion over consensus. Moreover, in the constant step-size case, the difference in the order in which the operations are performed causes   an anomaly in the behavior of consensus solutions in that they can become unstable even when all individual nodes are able to solve the inference task in a stable manner; see  \cite{shine_diffusion_consensus,SayedProcIEEE,NOW_ML}. Furthermore, consensus strategies of the form \eqref{eq:consensus} are usually limited to employing a {\em doubly}-stochastic combination matrix $A$. The analysis in the sequel will show that left-stochastic matrices actually lead to improved excess-risk performance (see Eqs. \eqref{eq:quadratic_optimization_p}--\eqref{eq:Metropolis_DS}) while convergence of the distributed implementation continues to be guaranteed (see Theorem \ref{thm:non_asym_convergence_main}).

\section{Main Assumptions}
\label{sec:assumptions}
Before proceeding with the analysis, we list in this section the main assumptions that are needed to facilitate the exposition. The conditions listed below are common in the broad stochastic optimization literature --- see the explanations in \cite{Polyak,tsitsiklis1986distributed,SayedProcIEEE,NOW_ML}. The first condition assumes that the $J_k(w)$ are strongly-convex, with a common minimizer  $w^o$ for  $k=1,\ldots,N$. This condition ensures that the optimization problem \eqref{eq:w_o} is well-conditioned. 
\begin{assumption}[\textbf{Properties of risk functions}]
	\label{ass:HessianAssumption}
	Each risk function $J_k(w)$ is twice continuously{\hyph}differentiable and its Hessian matrix is uniformly bounded from below and from above, namely, 
	\begin{equation}
		0 < \underline{\lambda} I \leq \nabla^2_w J_k(w) \leq \overline{\lambda} {I} < \infty
		\label{eq:HessianAssumption}
	\end{equation}
	Furthermore, the risks at the various agents are minimized at the same location:
	\begin{align}
		w^o = \underset{w}{\arg\min}\  J_k(w),\quad\quad k=1,\ldots,N
	\end{align}
	and the Hessian matrices are locally Lipschitz continuous at $w^o$, i.e., for all $\|w^o-w\| < \varepsilon$, there exists some $\tau_k \geq 0$ such that:
	\begin{align}
		\|\nabla^2_w J_k(w^o) - \nabla^2_w J_k(w)\| \leq \tau_k \cdot \|w^o - w\| \label{eq:Lipschitz_Hessian}
	\end{align}	
	We denote the value of the Hessian matrices at $w=w^o$ (assumed uniform across the agents) by 
	\begin{align}
		H \triangleq \nabla_w^2 J_{k} (w^o), \quad\quad k=1,\ldots,N
	\end{align}
	where $H>0$. We let $\lambda_{\min} > 0$ denote the smallest eigenvalue of $H$.\hfill$\blacksquare$
\end{assumption}
In fact, conditions \eqref{eq:HessianAssumption} and the locally Lipschitz condition \eqref{eq:Lipschitz_Hessian} jointly imply \eqref{eq:Lipschitz_Hessian} globally (see Lemma E.8 in \cite{NOW_ML}); i.e., for all $w \in \mathbb{R}^{M\times 1}$, we have that there exists some $\tau_k$ such that:
\begin{align}
	\|\nabla^2_w J_k(w^o) - \nabla^2_w J_k(w)\| \leq \tau_k \cdot \|w^o - w\| \label{eq:global_Lipschitz_Hessian}
\end{align}
Observe that
\begin{align}
	\lambda_{\min}\geq \underline{\lambda} \label{eq:lambda_min>=bar_lambda}
\end{align}
\noindent One useful implication that follows from Assumption 1 is the following. Consider the expected excess-risk \eqref{eq:ER_k} at node $k$. Using the following sequence of inequalities, we can bound the excess-risk by a square weighted norm:
\begin{subequations}
\begin{align}
	\textrm{ER}_k(i) &\triangleq \mathbb{E}_\w \{J_k(\w_{k,i-1}) - J_k(w^o)\}	\nonumber\\
	&\stackrel{\mathrm{(a)}}{=} \mathbb{E}_\w\|\widetilde{\w}_{k,i-1}\|_{{\bm{S}}_{k,i}}^2 \label{eq:MSE_ind}\\
	&\stackrel{\mathrm{(b)}}{\leq} \frac{\overline{\lambda}}{2} \cdot \mathbb{E}_\w\| \widetilde{\w}_{k,i-1}\|^2
	\label{eq:function_diff_bound}
\end{align}
\end{subequations}
where $\widetilde{\w}_{k,i} \triangleq w^o - \w_{k,i}$, $\E_\w\{\cdot\}$ denotes expectation over the distribution of $\w$, and the weighting matrix $\bm{S}_{k,i}$ is defined as
\begin{align}
	\bm{S}_{k,i} \triangleq \left[\int_0^1 t \int_0^1 \nabla_w^2 J_k(w^o - s \ t \ \widetilde{\w}_{k,i-1})ds dt \right]
	\label{eq:weight_T}
\end{align}
Step (a) is a consequence of applying the following mean-value theorem \cite[p.~24]{Polyak} \cite{NOW_ML} twice for an arbitrary real-valued differentiable function $f(\cdot)$:
\begin{equation}
	f(a+b) = f(a) + \left(\int_0^1 \nabla^\T f(a + t\cdot b)\ dt \right) \cdot b
	\label{eq:polyak1}
\end{equation}
and the fact that $w^o$ optimizes $J_k(w)$ so that $\nabla_w J_k(w^o) = 0$. Step (b) is due to \eqref{eq:HessianAssumption} and \eqref{eq:weight_T}.

Expression \eqref{eq:MSE_ind} shows that the expected excess-risk at node $k$ is {\em equal} to a weighted mean-square-error with weight matrix \eqref{eq:weight_T}. This means that one way to compute or bound the expected excess-risk is by evaluating weighted mean-square-error quantities of the form \eqref{eq:MSE_ind} or \eqref{eq:function_diff_bound}. This is the route that we will follow in this manuscript. We will analyze the right-hand side of \eqref{eq:MSE_ind} in order to draw conclusions regarding the evolution of the expected excess-risk. In particular, once we establish that the distributed algorithm converges in the mean-square-error sense, then inequality  \eqref{eq:function_diff_bound} will immediately allow us to conclude that the algorithm also converges in excess-risk. Similarly, we can obtain the asymptotic expression for the excess-risk by leveraging the  weighted-mean-square-error analysis developed for constant step-size distributed strategies \cite{jianshu_part1,jianshu_part2,NOW_ML}, adjusted for the decaying step-size case. Observe that these conclusions are different than the useful results in \cite{Bianchi}, which focused on studying convergence in distribution. The mean-square-error results will enable us to expose analytically various interesting differences in the performance of distributed strategies, such as diffusion and consensus.

Our second condition is on the gradient noise process, which is defined, for a generic vector $\w$, as  
\begin{align}
\v_{k,i}(\w) \triangleq \nabla_w Q(\w,\x_{k,i}) - \nabla_w J_k(\w)
		\label{eq:grad_model}	
\end{align}
We collect the noises from across the network into a column vector 
\begin{align}
	\g_{i}(\sw) &\triangleq \textrm{col}\{\v_{1,i}(\w_{1}),\ldots,\v_{N,i}(\w_{N})\}
\end{align}
where we are introducing the vector notation
\begin{align}
	\sw &\triangleq \mathrm{col}\{\w_1,\ldots,\w_N\}
\end{align}
for the collection of parameters across the agents. We denote the covariance matrix of the gradient noise vector by 
\begin{align}
		\mathcal{R}_{g,i}(\sw) &\triangleq \E\{\g_{i}(\sw) \g_{i}^\T(\sw) | \boldsymbol{\mathcal{F}}_{i-1}\} \label{eq:Mathcal_R_v1}	
\end{align}
where the conditioning is in relation to the past history of the estimators, 
$\boldsymbol{\mathcal{F}}_{i-1} \!\triangleq\! \{\w_{k,j}:k = 1 , \ldots, N\ \mathrm{and}\ j \leq i-1\}$.  The following conditions are relaxations of assumptions that are regularly considered in the stochastic approximation literature; they are generally satisfied in important scenarios, such as logistic regression or quadratic loss functions of the form \eqref{eq:RLR}--\eqref{eq:LMS} --- see \cite{NOW_ML}. 

\begin{assumption}[\textbf{Gradient noise model}]
	\label{ass:noiseModeling}
	We assume the gradient noise process satisfies: 
\begin{align}
		&\mathbb{E}\{\v_{k,i}(\w) | \boldsymbol{\mathcal{F}}_{i-1}\} = 0 \\
		&\mathbb{E}\{\|\v_{k,i}(\w)\|^4 | \boldsymbol{\mathcal{F}}_{i-1}\} \leq \alpha_4 \|w^o - \w\|^4 + \sigma_{v4}^4  \label{eq:noise_fourth_bound}\\
		&\E\{\v_{k,i}^\T(w^o)\v_{\ell,i}(w^o)\} = 0, \quad k \neq \ell \label{eq:uncorrelated_noise_in_space}
\end{align}
for some $\alpha_4\geq 0$, $\sigma_{v4}^4 \geq 0$, as well as: 
\begin{align}
		&\|\mathcal{R}_{g,i}(\psw) - \mathcal{R}_{g,i}(\mathds{1}_N\!\otimes\! w^o)\| \leq \kappa \|\mathds{1}_N\!\otimes\!w^o-\psw\|^\gamma \label{eq:noise_covar_lipschitz_local}\\
		&\lim_{i\rightarrow\infty} \E \{\mathcal{R}_{g,i}(\mathds{1}_N \otimes w^o)\} \triangleq \mathcal{R}_{g}^o \label{eq:noise_covariance}
	\end{align}	
	for some $\kappa\geq 0$, $0 < \gamma \leq 4$, and where
	\begin{align}
		\psw &\triangleq \mathrm{col}\{w_1,\ldots,w_N\}
\end{align}		
and \eqref{eq:noise_covar_lipschitz_local} is assumed to hold  for $\|\mathds{1}_N\!\otimes\!w^o-\psw\| \leq \epsilon$, for some small $\epsilon>0$. \hfill$\blacksquare$
\end{assumption}

Observe that Assumption \ref{ass:noiseModeling} implies that:
\begin{align}
\mathbb{E}\{\|\v_{k,i}(\w)\|^2 | \boldsymbol{\mathcal{F}}_{i-1}\} &\leq \alpha_2 \cdot\|w^o\!-\!\w\|^2 \!+\! \sigma_{v2}^2 \label{eq:noise_var_bound1}
\end{align}
for some $\alpha_2,\sigma_{v2}^2 \geq 0$ and $\w \in \boldsymbol{\mathcal{F}}_{i-1}$. In addition, the local Lipschitz condition \eqref{eq:noise_covar_lipschitz_local} of order $\gamma$ \cite[p.~53]{jeffreys1999methods} (sometimes referred to as the H\"older condition of order $\gamma$ \cite[p.~110]{howie2012real}) implies, under Assumption \ref{ass:HessianAssumption}, that the following global condition also holds \cite{NOW_ML, jianshu_part2}:
\begin{align}
	\|\mathcal{R}_{g,i}(\psw) - \mathcal{R}_{g,i}(\mathds{1}_N\!\otimes\! w^o)\| &\leq \kappa_1 \|\mathds{1}_N\!\otimes\! w^o-\psw\|^2 + \nonumber\\
	&\quad\ \kappa \|\mathds{1}_N\!\otimes\!w^o-\psw\|^\gamma \label{eq:noise_covar_lipschitz}
\end{align}
for some constant $\kappa_1\geq 0$ that depends on $\epsilon$ and where $\kappa$ is from \eqref{eq:noise_covar_lipschitz_local}. Furthermore, due to \eqref{eq:uncorrelated_noise_in_space}, we have that the matrix $\mathcal{R}_g^o$ is block-diagonal:
\begin{align}
	\mathcal{R}_g^o &\triangleq \textrm{blockdiag}\{R_{v,1}^o,\ldots,R_{v,N}^o\}	\label{eq:R_v} \\
	R_{v,k}^o &\triangleq \lim_{i\rightarrow\infty} \E\{\v_{k,i}(w^o) \v_{k,i}^\T(w^o)\}
\end{align}
Since nodes sample the data in an independent fashion, it is reasonable to expect the gradient noise to be uncorrelated across all nodes, as required by  \eqref{eq:uncorrelated_noise_in_space}. 

Our third condition relates to the structure of the network topology. We will assume that the network is strongly-connected, which means that (a) there exists  at least one nontrivial self-loop, $a_{kk}>0$ for some $k$, and (b) for any two agents $k$ and $\ell$, there exists a path  with nonzero weights $\{a_{k,m_1}, a_{m_1,m_2},\ldots,a_{m_r,\ell}\}$ from $k$ to $\ell$, either directly if they are neighbors or through other agents. It is well-known that the combination matrix $A$ for such networks is primitive \cite[p.~516]{HornJohnsonVol1}. That is, all entries of $A$ are non-negative and there exists some positive integer $n_o>0$ such that all entries of $A^{n_o}$ are strictly positive. One important property of primitive matrices follows from the Perron-Frobenius theorem \cite[p.~534]{HornJohnsonVol1}; $A$ will have a single eigenvalue at one, while all other eigenvalues of $A$ will lie strictly inside the unit circle. Moreover, if we let $p$ denote the right-eigenvector associated with the eigenvalue at one, and normalize its entries to add up to one, i.e.,
\begin{equation} 
	A p = p,\;\;\;\;\mathds{1}_N^\T p=1
\end{equation}
then all entries of $p$ will be strictly positive. We shall refer to $p$ as the Perron eigenvector of $A$. We formalize this assumption in the following:
\begin{assumption}[\textbf{Network Topology}]
\label{ass:topology}
The network is strongly-connected so that the combination matrix $A$ is primitive with $\mathds{1}^\T p = 1$ and $p \succ 0$, where $p$ denotes the the Perron eigenvector of $A$. \qed
\end{assumption}

\section{Main Results}
\label{sec:main_results}
In this section, we list the main results and defer the detailed proofs to the appendices. 
\subsection{Convergence Properties}
\label{ssec:asymBound}
Our first result provides conditions on the step-size sequence under which the diffusion strategy  converges both in the mean-square-error (MSE)  sense and also almost surely. The difference between the two sets of conditions that appear below is that in one case the step-size sequence is additionally required to be square-summable. 
\begin{theorem}[\textbf{Convergence rates}]
\label{thm:non_asym_convergence_main}
	Let Assumptions \ref{ass:HessianAssumption}-\ref{ass:topology} hold and let the step-size sequence satisfy
\begin{align}
	\mu(i) > 0, \quad \sum_{i=1}^\infty \mu(i) = \infty, \quad \lim_{i\rightarrow\infty} \mu(i) = 0.
	\label{eq:SS_conditions}
\end{align}
Then, $\w_{k,i}$ generated by \eqref{eq:A}--\eqref{eq:C2} converges in the MSE sense to $w^o$, i.e., 
	\begin{align}
		\lim_{i\rightarrow\infty} \E\|w^o-\w_{k,i}\|^{2} = 0.
		\label{eq:asym_bound_main}
	\end{align}	
	If the step-size sequence satisfies the additional square-summability condition:
	\begin{align}
		\sum_{i=1}^\infty \mu^2(i) < \infty,
		\label{eq:SS_SS_constraint}
	\end{align}
	then $\w_{k,i}$ converges to $w^o$ almost surely (i.e., with probability one) for all $k=1,\ldots,N$. Furthermore, when the step-size sequence is of the form  $\mu(i) = \mu/i$  with $\mu$ satisfying $2\underline{\lambda} \mu > 1$, then the second and fourth-order moments of the error vector converge at the rates of $O(1/i)$ and $O(1/i^2)$, respectively:
	\begin{align}
		\limsup_{i\rightarrow\infty} \frac{\E\|\widetilde{\w}_{k,i}\|^2}{i^{-1}} &\leq \frac{\sigma_{v2}^2 \mu^2}{2\underline{\lambda}\mu - 1} \label{eq:MSD_ind_quadratic}\\
		\limsup_{i\rightarrow\infty} \frac{\E\|\widetilde{\w}_{k,i}\|^4}{i^{-2}} &\leq \textrm{constant} \label{eq:fourth_ind}
	\end{align}
	where $\sigma_{v2}^2$ was introduced in \eqref{eq:noise_var_bound1}.
\end{theorem}
\begin{proof}
See Appendix \ref{app:local_mse_rec}.
\end{proof}\vspace{1\baselineskip}
Observe that \eqref{eq:MSD_ind_quadratic} implies that each node converges in the mean-square-error sense at the rate $1/i$. Combining this result with \eqref{eq:function_diff_bound}, we conclude that each node also converges in excess-risk at this rate:
\begin{align}
	\limsup_{i\rightarrow\infty} \frac{\mathrm{ER}_k(i)}{i^{-1}} &\leq \frac{\mu}{2} \cdot \frac{\bar{\lambda}\mu }{2\underline{\lambda}\mu - 1} \cdot \sigma_{v2}^2
\end{align}
when $2\underline{\lambda}\mu > 1$.

Note that this conclusion does not yet reveal the benefit of cooperation (for example, it does not show how the convergence rate depends on $N$). In the next section, we will derive closed-form asymptotic expressions for the mean-square-error and excess-risk, and from these expressions we will be able to highlight the benefit of network cooperation.

\subsection{Evolution of Excess-Risk Measure}
\label{ssec:asymApprox}
We continue to assume that the step-size sequence is selected as $\mu(i)=\mu/i$ for some $\mu>0$. This sequence satisfies conditions \eqref{eq:SS_conditions} and \eqref{eq:SS_SS_constraint}. Observe that in order to evaluate the excess-risk at node $k$, we must evaluate \eqref{eq:MSE_ind}. To do so, we first form the following network-wide error quantity:
\begin{align}
	\widetilde{\sw}_i &\triangleq \mathrm{col}\{\widetilde{\w}_{1,i},\ldots,\widetilde{\w}_{N,i}\}
\end{align}
and let $E_{kk}$ denote the $N\times N$ matrix with a single entry equal to one at the $(k,k)-$th location and all other entries equal to zero. Then, using \eqref{eq:MSE_ind}, we can write:
\begin{align}
	\textrm{ER}_k(i) \triangleq \E \|\widetilde{\sw}_{i-1}\|^2_{E_{kk}\otimes \bm{S}_{k,i}} \label{eq:def_ER_k}
\end{align}

In order to facilitate the analysis, we introduce the eigenvalue decomposition of matrix $H$:
\begin{align}
	H  = \Phi \Lambda \Phi^\T \label{eq:eigendecomp_H}
\end{align}
where $\Phi$ is an orthogonal matrix and $\Lambda$ is diagonal with positive entries $\lambda_1, \lambda_2, \ldots, \lambda_M$. Moreover, since the matrix $A$ is left-stochastic and primitive (by Assumption \ref{ass:topology}), we can express its Jordan decomposition in the form:
\begin{align}
	A = p \mathds{1}^\T + Y D_{N-1} R^\T 
	\label{eq:A_factorization_split}
\end{align}
where $R,Y \in \mathbb{R}^{N\times N-1}$ represent the remaining left and right eigenvectors while $D_{N-1}$ represents the Jordan structure associated with the eigenvalues inside the unit disc. 
\begin{theorem}[\textbf{Asymptotic Convergence of $\mathrm{ER}_k(i)$}]
	\label{thm:asymptotic_term}
	Let Assumptions \ref{ass:HessianAssumption}-\ref{ass:topology} hold and let $\lambda_{\min}$ denote the smallest eigenvalue of the matrix $H$:
	\begin{align}
		\lambda_{\min} \triangleq \min \{\lambda_1,\ldots,\lambda_M\}
	\end{align}		
	Then, when $2\lambda_{\min} \mu > 1$, it holds asymptotically that
	\begin{align}
	\mathrm{ER}_k(i) &\sim  \frac{\mu}{2 i} \cdot \sum_{m=1}^M  \frac{\lambda_m \mu }{2\lambda_m \mu - 1}\cdot  p^\T  R_{v,m} p \label{eq:asymptotic_theorem_result} \\
					&= \frac{\mu}{2 i} \cdot \sum_{m=1}^M  \frac{\lambda_m \mu }{2\lambda_m \mu - 1} \sum_{\ell=1}^N p_\ell^2 (\Phi^{\T}R_{v,\ell}^o\Phi)_{mm} \label{eq:asymptotic_theorem_result_simple} 
	\end{align}
where the notation $f(i) \!\sim\! g(i)$ implies that $\displaystyle{\lim_{i\rightarrow\infty}} f(i)/g(i) \!=\! 1$. Moreover, $\lambda_m$ is the $m$-th eigenvalue of $H$ and the $N\times N$ matrix $R_{v,m}$ is defined as:
\begin{align}
	R_{v,m} \triangleq \left[\begin{array}{ccc}
	(\Phi^\T R_{v,1}^o \Phi)_{mm} &  &  \\ 
	 & \ddots &  \\ 
	 &  & (\Phi^\T R_{v,N}^o \Phi)_{mm}
	\end{array} \right] \label{eq:R_vm}
\end{align}
where the notation $(X)_{mm}$ denotes the $m$-th diagonal element of the matrix $X$.
\end{theorem}
\begin{proof}
	See Appendix \ref{sec:Proof_Asymptotic}.
\end{proof}
Theorem \ref{thm:asymptotic_term} establishes a closed-form expression for the asymptotic excess-risk of the diffusion algorithm. We observe that the slowest rate at which the asymptotic term converges depends on the smallest eigenvalue of $H$, which is $\lambda_{\min}$, and the constant $\mu$. Interestingly, the only dependence on the topology of the network asymptotically is encoded in the Perron vector $p$ of the combination matrix $A$---i.e., most of the eigen-structure of the topology matrices becomes irrelevant asymptotically and only influences the convergence rate in the transient regime. We will see further ahead that the Perron vector $p$ can be optimized to minimize the excess-risk in the asymptotic regime. It is natural that the transient stage should depend on the network geometry because the networked agents are propagating their information over the entire network. The speed of information propagation over a sparsely connected network is determined by the second largest eigenvalue of the combination matrix \cite{jianshu_part1}, which is influenced by the degree of network connectivity. Our results show, however, that there is an asymptotic regime where the performance of the diffusion strategy can be made invariant to the specific network topology since the Perron vector $p$ can be designed for general connected networks, as we will see further ahead in \eqref{eq:Metropolis_DS}. Finally, we observe that all agents participating in the network will achieve the same asymptotic performance given by the right-hand-side of \eqref{eq:asymptotic_theorem_result} as this asymptotic expression for the excess-risk does not depend on any particular node index $k$ but only on the network-wide quantity $p^\T R_{v,m} p$.

When $\lambda_{\min}$ is not known, and thus it is not clear how to choose $\mu$ to satisfy $2\lambda_{\min}\mu > 1$, it is common to choose a large $\mu$ that forces $2\lambda_{\min}\mu \gg 1$. In this case, we obtain from \eqref{eq:asymptotic_theorem_result_simple} that
\begin{align}
	\mathrm{ER}_k(i) &\approx \frac{\mu}{4 i} \cdot \sum_{\ell=1}^N p_\ell^2 \Tr(R_{v,\ell}^o), \quad\quad\quad[2\lambda_{\min}\mu \gg 1]
\label{eq:MLSP_approx}
\end{align}
This approximation is close in form to the original steady-state performance expression derived for the diffusion algorithm when a constant step-size $\mu$ is used \cite{xiaochuan}. The main difference is that the ``steady-state'' term will now   diminish at the rate $1/i$ when $\mu(i) = \mu/i$ and $2\lambda_{\min}\mu \gg 1$. 

By specializing the previous results to the case $N=1$ (a stand-alone node), we obtain as a corollary the following result for the expected excess-risk that is delivered by the non-cooperative stochastic gradient algorithm \eqref{eq:non_coop}.
\begin{corollary}[\textbf{Stochastic gradient approximation}]
\label{cor:ind}
Consider recursion \eqref{eq:non_coop} and let Assumptions \ref{ass:HessianAssumption}-\ref{ass:noiseModeling} hold with $\mu(i) = \mu/i$ and $2\lambda_{\min} \mu > 1$. Then, the excess-risk approaches asymptotically:
\begin{align}
	\E_\w\{\!J(\w_{i-1})\!\!-\!\!J(w^o)\!\} &\sim \frac{\mu}{2 i} \!\sum_{m=1}^M  \!\!\frac{\lambda_m\mu}{2\lambda_m \mu \!-\! 1} (\Phi^{\T}\!R_{v}^o\Phi)_{mm} \label{eq:asymptotic_individual_result}\\
	&\approx \frac{\mu}{4 i} \cdot \Tr(R_v^o),\quad[2\lambda_{\min}\mu \gg 1]  \label{eq:asymptotic_individual_result_approx}
\end{align}
where
\begin{align}
	R_v^o \triangleq \lim_{i\rightarrow \infty} \E\{\v_{i}(w^o) \v_{i}^\T(w^o)\}
\end{align}
and $\v_i(\cdot)$ is the noise process modeled by Assumption \ref{ass:noiseModeling}.\hfill\qed
\end{corollary}

Observe that \eqref{eq:asymptotic_individual_result}--\eqref{eq:asymptotic_individual_result_approx} are stronger than those in Theorem~\ref{thm:non_asym_convergence_main} since we are not only stating that the convergence rate is $O(1/i)$ but we are also giving the exact constant that multiplies the factor $1/i$. In the next section, we examine the relationship between the derived constant and the network size and noise parameters across the network. Following this presentation, we will utilize our mean-square-error expressions to examine the differences between the diffusion strategy \eqref{eq:A}--\eqref{eq:C2} and the consensus strategy \eqref{eq:consensus}.

\subsection{Benefit of Cooperation}
\label{sec:gain}
\begin{table*}
\caption{Examples of fully-distributed combination rules and their corresponding Perron vector.}
\label{tbl:combination_rules}
\centering
\begin{tabular}{c||c|c}
\hline \hline
\rowcolor[gray]{0.9} Combination Rule & $a_{\ell k}$ & $p_k$ \\\hline
\rule[-1ex]{0pt}{4ex} Average Rule & $\begin{cases}
	\frac{1}{\displaystyle |\mathcal{N}_k|}, & \ell \in \mathcal{N}_k\\
	0, &\textrm{otherwise}
\end{cases}$ & $\frac{\displaystyle |\mathcal{N}_k|}{\displaystyle \sum_{m=1}^N |\mathcal{N}_m|}$ \\\hline
Metropolis Rule \cite{metropolis1953equation} & $\begin{cases}
			\displaystyle \frac{1}{\max\left(|\mathcal{N}_k|, |\mathcal{N}_\ell|\right)}, & \ell \in \mathcal{N}_k, \ell \neq k\\
			1 - \displaystyle \sum_{j\in {\cal N}_k\backslash\{k\}} a_{j k},& \ell = k\\
			0, & \textrm{otherwise}
\end{cases}$ & $\displaystyle \frac{1}{N}$\\\hline
Hastings Rule \cite{hastings1970monte} & $\begin{cases}
			\displaystyle \frac{1}{p_k\cdot \max\left(\frac{|\mathcal{N}_k|}{p_k}, \frac{|\mathcal{N}_\ell|}{p_\ell}\right)}, & \ell \in \mathcal{N}_k, \ell \neq k\\
			1 - \displaystyle \sum_{j\in {\cal N}_k\backslash\{k\}} a_{j k},& \ell = k\\
			0, & \textrm{otherwise}
	\end{cases}$ & $p_k$\\\hline\hline
\end{tabular}
\end{table*}
Up to this point in the discussion, the benefit of cooperation has not yet manifested itself explicitly; this benefit is actually already included in the vector $p$. Optimization over $p$ will help bring forth these advantages. Thus, observe that the expression for the asymptotic term in \eqref{eq:asymptotic_theorem_result} is quadratic in $p$. We can optimize the asymptotic expression over $p$ in order to reduce the excess-risk. We re-write the asymptotic excess-risk \eqref{eq:asymptotic_theorem_result} as:
\begin{align}
	\E\{J_{k}(\w_{k,i-1}) - J_k(w^o)\} \sim \frac{\mu}{2 i}  p^\T Z p \label{eq:benefit_cooperation_asym}
\end{align}
where
\begin{align}
	Z \triangleq \sum_{m=1}^M  \frac{\lambda_m\mu }{2\lambda_m \mu-1}  R_{v,m}
	\label{eq:Z}
\end{align}
Then, we consider the problem of optimizing \eqref{eq:benefit_cooperation_asym} over $p$:
\begin{align}
	\min_{p,A\in\mathbb{A}} \ p^\T Z p \quad\textrm{subject\ to} \ Ap = p,\ \ \mathds{1}^\T p = 1,\ \ p\succ0 \nonumber
\end{align}
where $\mathbb{A}$ denotes the set of left-stochastic and primitive combination matrices that satisfy the network topology structure. It is generally not clear how to solve this optimization problem over both $A$ and $p$. We pursue an indirect route. We first remove the optimization over $A$ and determine an optimal $p$. Subsequently, given the optimal $p^o$, we show that a left-stochastic and primitive matrix $A$ in $\mathbb{A}$ can be constructed. The relaxed problem is:
\begin{align}
	\min_{p} \ p^\T Z p \quad\quad \textrm{subject\ to\ } \ \mathds{1}^\T p = 1,\ \ p \succ 0 \label{eq:quadratic_optimization_p}
\end{align}
whose solution is 
\begin{align}
	p^o = \frac{Z^{-1} \mathds{1}}{\mathds{1}^\T Z^{-1} \mathds{1}} \label{eq:p^o}
\end{align}
It is straightforward to verify that $\mathds{1}^\T p^o = 1$. A combination matrix $A$ that has this $p^o$ as its Perron eigenvector is the Hastings rule \cite{boyd2004fastest,xiaochuan} \cite[Lemma 12.2]{NOW_ML}, which is given by
\begin{align}
	a_{\ell k} = \begin{cases}
			\displaystyle \frac{1}{p^o_k\cdot \max\left(\frac{|\mathcal{N}_k|}{p^o_k}, \frac{|\mathcal{N}_\ell|}{p^o_\ell}\right)}, & \ell \in \mathcal{N}_k, \ell \neq k\\
			1 - \displaystyle \sum_{j\in {\cal N}_k\backslash\{k\}} a_{j k},& \ell = k\\
			0, & \textrm{otherwise}
	\end{cases}
	\label{eq:Metropolis_DS}
\end{align}
It is possible to see that for agent $k$ to implement the Hastings rule, it needs to know its neighborhood $\mathcal{N}_k$ (which is known to agent $k$), as well as the number of neighbors that each of its neighbors has (this information is easily derived from the immediate neighbors), and the Perron vector $p^o$ that the network wishes to obtain. Therefore, the design of the weighting matrix $A$ can be done in a fully distributed manner. Table \ref{tbl:combination_rules} lists three fully-distributed combination rules (combination rules that can be implemented in a fully-distributed manner) and their corresponding Perron vector.

To see the effectiveness of this choice for $p$, we consider the case where $2\lambda_{\min} \mu \gg 1$, so that 
\begin{align}
	Z \approx \frac{1}{2} \textrm{diag}\{\Tr(R_{v,1}^o),\ldots,\Tr(R_{v,N}^o)\} \label{eq:L_approx}
\end{align} 
Substituting \eqref{eq:p^o} into \eqref{eq:benefit_cooperation_asym}, we obtain:
\begin{align}
	\E\{J_{k}&(\w_{k,i-1}) - J_k(w^o)\} \sim \frac{\mu}{2 i}  \cdot \frac{1}{\mathds{1}^\T Z^{-1} \mathds{1}} \label{eq:optimized_ER}\\
	&\stackrel{(a)}{\approx} \frac{\mu}{4 i} \cdot \frac{1}{\sum_{k=1}^N \left(\Tr(R_{v,k}^o)\right)^{-1}},\quad[2\lambda_{\min}\mu\gg 1] \label{eq:N_fold}
\end{align}
where step $(a)$ is due to \eqref{eq:L_approx}. First, we will compare this performance with that of the centralized algorithm \eqref{eq:full_gradient}. To do this, we first establish the following result:
\begin{corollary}[\textbf{Un-weighted Centralized processing}]
\label{cor:cent}
Let Assumptions \ref{ass:HessianAssumption}-\ref{ass:noiseModeling} hold with $\mu(i) = \mu/i$ with $2\lambda_{\min}\mu > 1$. Consider the centralized algorithm \eqref{eq:full_gradient}, which has access to all the data from across the network at each iteration. Then, the excess-risk asymptotically satisfies:
 \begin{align}
	\E_\w\{&J(\w_{i-1})-J(w^o)\} \nonumber\\
	&\sim \frac{\mu}{2 i N^2} \left(\sum_{m=1}^M  \frac{\lambda_m\mu}{2\lambda_m\mu-1} \!\sum_{k=1}^N (\Phi^{\T}R_{v,k}^o\Phi)_{mm}\right) \label{eq:asymptotic_centralized_result}\\
	&\approx \frac{\mu}{4 i N^2} \sum_{k=1}^N \Tr(R_{v,k}^o), \quad\quad\quad\quad [2\lambda_{\min}\mu \gg 1]  \label{eq:asymptotic_centralized_result_approx}
\end{align}
where $J(w) \triangleq \frac{1}{N} \sum_{k=1}^N J_k(w)$.
\end{corollary}
\begin{proof}
The centralized algorithm \eqref{eq:full_gradient} is a special case of the diffusion algorithm when $A = p \mathds{1}_N^\T$, where $p = \frac{1}{N} \mathds{1}_N$, which yields a network that satisfies Assumption \ref{ass:topology}. To see this, consider the diffusion algorithm \eqref{eq:A}--\eqref{eq:C2} with $A= p \mathds{1}_N^\T$:
\begin{align}
\boldsymbol{\psi}_{k,i} &= \w_{k,i-1} - \mu(i) \nabla_w Q(\w_{k,i-1},\x_{k,i}) \label{eq:A_cent}\\
			\w_{k,i}   &= \sum_{\ell=1}^N p_\ell \boldsymbol{\psi}_{\ell,i} \label{eq:C_cent}
\end{align}
First, observe that after the first iteration, the estimates across the network are now uniform since \eqref{eq:C_cent} does not depend on $k$. We can therefore drop the subscript $k$ from $\w_{k,i}$:
\begin{align}
\boldsymbol{\psi}_{k,i} &= \w_{i-1} - \mu(i) \nabla_w Q(\w_{i-1},\x_{k,i}) \label{eq:A_cent2}\\
 			\w_{i}   &= \sum_{\ell=1}^N p_\ell \boldsymbol{\psi}_{\ell,i} \label{eq:C_cent2}
\end{align}
Substituting \eqref{eq:A_cent2} into \eqref{eq:C_cent2}, we obtain:
\begin{align}
 			\w_{i}  &= \sum_{\ell=1}^N p_\ell \left(\w_{i-1} - \mu(i) \nabla_w Q(\w_{i-1},\x_{\ell,i})\right)\nonumber\\
 					&= \sum_{\ell=1}^N p_\ell \w_{i-1} - \mu(i) \sum_{\ell=1}^N p_\ell \nabla_w Q(\w_{i-1},\x_{\ell,i})\nonumber\\
 					&\stackrel{(a)}{=} \w_{i-1} - \mu(i) \sum_{\ell=1}^N p_\ell \nabla_w Q(\w_{i-1},\x_{\ell,i}) \nonumber\\
 					&\stackrel{(b)}{=} \w_{i-1} - \frac{\mu(i)}{N} \sum_{\ell=1}^N \nabla_w Q(\w_{i-1},\x_{\ell,i}) \label{eq:centralized_unweighted}
\end{align}
where step $(a)$ is due to $\mathds{1}^\T p = 1$ and step $(b)$ is obtained by substituting $p = \frac{1}{N} \mathds{1}_N$. Observe that \eqref{eq:centralized_unweighted} is identical to \eqref{eq:full_gradient}, the un-weighted centralized algorithm. Then, using the analysis of the diffusion algorithm in Theorem \ref{thm:asymptotic_term}, we have that
\begin{align}
	\mathrm{ER}_k(i) &\sim \frac{\mu}{2 i} \left(\sum_{m=1}^M  \frac{\lambda_m \mu }{2\lambda_m \mu - 1}\cdot  p^\T  R_{v,m} p \right)\nonumber\\
							   &= \frac{\mu}{2 i N^2} \left( \sum_{m=1}^M  \frac{\lambda_m \mu }{2\lambda_m \mu - 1}  \sum_{\ell=1}^N (\Phi^{\T}R_{v,\ell}^o\Phi)_{mm}\right) \label{eq:ER_k_unweighted_cent}
\end{align}
Since the right-hand-side of \eqref{eq:ER_k_unweighted_cent} does not depend on the agent index $k$ (all agents will achieve the same asymptotic performance), we have that the average excess-risk remains the same:
\begin{align}
	\E_\w\{&J(\w_{i-1})-J(w^o)\} = \frac{1}{N} \sum_{k=1}^N \mathrm{ER}_k(i) \nonumber\\
								&\sim \frac{\mu}{2 i N^2} \left( \sum_{m=1}^M  \frac{\lambda_m \mu }{2\lambda_m \mu - 1}  \sum_{\ell=1}^N (\Phi^{\T}R_{v,\ell}^o\Phi)_{mm}\right)\nonumber\\
								&\stackrel{(a)}{\approx} \frac{\mu}{4 i N^2}  \sum_{\ell=1}^N \Tr(R_{v,\ell}^o),\quad\quad\quad\quad[2\lambda_{\min}\mu \gg 1]
\end{align}
where step $(a)$ is due to \eqref{eq:L_approx}, which is the desired result.
\end{proof} 

Comparing \eqref{eq:N_fold} to \eqref{eq:asymptotic_individual_result_approx} in the special case when $\Tr(R_{v,k}^o) = \Tr(R_{v}^o)$ for all $1\leq k\leq N$, we find that the diffusion algorithm offers an $N$-fold improvement in the excess-risk over the non-cooperative solution. Also, comparing \eqref{eq:N_fold} to \eqref{eq:asymptotic_centralized_result_approx} in this case, we observe that asymptotically the diffusion algorithm achieves the same performance as the centralized algorithm \eqref{eq:full_gradient}. More generally, let us consider the case in which the noise profile is \textit{not} uniform across the agents. We call upon the following inequality:
\begin{align}
	\frac{1}{\sum_{k=1}^N \left(\Tr(R_{v,k}^o)\right)^{-1}} \leq \frac{1}{N^2} \sum_{k=1}^N \Tr(R_{v,k}^o) \label{eq:Harmonic_mean}
\end{align}
which follows from the fact that the harmonic mean of a set of numbers is upper-bounded by their arithmetic mean. Then, we conclude from \eqref{eq:Harmonic_mean} that the excess-risk of the diffusion strategy is upper-bounded by that of the centralized strategy \eqref{eq:full_gradient}, and  equality holds only  when the network experiences a spatially uniform gradient noise profile. This implies that the  diffusion strategy actually outperforms the implementation studied in  \cite{Bianchi}, which uses a doubly-stochastic combination matrix. Furthermore, in this case of non-uniform noise profile and since 
\begin{align}
	\frac{1}{\displaystyle \sum_{k=1}^N \left(\Tr(R_{v,k}^o)\right)^{-1}} \leq \min_{1\leq k\leq N} \Tr(R_{v,k}^o)
\end{align}
we also find that the diffusion strategy continues to outperform the non-cooperative solution, and guarantees that the performance of all nodes that utilize the diffusion strategy will outperform the best performing node in the non-cooperative solution.

On the other hand, the weighted centralized algorithm \eqref{eq:weighted_full_gradient} achieves the same performance as the diffusion strategy since it is a special case of diffusion when $A = \pi\mathds{1}^\T$.
\begin{corollary}[\textbf{Weighted Centralized processing}]
\label{cor:weighted_cent} 
Let Assumptions \ref{ass:HessianAssumption}-\ref{ass:noiseModeling} hold with $\mu(i) = \mu/i$ with $2\lambda_{\min}\mu > 1$. Consider the centralized algorithm \eqref{eq:weighted_full_gradient}, which has access to all the data from across the network at each iteration. Then, the excess-risk asymptotically satisfies:
 \begin{align}
	\E_\w\{&J(\w_{i-1})-J(w^o)\} \nonumber\\
	&\sim \frac{\mu}{2 i} \left(\sum_{m=1}^M  \frac{\lambda_m\mu}{2\lambda_m\mu-1} \!\sum_{k=1}^N \pi_k^2 (\Phi^{\T}R_{v,k}^o\Phi)_{mm}\right) \label{eq:asymptotic_weighted_centralized_result}\\
	&\approx \frac{\mu}{4 i} \cdot \sum_{k=1}^N \pi_k^2 \Tr(R_{v,k}^o), \quad\quad\quad [2\lambda_{\min}\mu \gg 1]  \label{eq:asymptotic_weighted_centralized_result_approx}
\end{align}
where $J(w) \triangleq \frac{1}{N} \sum_{k=1}^N J_k(w)$. 
\end{corollary} 
\begin{proof}
The centralized algorithm \eqref{eq:weighted_full_gradient} is a special case of the diffusion algorithm when $A = \pi \mathds{1}_N^\T$ where $\pi = \mathrm{col}\{\pi_1,\ldots,\pi_N\}$, which yields a network that satisfies Assumption \ref{ass:topology}. The derivations closely mirror that of Corollary \ref{cor:cent} except with $p = \pi$, where $\mathds{1}^\T \pi = 1$, instead of $p = \frac{1}{N} \mathds{1}$. We can then use the result of Theorem \ref{thm:asymptotic_term} to obtain \eqref{eq:asymptotic_weighted_centralized_result} and \eqref{eq:asymptotic_weighted_centralized_result_approx}.
\end{proof}

Comparing \eqref{eq:asymptotic_weighted_centralized_result} with \eqref{eq:asymptotic_theorem_result_simple}, we observe that the diffusion strategy achieves the same performance as the weighted centralized strategy \eqref{eq:weighted_full_gradient} since we are free to let $p_k = \pi_k$. Conversely, it is possible to optimize $\{\pi_k\}$ in the same manner as we did in \eqref{eq:quadratic_optimization_p} to obtain that the optimal choice of $\{\pi_k\}$ would yield an excess-risk performance for the weighted centralized algorithm of:
\begin{align}
	\E_\w\{J(\w_{i-1})-J(w^o)\} &\sim \frac{\mu}{2i} \cdot \frac{1}{\mathds{1}^\T Z^{-1} \mathds{1}} \label{eq:asymptotic_weighted_full_gradient_ER}
\end{align}
where $Z$ was defined earlier in \eqref{eq:Z}, which is identical to the one obtained by the diffusion strategy with the optimized $p$ found in \eqref{eq:optimized_ER}.

The reader may also notice that it is further possible to optimize the asymptotic excess-risk curves \eqref{eq:optimized_ER} and \eqref{eq:asymptotic_weighted_full_gradient_ER} over the initial step-size $\mu$ for the diffusion and weighted centralized strategies. This would entail solving an optimization problem of the form:
\begin{align}
	\min_{2\lambda_{\min}\mu>1}\quad&  \frac{1}{2} \cdot \frac{\mu}{\mathds{1}^\T Z^{-1} \mathds{1}} = \frac{1}{2}\cdot \frac{\mu}{\mathds{1}^\T \!\left(\displaystyle \sum_{m=1}^M \!\! \frac{\lambda_m\mu }{2\lambda_m \mu-1}  R_{v,m}\right)^{\!\!-1} \!\!\mathds{1}} \label{eq:optimization_mu}
\end{align}
While this problem is generally difficult to solve in closed-form, it is evident that the solution is $\mu = 1/\lambda$ when $\lambda_m = \lambda$ for all $1\leq m\leq M$ (which is the case where the asymptotic Hessian matrix $H$ is white). In general, however, the optimization problem can be solved numerically to obtain the optimal rate when $H$ is known or can be estimated (such as in the case of the quadratic cost \eqref{eq:LMS}).

\subsection{Comparison to Consensus Strategies}
\label{sec:consensus}
In this section, we use the performance results derived so far to show that the diffusion strategy \eqref{eq:A}--\eqref{eq:C2} has several advantages over the consensus strategy \eqref{eq:consensus}. Specifically, we will show that, asymptotically, the consensus excess-risk curve is worse than the diffusion excess-risk curve. In addition, we will observe through simulations that the overshoot during the transient phase may be significantly worse for consensus implementations. 

To examine the differences in behavior, we will consider the network excess-risk:
\begin{align*}
	\textrm{ER}(i) \triangleq \frac{1}{N} \sum_{k=1}^N \textrm{ER}_k(i)
\end{align*}

For simplicity, we assume $A$ is symmetric; more generally, we can consider combination policies $A$ that are close-to-symmetric and employ arguments similar to \cite{shine_diffusion_consensus}. The final conclusion will be similar to the arguments given here. We can now establish the following result.
\begin{theorem}[\textbf{Comparing network excess-risks}]
\label{thm:consensus}
Let Assumptions \ref{ass:HessianAssumption}-\ref{ass:topology} hold and  $2 \lambda_{\min} \mu \!>\! 1$ with a symmetric $A$. Then, it holds that the asymptotic network excess-risk achieved by the diffusion strategy \eqref{eq:A}-\eqref{eq:C2} is upper bounded by (and, hence, better than) that achieved by the consensus strategy \eqref{eq:consensus}:
\begin{align}
	\mathrm{ER}^\mathrm{diff}(i) \leq \mathrm{ER}^\mathrm{cons}(i),
	 \quad\quad\textrm{as\ } i\rightarrow\infty \label{eq:consensus_diffusion}
\end{align}
\end{theorem}
\begin{proof}
See Appendix \ref{sec:Proof_diffusion_consensus}.
\end{proof}
Result \eqref{eq:consensus_diffusion} implies that, asymptotically, the curve for the network excess-risk for the diffusion algorithm will be upper-bounded by the curve for the consensus algorithm. Even though both algorithms have the same computational complexity, the diffusion algorithm achieves better performance because it succeeds at diffusing the information more thoroughly through the network. This conclusion mirrors the result found in \cite{shine_diffusion_consensus}, which studied the difference in performance between the diffusion and consensus strategies in the \emph{constant} step-size scenario. We conclude, therefore, that the manner by which the diffusion strategy propagates information through the network continues to outperform the consensus strategy even in the diminishing step-size case.

\section{Application to distributed inference over regression and classification models}
\label{sec:sim}
In order to illustrate our results, we consider two applications related to distributed linear regression and distributed logistic regression (a classification model). In the former case, we will observe that our strategy for choosing the combination weights outperforms the doubly-stochastic weights of \cite{Bianchi}, and outperforms the consensus solution. We will also see that it is possible to select  $J_k(w)$ such that the excess-risk can be interpreted as the Kullback-Leibler (KL) divergence between the true likelihood function and an estimated likelihood function.

\subsection{Quadratic Risk}
\label{ssec:quadratic}
Consider the estimation of an unknown vector $w^o \in \mathbb{R}^{M\times 1}$, from observations:
\begin{align}
	\y_k(i) = \h_{k,i}^\T w^o + \z_k(i) \label{eq:y_k_linear_regression}
\end{align}
where $\z_k(i)$ is zero-mean Gaussian noise with variance $\sigma_{z,k}^2$ and all regressors $\h_{k,i} \in \mathbb{R}^{M \times 1}$ are zero-mean, spatially independent, and identically distributed. This is a common scenario in linear regression. The likelihood function is given by:
\begin{align}
	p(\y_{k}(i) | \h_{k,i}, w) = \frac{1}{\sqrt{2 \pi \sigma_{z,k}^2}} \exp\left(\!-\frac{(\y_{k}(i) \!-\! \h_{k,i}^\T w)^2}{2\sigma_{z,k}^2}\!\right) \label{eq:likelihood}
\end{align}
The maximum-likelihood estimate of $w^o$ is obtained by minimizing the negative log-likelihood function; this step motivates the following choice for the loss function:
\begin{align}
	Q(w, \h_{k,i}, \y_{k}(i)) &\triangleq  (\y_k(i) - \h_{k,i}^\T w)^2 \label{eq:Q_linear_regression}
\end{align}
In this case, the risk function \eqref{eq:J_w} becomes:
\begin{align}
	J_k(w) = \E (\y_k(i) - \h_{k,i}^\T w)^2
\end{align}
and, using \eqref{eq:likelihood}, the excess-risk \eqref{eq:ER_k} is equivalent to measuring the KL divergence:
\begin{align}
	D_{\textrm{KL}}( p(\y_{k}(i) &| \h_{k,i}, w^o)	 || p(\y_{k}(i) | \h_{k,i}, w)) \nonumber\\
	&\triangleq \E\left\{\log\left(\frac{p(\y_{k}(i) | \h_{k,i}, w^o)}{p(\y_{k}(i) | \h_{k,i}, w)}\right)\right\}
\end{align}
The results derived in Sec.\ref{sec:main_results} apply to this definition of excess-risk, and we can therefore expect that the KL divergence curve of the diffusion strategy \eqref{eq:A}--\eqref{eq:C2} will be below that of the consensus algorithm \eqref{eq:consensus}, and that the diffusion strategy will outperform both the non-cooperative solution \eqref{eq:non_coop} and the centralized algorithm \eqref{eq:full_gradient}. 
\begin{figure*}[t!] 
	\centering
    \begin{subfigure}[t]{0.25\textwidth}
        \includegraphics[width=\textwidth]{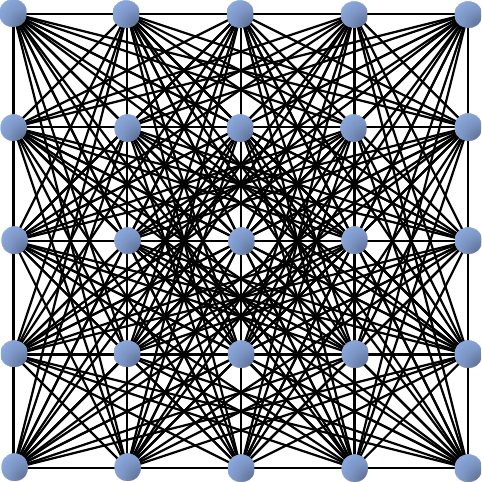}
        \caption{The fully-connected network topology used to obtain the simulation results in Sec.~\ref{ssec:quadratic}.}
        \label{fig:exp1_topology_FC}
    \end{subfigure}%
    \quad\quad
    \begin{subfigure}[t]{0.25\textwidth}
        \includegraphics[width=\textwidth]{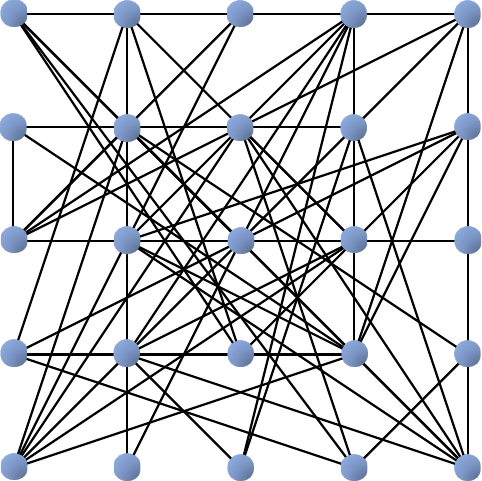}
        \caption{The sparsely-connected network topology used to obtain the simulation results in Sec.~\ref{ssec:quadratic}.}
        \label{fig:exp1_topology_Sparse}
    \end{subfigure}%
    \quad\quad
    \begin{subfigure}[t]{0.25\textwidth}
        \includegraphics[width=\textwidth]{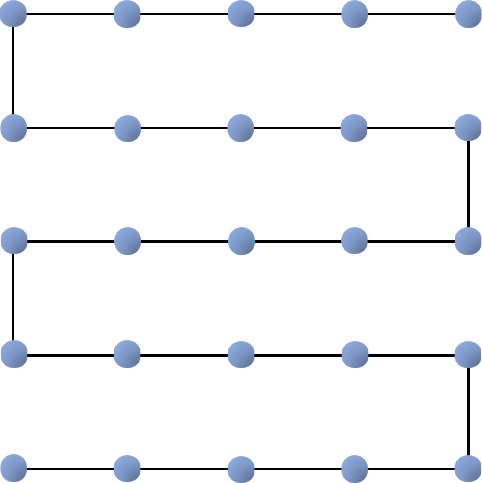}
        \caption{The line-topology used to obtain the simulation results in Sec.~\ref{ssec:quadratic}.}
        \label{fig:exp1_topology_Line}
    \end{subfigure}%
    \\\vspace{1\baselineskip}
    \begin{subfigure}[t]{0.4\textwidth}
        \includegraphics[width=\textwidth]{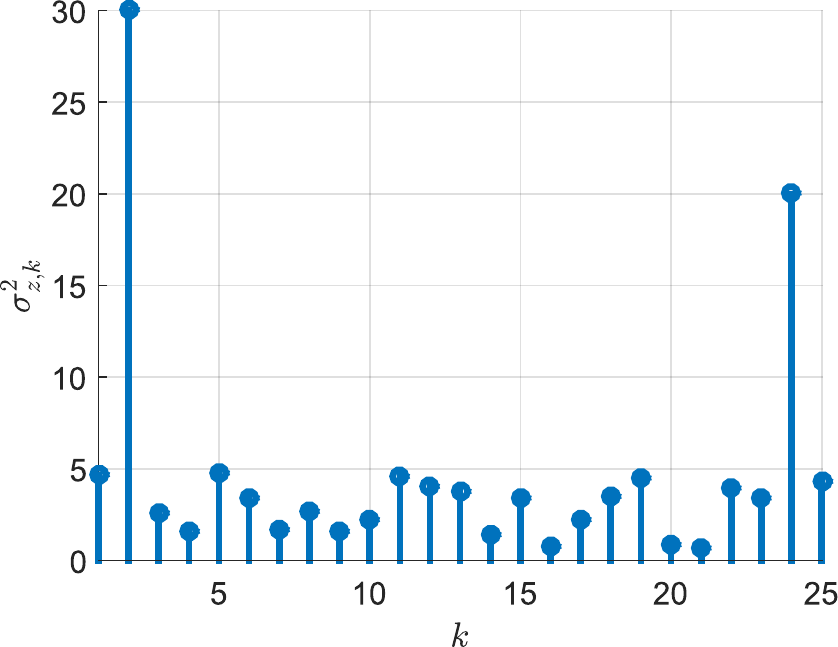}
        \caption{Plot of noise variances across the $50$ nodes in the simulation.}
        \label{fig:noise_variance}
    \end{subfigure}%
    \quad\quad
    \begin{subfigure}[t]{0.4\textwidth}
        \includegraphics[width=\textwidth]{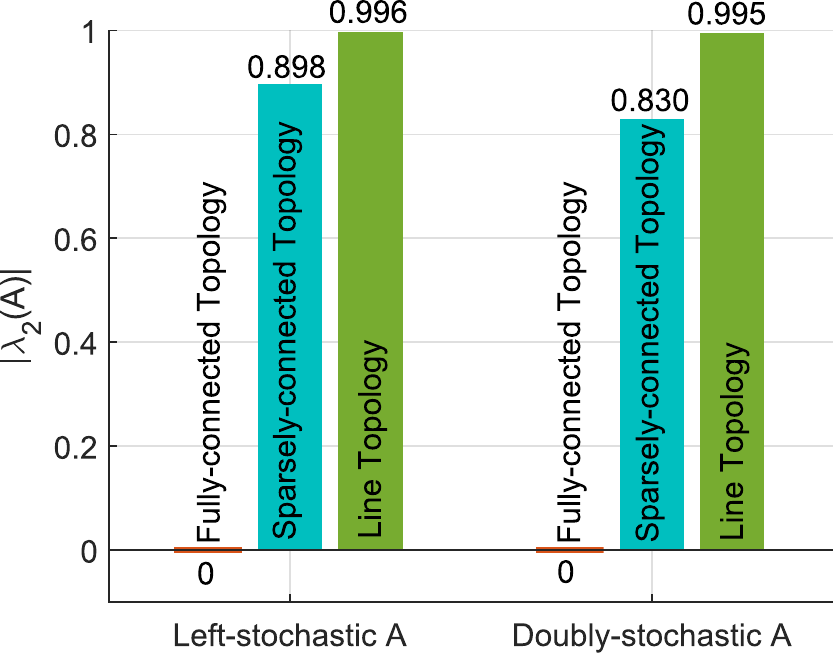}
        \caption{Plot of the second largest eigenvalue of the combination matrices for various topologies. A fully-connected $A$ matrix always has a single non-zero eigenvalue at $1$.} 
        \label{fig:evals_combination}
    \end{subfigure}%
    \caption{Simulation parameters for quadratic risk optimization: (a) The fully-connected network topology, (b) The sparsely-connected network topology, (c) The line-topology used in the simulations, (d) The noise variance at each node, and (e) The eigenvalue spectrum of the combination matrix $A$.}
    \label{fig:simulation_parameters}
\end{figure*} 
\begin{figure*}
\centering
\includegraphics[width=0.8\textwidth]{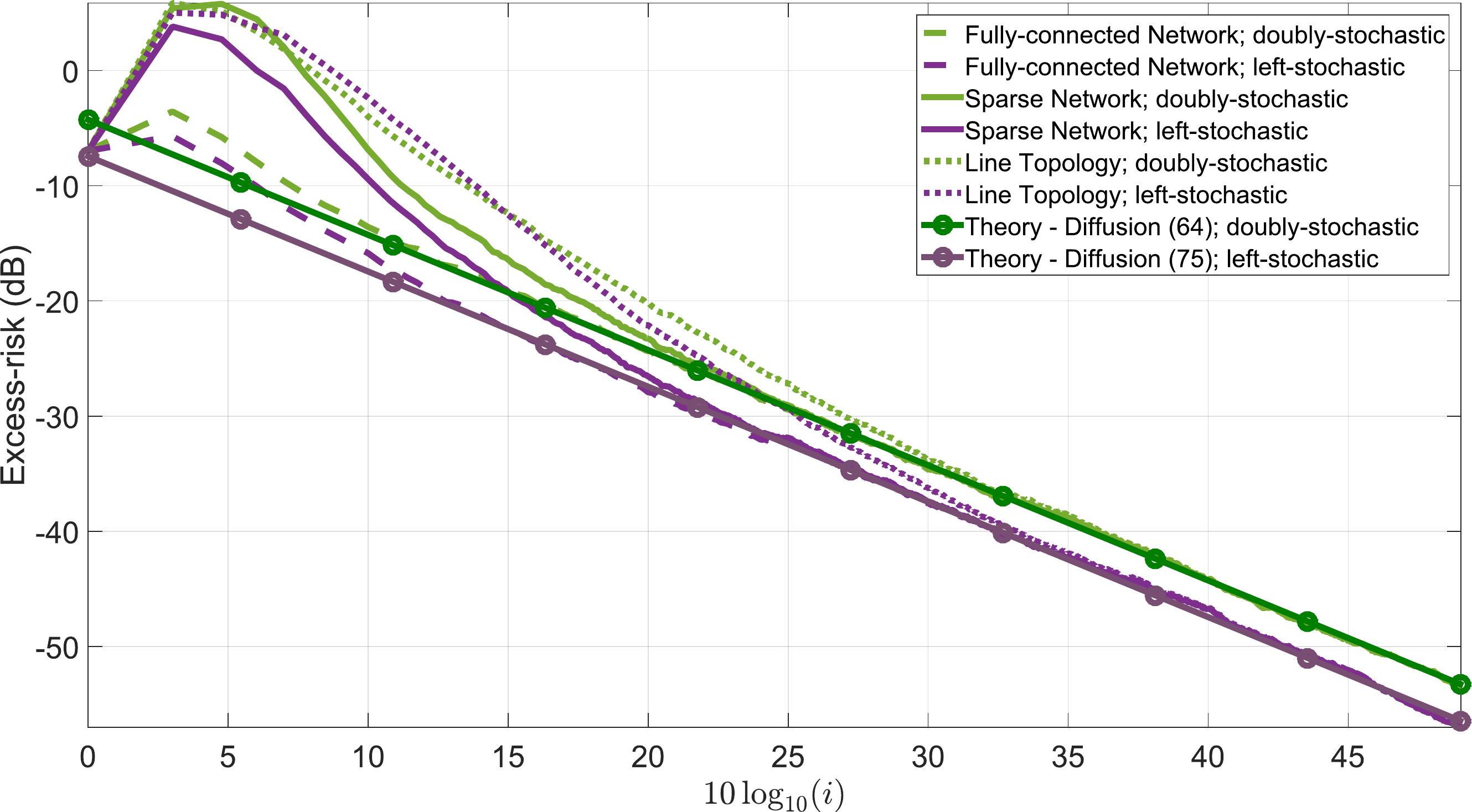}
\caption{Simulation of the diffusion strategy for (1) a fully-connected network; (2) a sparsely-connected network; and (3) a line-topology. The network topologies and simulation parameters are listed in Fig.~\ref{fig:simulation_parameters}. Best viewed in color.}
\label{fig:invariance_to_topology}
\end{figure*}
\begin{figure*}
		\centering
		\begin{subfigure}[t]{0.48\textwidth}
			\includegraphics[width=1\textwidth]{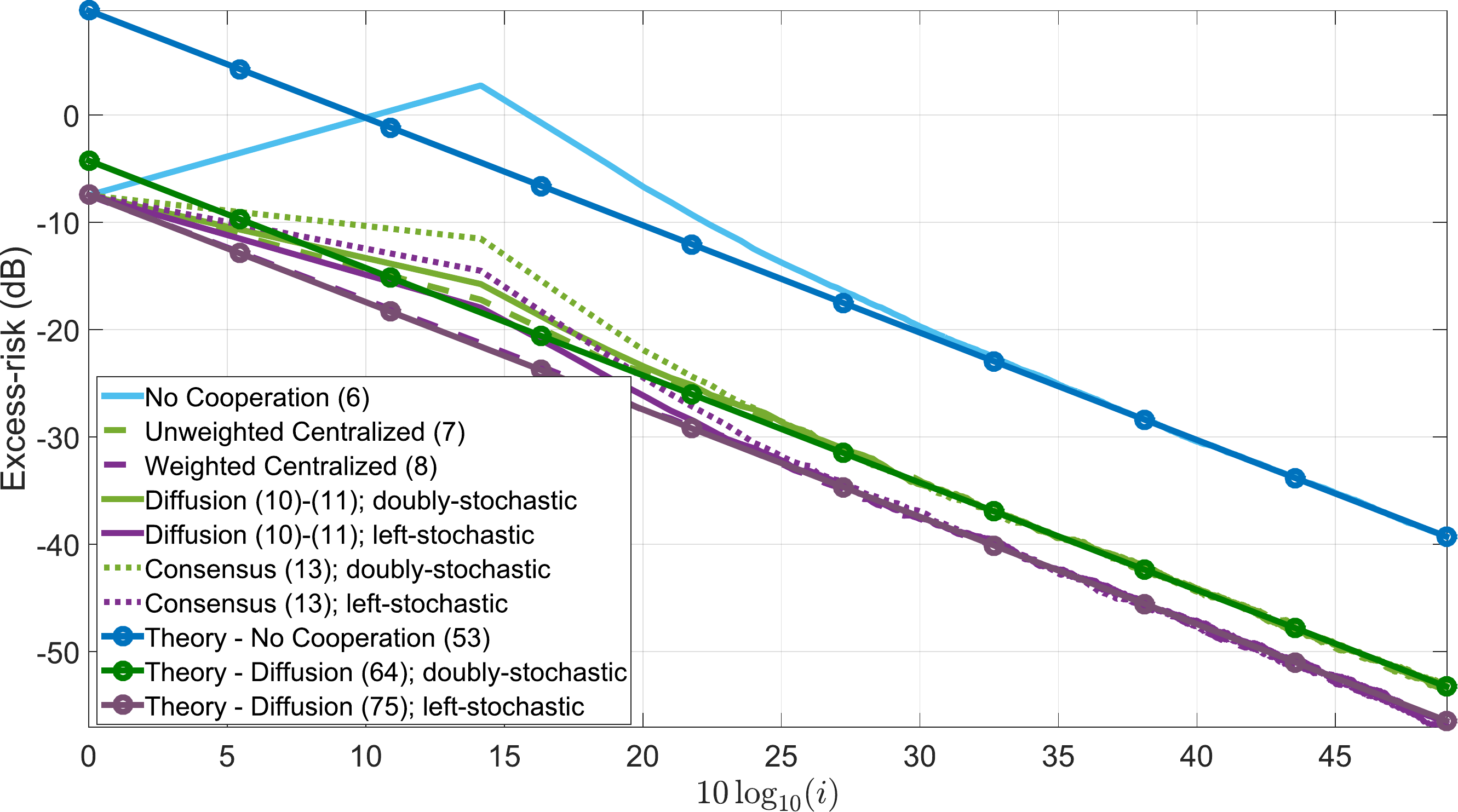}
			\caption{Condition number of the matrix $H$ is $1$. Comparison of non-cooperative, centralized, diffusion, and consensus strategies.}
			\label{fig:cond1}
		\end{subfigure}%  
		\quad\quad
		\begin{subfigure}[t]{0.48\textwidth}
			\includegraphics[width=1\textwidth]{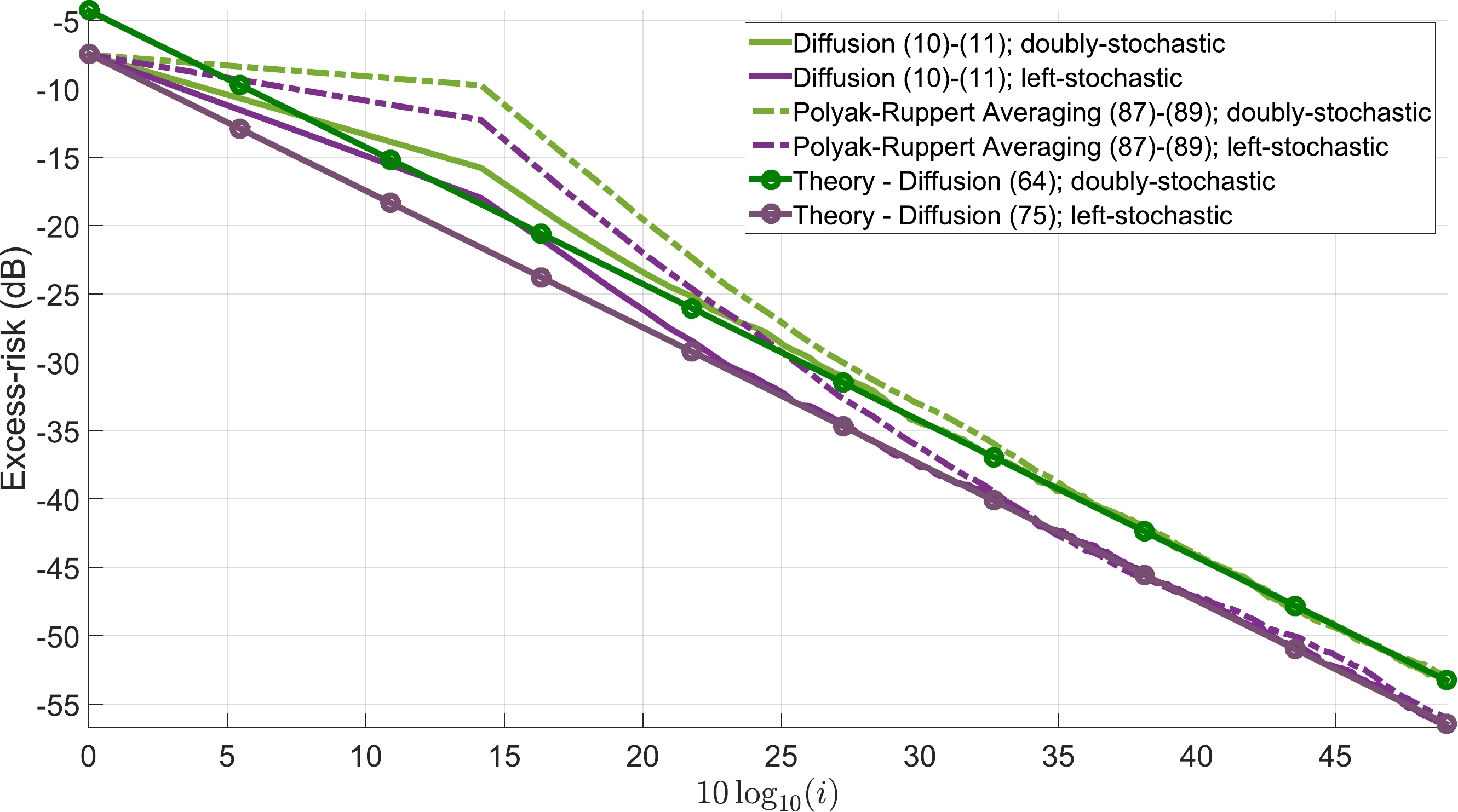}
			\caption{Condition number of the matrix $H$ is $1$. Comparison of non-averaged and averaged diffusion strategies.}
			\label{fig:cond1_pr}
		\end{subfigure}%
		\\ 
		\begin{subfigure}[t]{0.48\textwidth} 
			\includegraphics[width=1\textwidth]{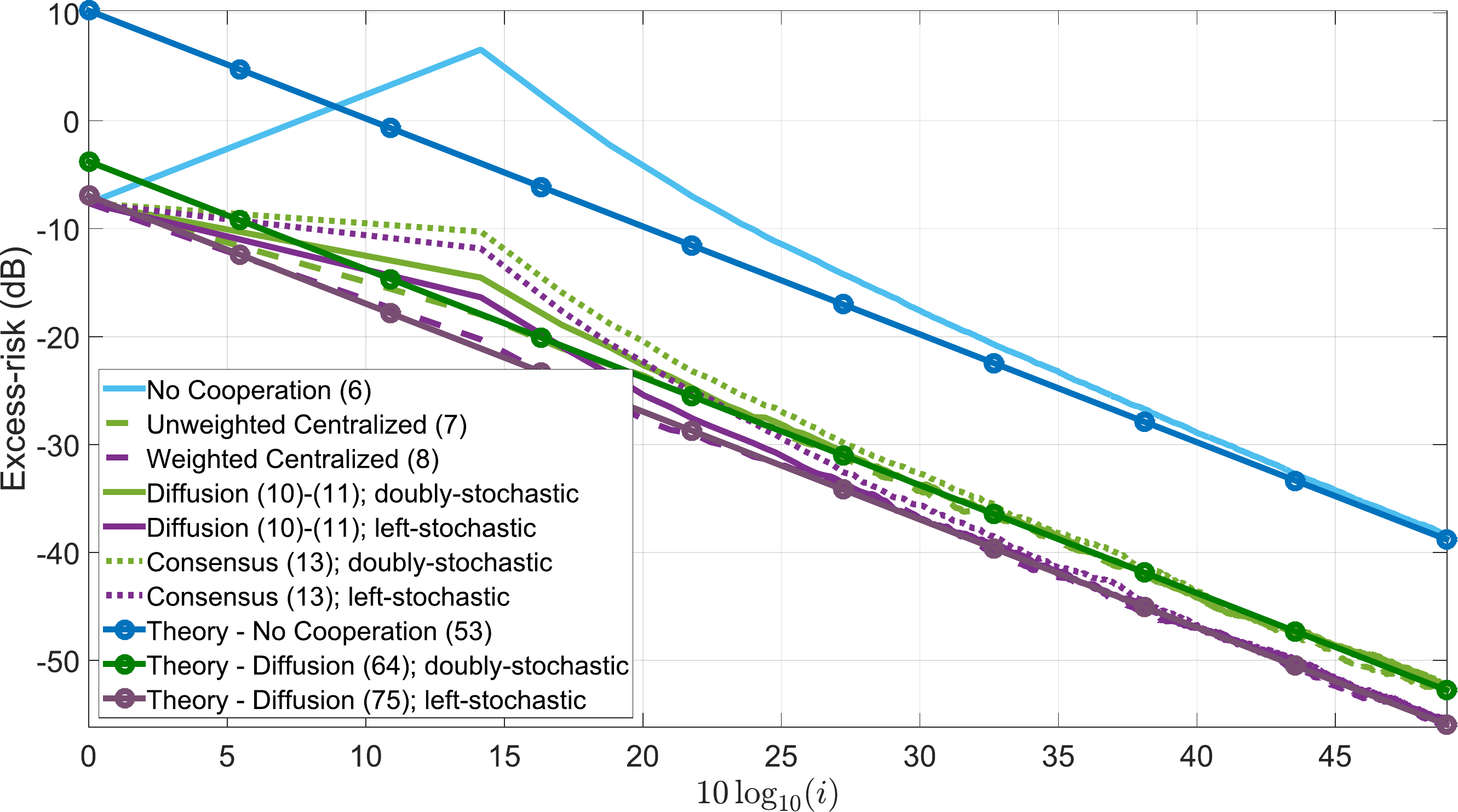}
			\caption{Condition number of the matrix $H$ is $2$. Comparison of non-cooperative, centralized, diffusion, and consensus strategies.}
			\label{fig:cond2}
		\end{subfigure}%
		\quad\quad
		\begin{subfigure}[t]{0.48\textwidth}
			\includegraphics[width=1\textwidth]{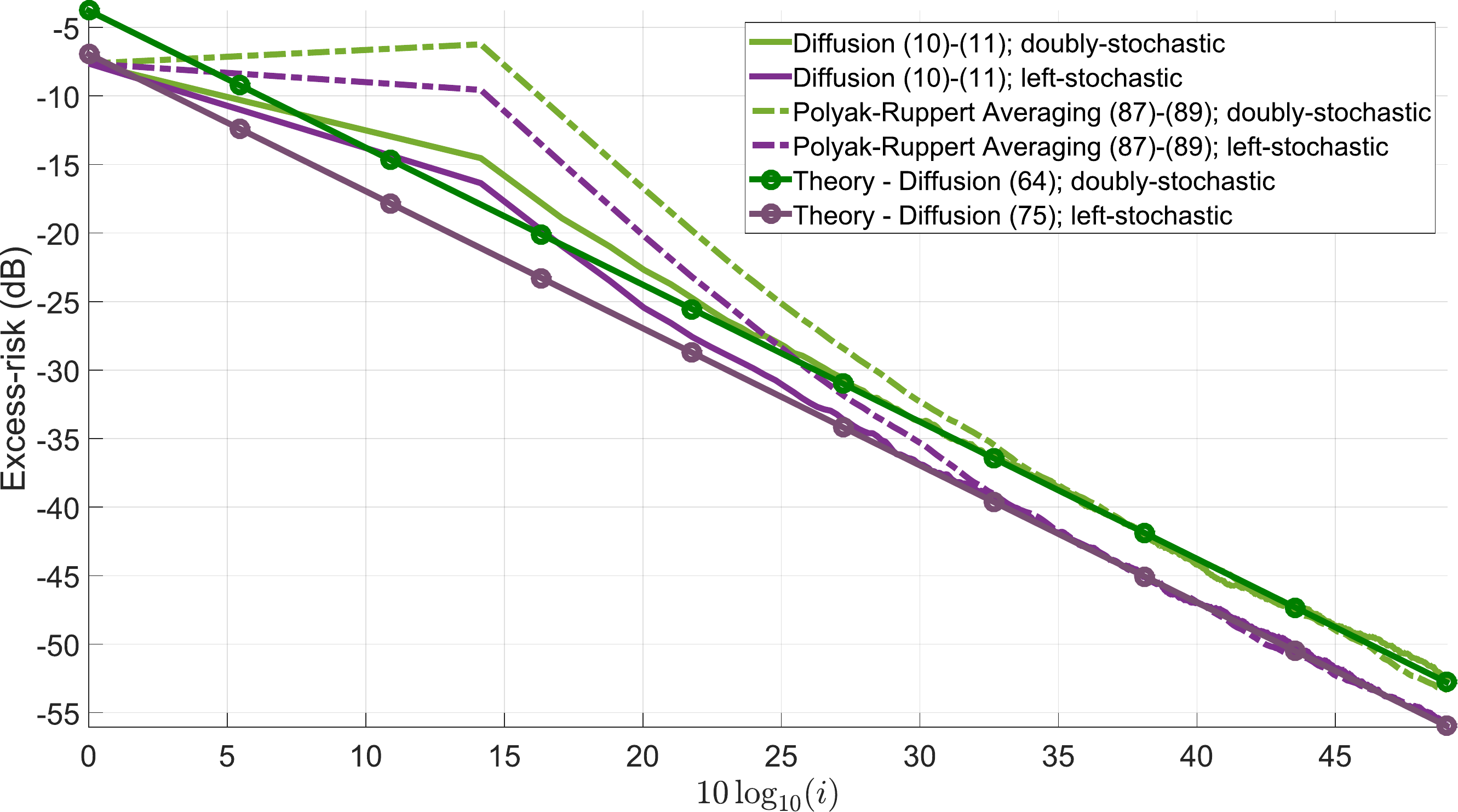}
			\caption{Condition number of the matrix $H$ is $2$. Comparison of non-averaged and averaged diffusion strategies.}
			\label{fig:cond2_pr}
		\end{subfigure}%		
		\\
		\begin{subfigure}[t]{0.48\textwidth} 
			\includegraphics[width=1\textwidth]{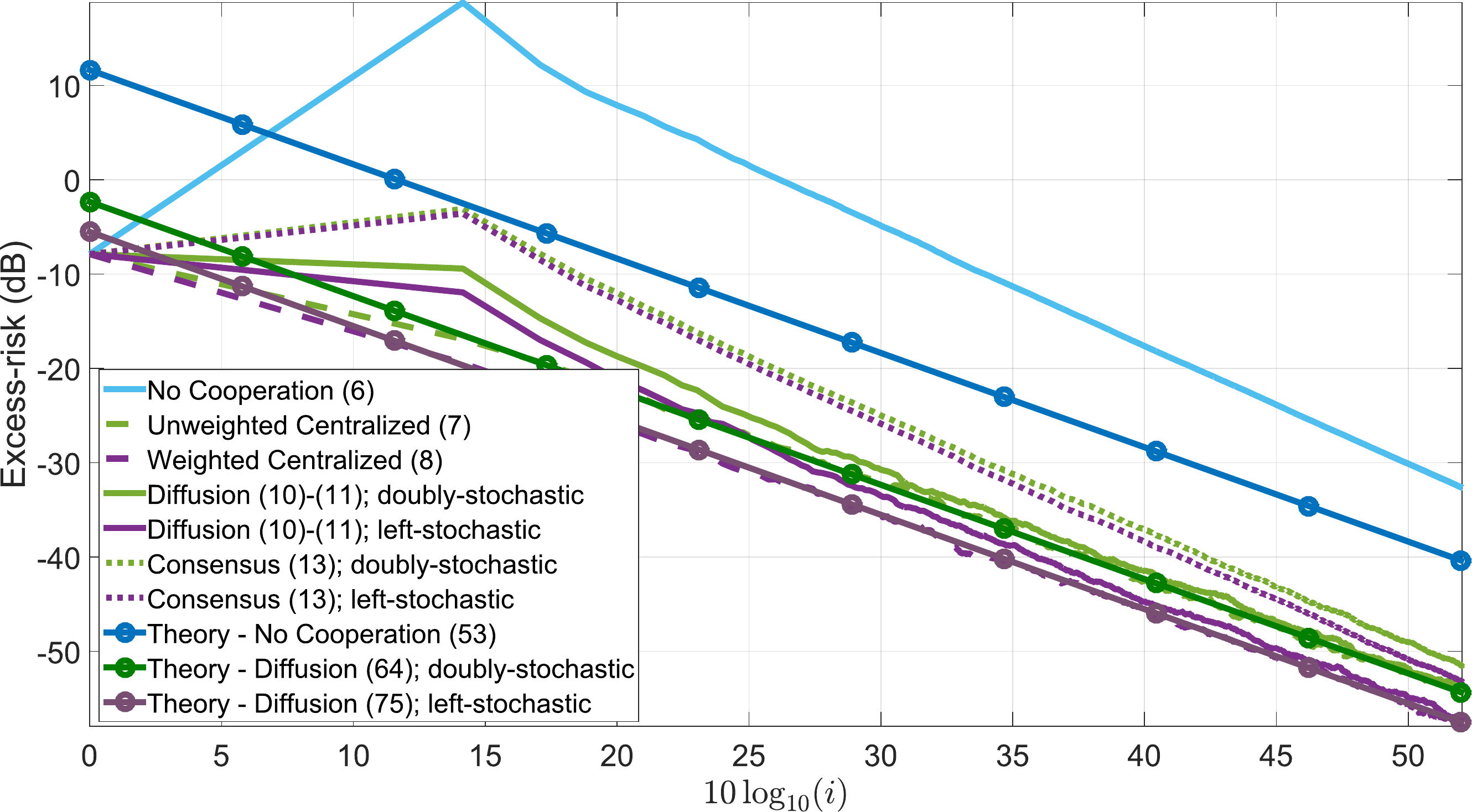}
			\caption{Condition number of the matrix $H$ is $4$. Comparison of non-cooperative, centralized, diffusion, and consensus strategies.}
			\label{fig:cond4} 
		\end{subfigure}%
		\quad\quad
		\begin{subfigure}[t]{0.48\textwidth}
			\includegraphics[width=1\textwidth]{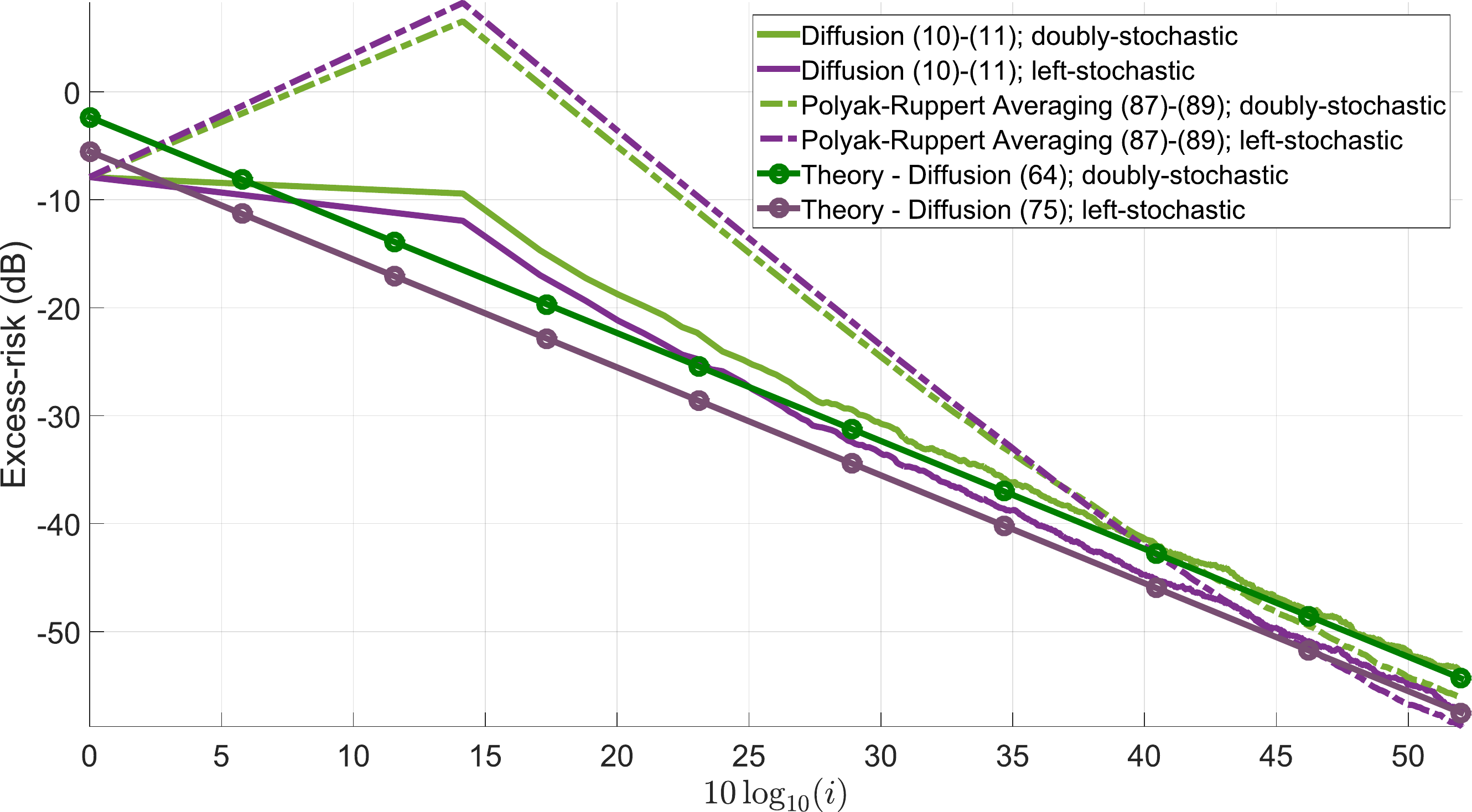}
			\caption{Condition number of the matrix $H$ is $4$. Comparison of non-averaged and averaged diffusion strategies.}
			\label{fig:cond4_pr}
		\end{subfigure}%		
		\caption{Comparison between learning curves of non-cooperative processing \eqref{eq:non_coop}, diffusion algorithm \eqref{eq:A}-\eqref{eq:C2}, consensus algorithm \eqref{eq:consensus}, unweighted centralized \eqref{eq:full_gradient}, weighted centralized \eqref{eq:weighted_full_gradient}, and Polyak-Ruppert averaging scheme \eqref{eq:PR_A}--\eqref{eq:PR_AV} for quadratic loss minimization. Best viewed in color.}
		\label{fig:sim_bound_approx}  
	\end{figure*}

We simulate the above results by considering a network with $N=25$ nodes as illustrated in Figs.~\ref{fig:exp1_topology_FC}--\ref{fig:exp1_topology_Line}. The quadratic loss function optimized at each node is \eqref{eq:Q_linear_regression}, where $\h_{k,i}$ is a random vector in $\mathds{R}^{M \times 1}$ and $M=4$ and is a Gaussian random vector with i.i.d. elements and zero mean and unit variance. In addition, the scalar observation $\y_k(i)$ is generated according to \eqref{eq:y_k_linear_regression}, where $\z_k(i)$ is a Gaussian random variable with zero mean and variance $\sigma_{z,k}^2$ as illustrated in Fig.~\ref{fig:noise_variance}. The combination weights are chosen according to the Hastings rule \eqref{eq:Metropolis_DS} using either the optimized Perron vector \eqref{eq:p^o} or the uniform Perron vector that leads to the doubly-stochastic combination matrix $p = \frac{1}{N} \mathds{1}_N$. The magnitude of the second largest eigenvalue of the left- and doubly-stochastic matrices are illustrated in Fig.~\ref{fig:evals_combination}. The step-size was optimized via \eqref{eq:optimization_mu}. Observe that we let two out of the $25$ nodes have particularly bad noise conditions. This should highlight the benefit of using left-stochastic weights as opposed to doubly-stochastic weights. The simulation results are illustrated in Fig.~\ref{fig:sim_bound_approx}. The curves are averaged over $100$ experiments. The network topologies examined in the simulations in this section are illustrated illustrated in Fig.~\ref{fig:exp1_topology_FC}--\ref{fig:exp1_topology_Line}.

\subsubsection{Asymptotic Performance invariance due to network topology}
We first illustrate one of the main conclusions we draw from Theorem \ref{thm:asymptotic_term}, which is that the diffusion strategy's asymptotic performance is independent of the particular network topology. To do this, we simulate the diffusion strategy with three network topologies: (1) Fully-connected (Fig.~\ref{fig:exp1_topology_FC}); (2) Sparsely-connected (Fig.~\ref{fig:exp1_topology_Sparse}); and (3) Line topology (Fig.~\ref{fig:exp1_topology_Line}). We utilize the Hastings rule \eqref{eq:Metropolis_DS} to obtain both left-stochastic and doubly-stochastic combination matrices by setting the desired Perron vector $p$ to the optimal vector \eqref{eq:p^o} or $\frac{1}{N} \mathds{1}_N$. Such combination matrices can be generated for any connected network topology. The eigen-spectrum of the resulting combination matrices for each topology are illustrated in Fig.~\ref{fig:evals_combination}. We observe that as the network becomes better connected, the second-largest eigenvalue is reduced. This eigen-value generally determines the convergence speed of the transient terms.

Figure~\ref{fig:invariance_to_topology} illustrates the simulation results for all of the simulated network topologies. We observe that regardless of the network topology, the diffusion strategy converges to the curve described by \eqref{eq:asymptotic_theorem_result}, albeit at different speeds. The fully-connected network case is theoretically equivalent to the centralized strategies \eqref{eq:full_gradient} and \eqref{eq:weighted_full_gradient}, depending on the vector $p$ used. We conclude, therefore, that the diffusion strategy is invariant to the particular network topology and its asymptotic excess-risk curve only depends on the Perron vector of the combination matrix, which can be freely designed for connected networks.

\subsubsection{Performance comparisons}

We now simulate the diffusion strategy as well as other algorithms discussed in the manuscript. In addition to the non-cooperative \eqref{eq:non_coop}, diffusion \eqref{eq:A}--\eqref{eq:C2}, consensus \eqref{eq:consensus}, unweighted centralized \eqref{eq:full_gradient}, and weighted centralized \eqref{eq:weighted_full_gradient} strategies, we also examine the performance of non-cooperative and distributed Polyak-Ruppert (PR) averaging strategies. The non-cooperative PR algorithm can be described as:
\begin{align}
	\boldsymbol{\phi}_{k,i} &= \boldsymbol{\phi}_{k,i-1}\!-\!\mu(i) \nabla_{\!w} Q(\!\boldsymbol{\phi}_{k,i-1},\!\x_{k,i}\!) & &[\textrm{adaptation}]\\
	\w_{k,i}  &= \w_{k,i-1} + \frac{1}{i} \left(\boldsymbol{\phi}_{k,i} - \w_{k,i-1}\right)& &[\textrm{averaging}]
\end{align}
which is simply the non-cooperative algorithm \eqref{eq:non_coop} operating independently, and a time-averaging process of the iterates to smooth the estimates. On the other hand, the distributed version that we implement for comparison takes the form:
\begin{align}
			\boldsymbol{\psi}_{k,i} \!&=\! \boldsymbol{\phi}_{k,i-1} \!-\! \mu(i) \nabla_w Q(\boldsymbol{\phi}_{k,i-1} ,\x_{k,i}) & &\!\textrm{[adaptation]} \label{eq:PR_A}\\
			\boldsymbol{\phi}_{k,i} \!&=\! \sum_{\ell=1}^N a_{\ell k} \boldsymbol{\psi}_{\ell,i} & &\!\textrm{[aggregation]} \label{eq:PR_C2}\\
			\w_{k,i}  &= \w_{k,i-1} + \frac{1}{i} \left(\boldsymbol{\phi}_{k,i} - \w_{k,i-1}\right)  & &\!\textrm{[averaging]} \label{eq:PR_AV}
\end{align} 
where the combination matrix with elements $\{a_{\ell k}\}$ is the same as the one used for the diffusion and consensus strategies.

It has been shown for single-agent stochastic approximation algorithms \cite{Bach_NIPS,polyak1992acceleration,Nemirovski,shalev2009stochastic} that the  averaging operation allows the algorithm to be robust to loss of strong-convexity in the risk function (the algorithm will continue to converge at the fastest possible rate for the non-strongly-convex problem, albeit slower than the rate it would have achieved for the strongly-convex problem) in addition to 
 achieving the best possible rate regardless of the eigen-structure of the Hessian matrix $H$ (which is important when $H$ is ill-conditioned). Unfortunately, these algorithms have also been shown to be sensitive to initialization \cite{Bach_NIPS} in that transient terms in their excess-risk curve decay at a relatively slow rate of $1/i^2$---in comparison to stochastic gradient descent, which forgets the initialization at a rate of $1/i^{2\lambda_{\min}\mu}$ asymptotically (Theorem \ref{thm:transient_term}), which can be made to decay faster with the choice of $\mu$, and faster in the earlier stages \cite{Bach_NIPS}. In addition, the step-size sequence $\mu(i)$ for the PR algorithms must decay slower than $1/i$; for example $\mu(i) = \mu/i^{2/3}$. Various methods have been suggested for avoiding such drawbacks such as to start the averaging process after some number of iterations \cite{kushner2003stochastic} or selecting a small initial step-size $\mu$ to reduce the size of the over-shoot. Unfortunately, it is usually not clear how many iterations should elapse before averaging starts. This is still a useful research direction to explore. In our simulations we use $\mu(i) = \mu/i^{2/3}$ as suggested in \cite{Bach_NIPS} for the PR averaging strategies.

The curves illustrate that the difference in performance between non-cooperative processing \eqref{eq:non_coop} and the diffusion algorithm with doubly-stochastic weights (i.e., the algorithm studied in \cite{Bianchi}) is about $14$dB ($10\log_{10}(N)$). We also observe that \eqref{eq:A}-\eqref{eq:C2} achieves $10$dB per decade decay in simulation ($1/i$ rate), and that the diffusion strategy with left-stochastic weights outperforms the un-weighted centralized solution \eqref{eq:full_gradient} and the doubly-stochastic implementation. In comparison to the consensus algorithm \eqref{eq:consensus}, the diffusion algorithm \eqref{eq:A}-\eqref{eq:C2} is seen to have better transient performance, and the  network excess-risk for the consensus strategy remains higher than that for diffusion regardless of the combination matrix used (as predicted by Theorem \ref{thm:consensus}). In addition, we observe that the left-stochastic diffusion algorithm approaches the same performance curve as the weighted centralized strategy \eqref{eq:weighted_full_gradient} as predicted by theory. 

We also observe that the Polyak-Ruppert averaging diffusion scheme exhibits worse performance than the non-averaged case when the condition number of $H$ is unity. This is due to the poor transient performance of the averaged algorithm \cite{Bach_NIPS,bottou2012stochastic}. However, as the condition number of $H$ worsens, the averaged diffusion algorithm outperforms the standard diffusion scheme.

\subsection{Regularized Logistic Regression}
\label{ssec:logistic}
\begin{figure*}
	\centering
	\begin{subfigure}[t]{0.25\textwidth}
			\centering
			\includegraphics[width=1\textwidth]{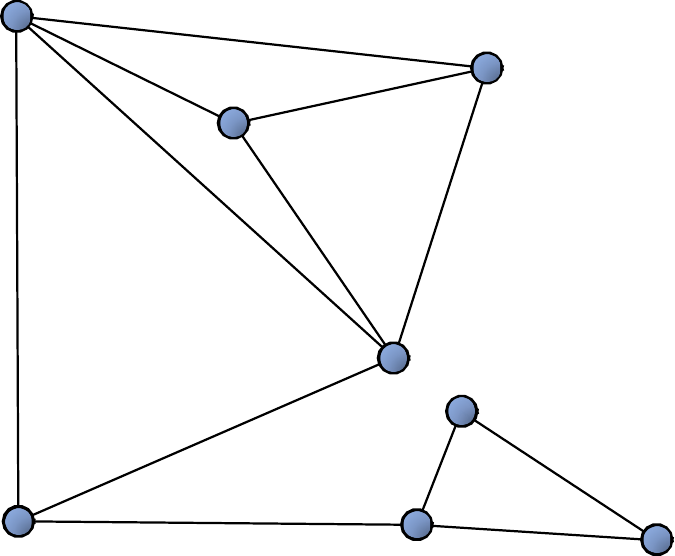}
			\caption{Network topology used to obtain the simulation results in Sec.~\ref{ssec:logistic}.}
			\label{fig:topology_RLR}
		\end{subfigure}%  
	\quad\quad\quad
	\begin{subfigure}[t]{0.4\textwidth}
			\centering
			\includegraphics[width=1\textwidth]{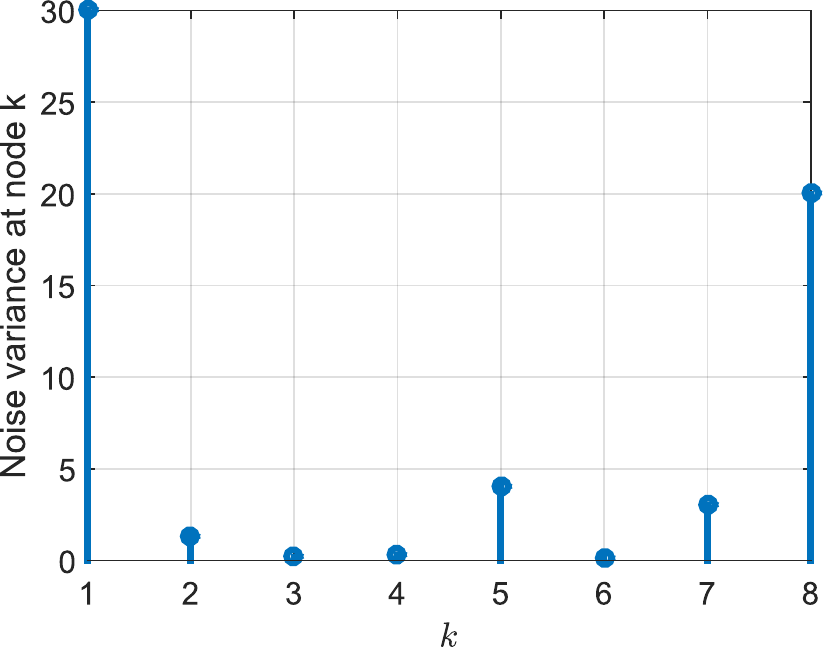}
			\caption{Plot of noise variances across the $8$ nodes in the simulation.}
			\label{fig:noise_variance_RLR}
		\end{subfigure}%  
		\\\vspace{1\baselineskip}
		\begin{subfigure}[t]{0.75\textwidth}
			\centering
			\includegraphics[width=1\textwidth]{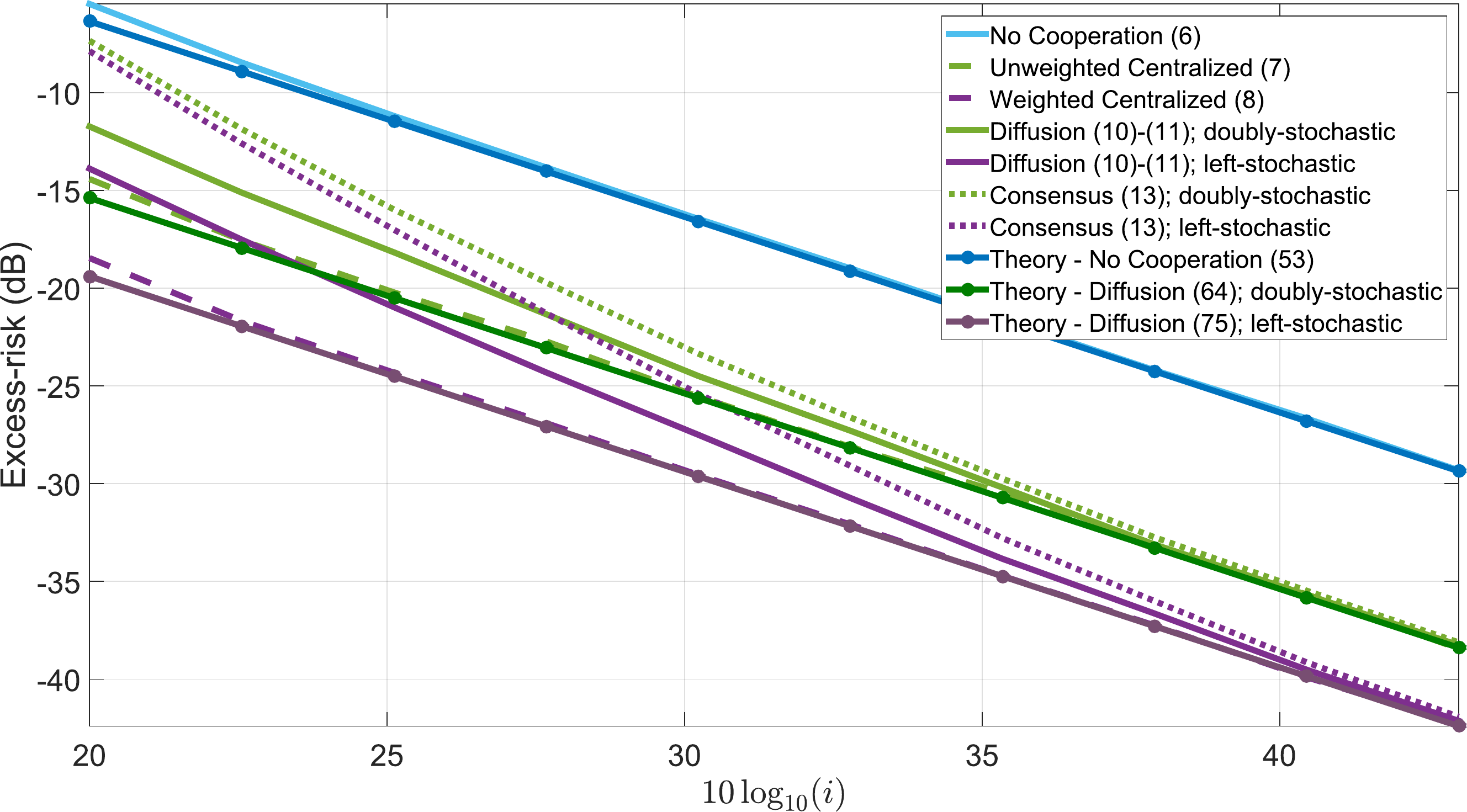}
			\caption{Regularized KL divergence attained by nodes that utilize the non-cooperative
algorithm \eqref{eq:non_coop}, the unweighted centralized algorithm \eqref{eq:full_gradient}, the weighted centralized algorithm \eqref{eq:weighted_full_gradient}, the diffusion strategy \eqref{eq:A}-\eqref{eq:C2}, and consensus strategy \eqref{eq:consensus} with optimized weights \eqref{eq:p^o} and with doubly-stochastic weights ($p = \frac{1}{N}\mathds{1}_N$). Theoretical curves are from \eqref{eq:asymptotic_individual_result_approx}, \eqref{eq:N_fold} and \eqref{eq:asymptotic_centralized_result_approx}. Curves are averaged over $2000$ experiments.}
			\label{fig:simulation}
		\end{subfigure}%
	\caption{Plot of the simulation parameters and results for the regularized logistic regression simulation.}
	\label{fig:RLR_sim}
\end{figure*}  
 
Now consider an application where the dependent variable $\y_{k}(i)$ is binary (i.e, it assumes values from the set $\{+1,-1\}$) with  $\h_{k,} \in \mathbb{R}^{M\times 1}$. The log-odds function is assumed to obey the  linear model \cite[p.~117]{Theodoridis}:
\begin{align}
	\log\left(\frac{p(\y_{k}(i) | \h_{k,i}; w)}{1-p(\y_{k}(i) | \h_{k,i}; w)}\right) = \y_{k}(i)\h_{k,i}^\T w
\end{align}
Solving for $p(\y_{k,i} | \h_{k,i}; w)$, we have that the likelihood is given by
\begin{align}
	p(\y_{k}(i) | \h_{k,i};  w) = \frac{1}{1+e^{-\y_{k}(i) \h_{k,i}^\T w}}
\end{align}
Then, the maximum-likelihood estimate of $w^o$ minimizes the negative log-likelihood function:
\begin{align}
	Q(w,\h_{k,i}, \y_{k}(i)) = \log\left(1+e^{-\y_{k}(i) \h_{k,i}^\T w}\right)
\end{align}
Observe that $Q(w,\h_{k,i},\y_k(i))$ is not strongly-convex. In order to alleviate this difficulty, it is possible to add a small regularizer so that \cite{zhu2009maximum,lafferty2001boosting}:
\begin{align}
	Q(w,\h_{k,i}, \y_{k}(i)) = \frac{\rho}{2} \|w\|^2 + \log\left(1+e^{-\y_{k}(i) \h_{k,i}^\T w}\right)
\end{align}
and the excess-risk then has the interpretation of being the \emph{regularized} KL divergence:
\begin{align}
	D_{\textrm{RKL}}&\big(p(\y_{k}(i) | \h_{k,i};  w^o)\parallel p(\y_{k}(i) | \h_{k,i}; w)\big) \nonumber\\
	&\triangleq \frac{\rho}{2}\left(\|w\|^2 \!-\! \|w^o\|^2\right) + \E \left\{\log\left(\frac{p(\y_{k}(i) | \h_{k,i}; w^o)}{p(\y_{k}(i) | \h_{k,i};  w)}\right)\right\}
	\label{eq:KL_regularized}
\end{align}

We generate a random adhoc network of $N=8$ nodes (illustrated in Fig.~\ref{fig:topology_RLR}) where each node samples $M=2$ dimensional feature vectors from a Gaussian mixture with two components and probability density function $\h_{k,i} \sim \frac{1}{2} \mathcal{N}(2\mathds{1},I_M) + \frac{1}{2} \mathcal{N}(-2\mathds{1},I_M)$, where $\mathcal{N}(\gamma,\Sigma)$ denotes the multivariate Gaussian density function with mean vector $\gamma$ and covariance matrix $\Sigma$. The labels $\y_{k}(i)$ are generated according to
\begin{align}
	p(\y_{k}(i) | \h_{k,i};  w^o) = \frac{1}{1+e^{-\y_{k}(i) \h_{k,i}^\T w^o}}
	\label{eq:logistic_likelihood}
\end{align}
where $w^o$ is chosen arbitrary at the start of each experiment. Each node performs the diffusion algorithm listed in \eqref{eq:A}-\eqref{eq:C2} by computing the gradient according to the instantaneous approximation of the strongly-convex regularized KL divergence \eqref{eq:KL_regularized}. Since in this scenario, the noise profile is uniform across the agents, the optimal weights \eqref{eq:p^o} become $p^o = \frac{1}{N} \mathds{1}_N$, which leads to a doubly-stochastic combination policy, $A$. In order to simulate a scenario involving a left-stochastic combination rule, we modify the instantaneous approximations for the gradient vectors across the agents by adding a small additive noise component to these approximations at each agent with varying power profile shown in Fig. \eqref{fig:noise_variance_RLR}. These noise perturbations help model unknown effects or even round-off errors. 

The regularization constant $\rho$ is chosen to be $5$ in the simulation since it will determine $\lambda_{\min}$ and in order to ensure that our approximation in \eqref{eq:N_fold} is valid, we choose $\mu$ to be $5$. The optimizer is found by optimizing the sum of the regularized KL divergences and using empirical average for the expectation.
  
Figure~\ref{fig:simulation} shows the resulting curves, which are averaged over $500$ experiments (with a different random network topology per experiment). We plot the performance of the non-cooperative algorithm along with its theoretical performance (both averaged over the nodes). In addition, we show the performance of the \emph{centralized} algorithms \eqref{eq:full_gradient} and \eqref{eq:weighted_full_gradient} along with their theoretical performance obtained from the right-hand-side of \eqref{eq:asymptotic_centralized_result_approx} and \eqref{eq:asymptotic_weighted_centralized_result_approx}. We see that the diffusion algorithm outperforms the non-cooperative algorithm and matches the centralized algorithm listed in \eqref{eq:full_gradient} when the combination policy $A$ is doubly-stochastic, as predicted by Theorem~\ref{thm:asymptotic_term}. We also observe that when $A$ is chosen to be left-stochastic, the diffusion strategy outperforms \eqref{eq:full_gradient}, as predicted by \eqref{eq:Harmonic_mean} and approaches to same performance as the weighted centralized algorithm \eqref{eq:weighted_full_gradient}, as expected.

\section{Conclusions and Future Work}
We performed a detailed mean-square-error analysis of the asymptotic convergence behavior of distributed strategies of the consensus and diffusion types. We established that the algorithms converge asymptotically at the rate of $O(1/i)$ and, more importantly, derived the exact constant that multiplies the convergence rate factor $1/i$. By using the derived expressions, we were able to compare the performance of various implementations against each other including non-cooperative strategies, centralized strategies, and distributed consensus and diffusion strategies. The analysis revealed that diffusion implementations outperform the other strategies and deliver the best excess-risk performance asymptotically as centralized strategies; we also showed how to optimize this performance over the choice of combination weights. We showed that the optimal weights are left-stochastic rather than doubly-stochastic and derived an expression for them in terms of the Hastings rule. We also observed that the diffusion implementation leads to smaller overshoot during the transient phase in relation to the other strategies in simulations. It was also demonstrated that the asymptotic performance of the difusion strategy is independent of the particular network topology that connects the agents. The analysis in this work was done under the assumption of stationary data distributions, in which case decaying step-sizes are justified. If the statistical distribution of the data drifts with time, then constant step-sizes need to be employed. In this case, the dynamics of the learning network is modified in one important way in that the gradient noise does not die out anymore (it is annihilated by decaying step-sizes but not by constant step-sizes). The performance of the networks in the constant step-size regime is studied in \cite{NOW_ML,SayedProcIEEE,jianshu_part1,jianshu_part2}. A useful extension of this work is to study the Polyak-Ruppert averaging version \eqref{eq:PR_A}--\eqref{eq:PR_AV} for robustness against ill-conditioning and lack of strong-convexity.

\appendices

\section{Proof of Theorem \ref{thm:non_asym_convergence_main}}
\label{app:local_mse_rec}
We denote the error vectors at node $k$ at time $i$ by $\tilde{\bpsi}_{k,i} \triangleq w^o - \bpsi_{k,i}$
and $\widetilde{\w}_{k,i}    \triangleq w^o -\w_{k,i}$. We subtract \eqref{eq:A}-\eqref{eq:C2} from $w^o$ using \eqref{eq:grad_model} to get
\begin{align}
	\tilde{\bpsi}_{k,i} &= \widetilde{\w}_{k,i-1} \!+\! \mu(i) \left[\nabla_w J_k(\w_{k,i-1}) \!+\! \v_{k,i}(\w_{k,i-1})\right] \label{eq:A_err}\\
	\widetilde{\w}_{k,i}   &= \sum_{\ell=1}^N a_{\ell k} \tilde{\bpsi}_{\ell,i} \label{eq:C2_err}
\end{align}
Using the mean-value relation \eqref{eq:polyak1}, we can write 
\begin{align}
	\nabla_w J_k(\w_{k,i-1}) = -\H_{k,i}\widetilde{\w}_{k,i-1}
	\label{eq:gradient_expansion}
\end{align}
where 
\begin{equation}
	\H_{k,i} \!\triangleq\! \int_0^1 \nabla^2_w J_k(w^o\!-\!t \widetilde{\w}_{k,i-1})dt
\end{equation}
and where we used the fact that $\nabla_w J_k(w^o) = 0$ since $w^o$ optimizes $J_k(w)$. Substituting \eqref{eq:gradient_expansion} into \eqref{eq:A_err}, we get
\begin{align}
	\tilde{\bpsi}_{k,i} &= \left[I-\mu(i) \H_{k,i-1}\right] \widetilde{\w}_{k,i-1} + \mu(i) \v_{k,i}(\w_{k,i-1})
	\label{eq:A2_err}
\end{align}
We now derive mean-square-error (MSE) recursions by noting that $\|x\|^2 \triangleq x^\T x$ is a convex function of $x$. Therefore, applying Jensen's inequality \cite[p.~77]{cvx_book} to \eqref{eq:C2_err}  we obtain: 
\begin{align}
	\mathbb{E} \{\|\widetilde{\w}_{k,i}\|^2| \boldsymbol{\mathcal{F}}_{i-1} \} &\leq \sum_{\ell=1}^N a_{\ell k} \mathbb{E}\{\|\tilde{\bpsi}_{\ell,i}\|^2| \boldsymbol{\mathcal{F}}_{i-1} \}
	\label{eq:A_var_ineq}
\end{align}
for $k=1,\ldots,N$.
From \eqref{eq:A2_err} and using Assumption \ref{ass:noiseModeling}, we get
\begin{align}
	\mathbb{E} \{\|\tilde{\bpsi}_{k,i}\|^2 &| \boldsymbol{\mathcal{F}}_{i-1} \} = \mathbb{E} \{\|\widetilde{\w}_{k,i-1}\|^2_{\bSigma_{k,i}}| \boldsymbol{\mathcal{F}}_{i-1} \} +\nonumber\\
	& \mu^2(i)\cdot \E\{\left\Vert\v_{k,i}(\w_{k,i-1})\right\Vert^2| \boldsymbol{\mathcal{F}}_{i-1} \}
	\label{eq:A_var_eq} 
\end{align}
where $\bSigma_{k,i} \triangleq \left(I_M - \mu(i) \H_{k,i-1}\right)^2$.
The matrices $\bSigma_{k,i}$ can be seen to be positive semi-definite for large enough $i$ and bounded by 
\begin{align}
	0 \leq \bSigma_{k,i} \leq \gamma^2(i) I_M
	\label{eq:bSigma_bound}
\end{align}
where
\begin{align}
	\gamma(i) \triangleq \max\left\{\left|1-\mu(i) \bar{\lambda}\right|,\left|1-\mu(i) \underline{\lambda}\right|\right\}
	\label{eq:gamma}
\end{align}
Now note that $\gamma^2(i)$ can be upper-bounded by:
\begin{align}
	\gamma^2(i) &= \max\left\{1-2 \mu(i) \bar{\lambda}+\mu^2(i) \bar{\lambda}^2,1-2 \mu(i) \underline{\lambda}+\mu^2(i) \underline{\lambda}^2\right\} \nonumber\\
				   &\leq 1-2 \mu(i) \underline{\lambda} + \mu^2(i) \bar{\lambda}^2
\end{align}
In order to simplify the notation, we introduce the upper-bound
\begin{align}
	\beta(i) &\triangleq 1-2 \mu(i) \underline{\lambda} \left(1 - \frac{(\bar{\lambda}^2+\alpha_2)}{2\underline{\lambda}} \mu(i)\right)
\end{align}
where $\alpha_2$ is defined by \eqref{eq:noise_var_bound1}, namely,
\begin{align}
	\E\{\left\Vert\v_{k,i}(\w_{k,i-1})\right\Vert^2| \boldsymbol{\mathcal{F}}_{i-1} \} \leq \alpha_2 \|\widetilde{\w}_{k,i-1}\|^2 +\sigma_{v2}^2
	\label{eq:noise_var_bound}
\end{align}
Combining \eqref{eq:A_var_eq}, \eqref{eq:bSigma_bound}, and \eqref{eq:noise_var_bound}, we obtain for $k=1,\ldots,N$:
\begin{align}
	\mathbb{E} \{\|\tilde{\bpsi}_{k,i}\|^2| \boldsymbol{\mathcal{F}}_{i-1} \} &\leq \beta(i)\, \|\widetilde{\w}_{k,i-1}\|^2 + \mu^2(i) \sigma_{v2}^2
	\label{eq:C2_var_ineq}
\end{align}
We introduce the MSE vectors:
\begin{align}
	\boldsymbol{\mathcal{W}}_i &\triangleq \big[\|\widetilde{\w}_{1,i}\|^2,\dots,\|\widetilde{\w}_{N,i}\|^2\big]^\T\\
	\boldsymbol{\mathcal{Y}}_i &\triangleq \big[\|\tilde{\bpsi}_{1,i}\|^2,\dots,\|\tilde{\bpsi}_{N,i}\|^2\big]^\T
\end{align}
and rewrite \eqref{eq:A_var_ineq}, and \eqref{eq:C2_var_ineq} as:
\begin{align}
	\E\{\boldsymbol{\mathcal{Y}}_i| \boldsymbol{\mathcal{F}}_{i-1} \} &\preceq \beta(i) \boldsymbol{\mathcal{W}}_{i-1} + \mu^2(i) \sigma_{v2}^2\mathds{1}_N\\
	\E\{\boldsymbol{\mathcal{W}}_i| \boldsymbol{\mathcal{F}}_{i-1} \} &\preceq A^\T \E\{\boldsymbol{\mathcal{Y}}_{i}| \boldsymbol{\mathcal{F}}_{i-1} \}
\end{align}
where $x \preceq y$ indicates that each element of the vector $x$ is less than or equal to the corresponding element of vector $y$. Using the fact that if $x \preceq y$ then $B x \preceq B y$ for any matrix $B$ with non-negative entries, we can combine the above inequality recursions into a single recursion for $\mathcal{W}_i$:
\begin{align}
	\E\{\boldsymbol{\mathcal{W}}_i| \boldsymbol{\mathcal{F}}_{i-1}\} &\preceq \beta(i) A^\T \boldsymbol{\mathcal{W}}_{i-1} + \mu^2(i) \sigma_{v2}^2 \mathds{1}_N
				  \label{eq:bigW_recur_1}
\end{align}
Now, we multiply both sides by $p^\T$, where $p$ is the Perron eigenvector of $A$. This yields the recursion:
\begin{align}
	\E\{p^\T\boldsymbol{\mathcal{W}}_i| \boldsymbol{\mathcal{F}}_{i-1}\} &\leq \beta(i) p^\T \boldsymbol{\mathcal{W}}_{i-1} + \mu^2(i) \sigma_{v2}^2
	\label{eq:scalar_recursion}
\end{align}

\subsection{MSE Convergence}
For the first part of Theorem \ref{thm:non_asym_convergence_main} (asymptotic MSE convergence), we evaluate  the expectation of both sides of the inequality in \eqref{eq:scalar_recursion} to get 
\begin{align}
	\E\{p^\T\boldsymbol{\mathcal{W}}_i\} &\leq  \beta(i) \E\{p^\T  \boldsymbol{\mathcal{W}}_{i-1}\} \!+\! \mu^2(i) \sigma_{v2}^2
	\label{eq:scalar_recursion_unconditional}
\end{align}
Now since $\mu(i) \rightarrow 0$, we conclude that for large enough $i > i_o$, the sequence
$\mu^2(i)$ will assume smaller values than $\mu(i)$. Therefore, a large enough time
index, $i_o>0$ exists such that the following condition is satisfied:
\begin{align}
	0 < (1-\beta(i)) = (2\mu(i) \underline{\lambda} + (\bar{\lambda}^2 + \alpha_2) \mu^2(i)) < 1 \label{eq:condition_polyak_lemma}
\end{align}
Furthermore, noting that $\sum_{i=1}^\infty (1\!-\!\beta(i)) = \infty$,
$\lim_{i\rightarrow\infty} \mu^2(i)\sigma_{v2}^2/(1-\beta(i)) = 0$, we then invoke Lemma \ref{lem:Polyak} (Appendix \ref{app:lemmas}) to arrive at the desired result.

\subsection{Almost-Sure Convergence}
We see that \eqref{eq:scalar_recursion} fits the form of Lemma \ref{lem:Gladyshev} (Appendix \ref{app:lemmas}), where \eqref{eq:condition_polyak_lemma} will hold for large enough $i$, $\sum_{i=1}^\infty (1\!-\!\beta(i)) = \infty$, $\lim_{i\rightarrow\infty} \mu^2(i)\sigma_{v2}^2/(1-\beta(i)) = 0$, and $\sum_{i=1}^\infty \mu^2(i)\sigma_{v2}^2 < \infty$, and we conclude that $p^\T \boldsymbol{\mathcal{W}}_i \rightarrow 0$ almost surely so $\boldsymbol{\mathcal{W}}_i \rightarrow 0$ almost surely as well since all the entries of $p$ are strictly positive. This implies that $\w_{k,i} \rightarrow w^o$ almost surely for all $k=1,\ldots,N$.

\subsection{MSE Convergence Rate}
\label{app:MSE_convergence_rate}
For the convergence rate, we take the expectation of 
\begin{align}
	\E\{p^\T\boldsymbol{\mathcal{W}}_i| \boldsymbol{\mathcal{F}}_{i-1}\} &\leq \beta(i) p^\T \boldsymbol{\mathcal{W}}_{i-1} + \mu^2(i) \sigma_{v2}^2
	\label{eq:scalar_recursion_beta}
\end{align}
to get
\begin{align}
\E\{p^\T\boldsymbol{\mathcal{W}}_i\} &\leq   \left(1-\frac{a}{i}+\frac{b}{i^2}\right) \E\{p^\T  \boldsymbol{\mathcal{W}}_{i-1}\} + \frac{c}{i^2}
\end{align}
where $a \triangleq 2\mu\underline{\lambda}$, $b \triangleq \mu^2(\bar{\lambda}+\alpha_2)$, and $c \triangleq \sigma_{v2}^2 \mu^2$. Using Lemma \ref{lem:convergence_deterministic_sequence} (Appendix \ref{app:lemmas}), we have that when $a = 2\underline{\lambda} \mu > 1$,
\begin{align}
	\limsup_{i\rightarrow \infty} \frac{\E\{p^\T\boldsymbol{\mathcal{W}}_i\}}{i^{-1}} \leq \frac{c}{a-1}, \label{eq:rate_mse}
\end{align}
which in turn implies that each $\E \|\w_{k,i}\|^2$ diminishes at the rate $i^{-1}$, which is our desired result.

\subsection{Fourth-Order-Moment Convergence Rate}
Observe that by Jensen's inequality, we have that
\begin{align}
	\E \|\widetilde{\w}_{k,i}\|^4 \leq \sum_{\ell = 1}^N a_{\ell k} \E \|\tilde{\bpsi}_{\ell,i}\|^4
\end{align}
and
\begin{align}
	\|\tilde{\bpsi}_{k,i}\|^4 = \|(I_M \!-\! \mu(i) \H_{k,i-1}) \widetilde{\w}_{k,i-1} \!+\! \mu(i) \v_{k,i}(\w_{k,i-1})\|^4
\end{align}
By utilizing Lemma 5 from \cite[p.~32]{Xiaochuan_PartI}, we obtain:
\begin{align}
	\E\{\|&\tilde{\bpsi}_{k,i}\|^4 | \boldsymbol{\mathcal{F}}_{i-1}\} \leq \theta_1(i) \|\widetilde{\w}_{k,i-1}\|^4 + \nonumber\\
	&\quad\ 3 \mu^4(i) \E \{\|\v_{k,i}(\w_{k,i-1})\|^4| \boldsymbol{\mathcal{F}}_{i-1}\} + \nonumber\\
	&\quad\ 8 \theta_2(i) \|\widetilde{\w}_{k,i-1}\|^2 \cdot \E \{\|\v_{k,i}(\w_{k,i-1})\|^2 | \boldsymbol{\mathcal{F}}_{i-1}\} \label{eq:noise_fourth_xiaochuan}
\end{align}
where
\begin{align}
	\theta_1(i) &\triangleq 1 - 4\mu(i)\underline{\lambda} + 2 \mu^2(i) (2\underline{\lambda}^2 + \bar{\lambda}^2) + \mu^4(i) \bar{\lambda}^4\\
	\theta_2(i) &\triangleq \beta(i) \cdot \mu^2(i)
\end{align}
Now, using \eqref{eq:noise_fourth_bound}, we have that
\begin{align}
	&\E \{\|\v_{k,i}(\w_{k,i-1})\|^4| \boldsymbol{\mathcal{F}}_{i-1}\} \leq \alpha_4 \|\widetilde{\w}_{k,i-1}\|^4 + \sigma_{v4}^4 \label{eq:noise_fourth_app}
\end{align}
and using \eqref{eq:noise_var_bound1} and \eqref{eq:noise_fourth_app} in \eqref{eq:noise_fourth_xiaochuan}, we obtain
\begin{align}
	\E\{\|\tilde{\bpsi}_{k,i}\|^4 | \boldsymbol{\mathcal{F}}_{i-1}\} &\leq \theta_1(i) \|\widetilde{\w}_{k,i-1}\|^4 + \nonumber\\
	&\quad\ 3 \mu(i)^4 (\alpha_4 \|\widetilde{\w}_{k,i-1}\|^4 + \sigma_{v4}^4) +\nonumber\\
	&\quad\ 8 \theta_2(i) \|\widetilde{\w}_{k,i-1}\|^2 \cdot (\alpha_2 \|\widetilde{\w}_{k,i-1}\|^2 \!+\! \sigma_{v2}^2)
\end{align}
or, equivalently,
\begin{align}
	\E\{\|\tilde{\bpsi}_{k,i}\|^4 | \boldsymbol{\mathcal{F}}_{i-1}\} &\leq \left(1-\frac{a}{i} + O(i^{-2})\right) \|\widetilde{\w}_{k,i-1}\|^4 +  \nonumber\\
	&\quad\ O(i^{-2}) \|\widetilde{\w}_{k,i-1}\|^2 \!+\! O(i^{-4})
\end{align}
where $a \triangleq 4\underline{\lambda}\mu$. Taking the unconditional expectation of $\E\{\|\tilde{\bpsi}_{k,i}\|^4 | \boldsymbol{\mathcal{F}}_{i-1}\}$, we obtain,
\begin{align}
	\E \|\tilde{\bpsi}_{k,i}\|^4 &\leq \left(1-\frac{a}{i} + O(i^{-2})\right) \E \|\widetilde{\w}_{k,i-1}\|^4 + \nonumber\\
	&\quad\ O(i^{-2}) \E\|\widetilde{\w}_{k,i-1}\|^2 \!+\! O(i^{-4})
\end{align}
Now, observing that $\E\|\widetilde{\w}_{k,i-1}\|^2 \sim O(i^{-1})$ according to \eqref{eq:rate_mse}, we have that
\begin{align}
	\E \|\tilde{\bpsi}_{k,i}\|^4 &\leq \left(1\!-\!\frac{a}{i} \!+\! O(i^{-2})\right) \E \|\widetilde{\w}_{k,i-1}\|^4 + O(i^{-3})	
\end{align}
We can now form the network variables
\begin{align}
	\mathcal{W}^4_i &\triangleq \big[\E\|\widetilde{\w}_{1,i}\|^4,\dots,\E\| \widetilde{\w}_{N,i}\|^4\big]^\T \label{eq:W_4}
\end{align}
so that
\begin{align}
	\mathcal{W}^4_i \preceq \left(1-\frac{a}{i} + O(i^{-2})\right) A^\T \mathcal{W}^4_{i-1} + O(i^{-3})
\end{align}
and taking the maximum-norm, we obtain:
\begin{align}
	\|\mathcal{W}^4_i\|_{\infty} \leq \left(1-\frac{a}{i} + O(i^{-2})\right) \| \mathcal{W}^4_{i-1} \|_{\infty} + \frac{e}{i^3}
\end{align}
where $e$ is a constant independent of $i$. Using Lemma \ref{lem:convergence_deterministic_sequence} (Appendix \ref{app:lemmas}), we have that when $a = 4\underline{\lambda}\mu > 2$:
\begin{align}
	\limsup_{i\rightarrow \infty} \frac{\|\mathcal{W}^4_i\|_{\infty}}{i^{-2}} \leq \frac{e}{a-2}, \label{eq:rate_fourth_order}
\end{align}
and therefore $\E\|\widetilde{\w}_{k,i}\|^4 = O(i^{-2})$ when $2\underline{\lambda}\mu > 1$.

\section{Proof of Theorem \ref{thm:asymptotic_term}}
\label{sec:Proof_Asymptotic}

First, we establish an equality regarding the relationship between the excess-risk and weighted means-square-error quantities:

\begin{lemma}[\textbf{Time-evolution of excess-risk}]
	\label{lem:ER}
	Let Assumptions \ref{ass:HessianAssumption}--\ref{ass:noiseModeling} hold. Then, the excess-risk at agent $k$ is given by
	\begin{align}
		\textrm{ER}_k(i) \!=\! \textrm{AT}(i) \!+\! \textrm{TT}(i) \!+\! \textrm{HO}_{1,k}(i) \!+\! \textrm{HO}_{2}(i) \!+\! \textrm{HO}_{3}(i) \label{eq:unrolled_recursion}
	\end{align}
	where
	\begin{alignat}{2}
\textrm{AT}(i) &\triangleq \sum_{j=1}^{i-1} \!\mu^2(j) \Tr\!\left(\G_i \Omega_{j,i}\right),&&\!\!\!\!\!\!\!\!\!\!\!\!\!\!\!\!\!\!\!\!\!\!\!\!\!\!\!\!\!\!\!\!\!\!\!\!\!\!\!\![\textrm{Asymptotic Term}] \label{eq:def_AT}\\
\textrm{TT}(i) &\triangleq \E \|\widetilde{\sw}_{0}\|^2_{\Omega_{0,i}},&&\!\!\!\!\!\!\!\!\!\!\!\!\!\!\!\!\!\!\!\!\!\!\!\!\!\!\!\!\!\!\!\!\!\!\!\!\!\!\!\![\textrm{Transient Term}] \label{eq:def_TT}\\
\textrm{HO}_{1,k}(i) &\triangleq \E \widetilde{\w}_{k,i-1}^\T \widetilde{\boldsymbol{S}}_{k,i}\widetilde{\w}_{k,i-1},\,&&\!\!\!\!\!\!\!\!\!\!\!\!\!\!\!\!\!\!\!\!\!\!\!\!\!\!\!\!\!\!\!\!\!\!\!\!\!\!\!\![\textrm{High-order\ Term}] \label{eq:def_HO1}\\
\textrm{HO}_{2}(i) &\triangleq \sum_{j=1}^{i-1} \Tr\!\left(\mathcal{R}_{d,j} \Omega_{j,i}\right),&&\!\!\!\!\!\!\!\!\!\!\!\!\!\!\!\!\!\!\!\!\!\!\!\!\!\!\!\!\!\!\!\!\!\!\!\!\!\!\!\![\textrm{High-order Term}]  \label{eq:def_HO2}\\
\textrm{HO}_{3}(i) &\triangleq 2 \!\sum_{j=1}^{i-1} \Tr\left(\Omega_{j,i} \E\{\B_{j} \widetilde{\sw}_{j-1} \d_{j}^\T\}\right),&&\!\!\!\!\!\!\!\!\!\!\!\!\!\!\!\!\!\![\textrm{High-order Term}]  \label{eq:def_HO3}\\
\Omega_{j,i} &\triangleq \left(\prod_{t=j+1}^{i-1} \B_t^\T\right) \!\!\Sigma \!\left(\prod_{t=j+1}^{i-1} \B_t^\T\right)^{\!\!\!\!\!\T} &&  \label{eq:def_Gamma}\\
\Sigma &\triangleq E_{kk} \otimes \frac{1}{2} H && \label{eq:sigma}\\	
	\widetilde{\bm{S}}_{k,i} &\triangleq {\bm{S}}_{k,i} - \frac{1}{2} H &&
	\label{eq:S_tilde}\\
	\mathcal{A} &\triangleq A \otimes I_M &&\\
	\G_i &\triangleq \mathcal{A}^\T \E\{\mathcal{R}_{g,i}(\sw_{i-1})\} \mathcal{A} &&\label{eq:G_i}\\
	\D &\triangleq I_N \otimes H &&\label{eq:deterministic_F}\\
	\mathcal{B}_i &\triangleq \mathcal{A}^\T (I\!-\!\mu(i) \D) &&\label{eq:B_i}\\
	\boldsymbol{\D}_{i-1} \!&\triangleq\! \underset{1\leq k\leq N}{\mathrm{diag}} \left\{\int_0^1 \nabla^2_w J_k(w^o-t \widetilde{\w}_{k,i-1}) dt\right\} &&\label{eq:bold_script_H}\\
	\d_i \!&\triangleq\! \mu(i) \mathcal{A}^\T (\D - \boldsymbol{\D}_{i-1}) \widetilde{\sw}_{i-1} &&\\
	\mathcal{R}_{d,i} &\triangleq \E \{\d_i\d_i^\T\}&&
\end{alignat}
Moreover, the products that appear in expression \eqref{eq:def_Gamma} for $\Omega_{j,i}$ are performed in the following order:
\begin{align}
	\Omega_{j,i} = \B_{j+1}^\T\cdots\B_{i-1}^\T\Sigma\B_{i-1}\cdots\B_{j+1}
\end{align}
\end{lemma}
\begin{proof}
	First, observe from \eqref{eq:def_ER_k} that:
\begin{align}
	\textrm{ER}_k(i) &= \E \|\widetilde{\sw}_{i-1}\|^2_{E_{kk}\otimes \bm{S}_{k,i}} \nonumber\\
	&= \E \widetilde{\w}_{k,i-1}^\T \bm{S}_{k,i} \widetilde{\w}_{k,i-1} \nonumber\\
	&= \E\widetilde{\w}_{k,i-1}^\T\!\left(\!\frac{1}{2}H \!+\! \widetilde{\boldsymbol{S}}_{k,i}\!\right) \!\!\widetilde{\w}_{k,i-1} \nonumber\\
	&= \E \|\widetilde{\sw}_{i-1}\|^2_{\Sigma} \!+\! \mathrm{HO}_{1,k}(i) \label{eq:w_tilde_HO1}
\end{align}
Now to study $\E \|\widetilde{\sw}_{i-1}\|^2_{\Sigma}$, we first  verify from \eqref{eq:A}--\eqref{eq:C2} that the error quantity $\widetilde{\sw}_i$ evolves according to the following dynamics:
\begin{align}
	\widetilde{\sw}_{i-1} &= \mathcal{B}_{i-1}\widetilde{\sw}_{i-2} \!+\! \mu(i-1) \mathcal{A}^\T \g_{i-1}(\w_{i-2}) \!+\! \d_{i-1} \label{eq:error_recursion1}
\end{align}
Equating the energies of both sides of \eqref{eq:error_recursion1} and taking expectations, we arrive at the following recursion, which relates weighted variances of two successive network error vectors, $\widetilde{\sw}_{i-1}$ and $\widetilde{\sw}_{i-2}$: 
\begin{align}
	\E \|\widetilde{\sw}_{i-1}\|_\Sigma^2 &= \E \left\{\|\widetilde{\sw}_{i-2}\|_{{\B}_{i-1}^\T \Sigma {\B}_{i-1}}^2\right\} + \nonumber\\
	&\quad\  \mu^2(i-1) \Tr(\G_{i-1} \Sigma) + \E \|\d_{i-1}\|^2_\Sigma + \nonumber\\
	&\quad\ 2 \E\{\d_{i-1}^\T \Sigma {\B}_{i-1} \widetilde{\sw}_{i-2}\}
	\label{eq:approx_MSE_recursion}
\end{align}
where 
\begin{align}
	\G_i \triangleq \mathcal{A}^\T \E\{\mathcal{R}_{g,i}(\sw_{i-1})\} \mathcal{A}
\end{align}
If we now iterate the recursion \eqref{eq:approx_MSE_recursion}, we obtain that 
\begin{align}
	\E \|\widetilde{\sw}_{i-1}\!\|_\Sigma^2 &= \textrm{AT}(i) + \textrm{TT}(i) + \textrm{HO}_{2}(i) + \textrm{HO}_{3}(i) \label{eq:WMSE}
\end{align}
where $\textrm{AT}(i)$, $\textrm{TT}(i)$, $\textrm{HO}_{2}(i)$, and $\textrm{HO}_{3}(i)$ were defined in \eqref{eq:def_AT}, \eqref{eq:def_TT}, \eqref{eq:def_HO2}, and \eqref{eq:def_HO3}, respectively. Then, combining \eqref{eq:w_tilde_HO1} and \eqref{eq:WMSE}, we obtain the following representation for  the excess-risk:
\begin{equation}
	\textrm{ER}_k(i) = \textrm{AT}(i) + \textrm{TT}(i) + \textrm{HO}_{1,k}(i) + \textrm{HO}_{2}(i) + \textrm{HO}_{3}(i)  
\end{equation}
 where $\mathrm{HO}_{1,k}(i)$ was defined in \eqref{eq:def_HO1}.
\end{proof}
The first term on the right-hand side of \eqref{eq:unrolled_recursion} models the asymptotic behavior of the network since, as we will see further ahead, it will decay at the slowest rate in comparison to all other terms in \eqref{eq:unrolled_recursion}. In comparison, the second term on the right-hand side of \eqref{eq:unrolled_recursion} models the transient behavior of the diffusion network since it is governed by the initial error in estimating $w^o$ at the different nodes. The remaining terms will decay faster than the first two terms, and thus will be considered ``high-order'' terms. It is necessary to study the behavior of all terms in \eqref{eq:unrolled_recursion} to understand the dynamics of the network. 

First, we will establish the convergence rate of the so-called ``asymptotic-term'' $\textrm{AT}(i)$:
\begin{theorem}[\textbf{Convergence rate of asymptotic term $\AT(i)$}]
	\label{thm:asymptotic_term_sub}
	Let Assumptions \ref{ass:HessianAssumption}-\ref{ass:noiseModeling} hold. It then holds asymptotically that
	\begin{align}
\AT(i) \sim \frac{\mu^2}{2} p^\T \left(\sum_{m=1}^M \lambda_m \alpha_m(i)  R_{v,m}\right) p
\label{eq:asymptotic_theorem_result_sub}
\end{align}
where the notation $f(i) \!\sim\! g(i)$ implies that $\displaystyle{\lim_{i\rightarrow\infty}} f(i)/g(i) \!=\! 1$ and $\alpha_m(i)$ is defined as
\begin{align}
	\alpha_m(i) \!\!\triangleq\!\! \begin{cases}
																										\frac{1}{2\lambda_m \mu - 1}\cdot \frac{1}{i} + \Theta\left(\frac{1}{i^{2\lambda_m\mu}}\right), & 2\lambda_{m}\mu \neq 1\\
																										\log(i)/i, &2\lambda_{m}\mu = 1
																										\end{cases} \label{eq:alpha_m_AT}
\end{align}
Moreover, $\lambda_m$ is the $m$-th eigenvalue of $H$ and the $N\times N$ matrix $R_{v,m}$ is defined as:
\begin{align}
	R_{v,m} \triangleq \mathrm{diag}\{(\Phi^\T R_{v,1}^o \Phi)_{mm}, \ldots, (\Phi^\T R_{v,N}^o \Phi)_{mm}\}
\end{align}
where the notation $(X)_{mm}$ denotes the $m$-th diagonal element of the matrix $X$. Specifically, when $2\lambda_{\min}\mu > 1$, we have that:
\begin{align}
	\AT(i) \sim \frac{\mu}{2 i} \sum_{m=1}^M \frac{\lambda_m \mu}{2\lambda_m \mu - 1}  \cdot p^\T R_{v,m} p
\end{align}
\end{theorem}
\begin{proof}
	Observe that the covariance matrix $\E \{\mathcal{R}_{g,i}(\sw_{i-1})\}$ converges to a constant matrix when $2\underline{\lambda}\mu > 1$:
	\begin{equation}
		\E \{\!\mathcal{R}_{g,i}(\sw_{i-1})\!\} \!=\! \E\{\mathcal{R}_{g,i}(\mathds{1}_N\!\otimes\! w^o)\} + O\!\left(\!\frac{1}{i^{\min(1,\gamma/2)}}\!\right)
	\end{equation}
	where $0 < \gamma \leq 4$ is from \eqref{eq:noise_covar_lipschitz_local}. To see this, consider the quantity:
	\begin{align}
	\|\E\mathcal{R}_{g,i}(\sw_{i-1}) &-\E\{\mathcal{R}_{g,i}(\mathds{1}_N\otimes w^o)\}\| \nonumber\\
	&\stackrel{(a)}{\leq} \E \|\mathcal{R}_{g,i}(\sw_{i-1}) \!-\! \mathcal{R}_{g,i}(\mathds{1}_N\!\otimes\! w^o)\| \nonumber\\
														  &\stackrel{(b)}{\leq} \kappa_1 \cdot \E \|\widetilde{\sw}_{i-1}\|^2 + \kappa_2 \cdot \E \|\widetilde{\sw}_{i-1}\|^\gamma\nonumber\\
														  &\stackrel{(c)}{\leq} \kappa_1 \cdot \E \|\widetilde{\sw}_{i-1}\|^2 + \kappa_2 \cdot (\E \|\widetilde{\sw}_{i-1}\|^2)^\frac{\gamma}{2} \nonumber\\
														  &\stackrel{(d)}{\sim} O(1/i) + O(1/i^{\gamma/2})
	\end{align}
	where steps $(a)$ and $(c)$ are due to Jensen's inequality, step $(b)$ is due to \eqref{eq:noise_covar_lipschitz}, and step $(d)$ is due to \eqref{eq:MSD_ind_quadratic}--\eqref{eq:fourth_ind} when $2\underline{\lambda} \mu > 1$.	We conclude, therefore, that:
\begin{align}
	\lim_{i\rightarrow\infty} \E \{\mathcal{R}_{g,i}(\sw_{i-1})\} = \mathcal{R}_{g}^o
\end{align} 
where $\mathcal{R}_g^o$ is defined in \eqref{eq:noise_covariance}. Combining with \eqref{eq:G_i}, we obtain
\begin{align}
	\lim_{i\rightarrow\infty} \G_i = \G \triangleq \mathcal{A}^\T \mathcal{R}_g^o \mathcal{A}
\end{align}
where, as pointed out in \eqref{eq:R_v}, $\mathcal{R}_g^o$ is block-diagonal. Since we are interested in the asymptotic performance of the term $\AT(i)$, we may replace $\G_i$ with $\G$ in the following analysis. Now, with $\Sigma = \frac{1}{2} E_{kk} \!\otimes\! H$, we rewrite $\AT(i)$ as 
\begin{align}
&\AT(i) = \sum_{j=1}^{i-1} \mu^2(j) \Tr(\mathcal{R}_g^o \Omega_{j,i})
\label{eq:asym_Tr_00}
\end{align}
We observe that \eqref{eq:asym_Tr_00} is in the exact form required by Lemma \ref{lem:series_of_traces} (Appendix \ref{app:lemmas}), where we make the identifications:
\begin{align}
\mathcal{L} \leftarrow \mathcal{R}_g^o,\quad\quad b \leftarrow 0,\quad\quad q \leftarrow 1
\end{align}
where $\mathcal{R}_g^o$ is a block-diagonal matrix with block-matrices $\{R_{1},\ldots,R_{N}\}$. We thus obtain \eqref{eq:asymptotic_theorem_result_sub}.
\end{proof}

Next, we consider the remaining terms of \eqref{eq:unrolled_recursion}, beginning with $\textrm{TT}(i)$:
\begin{theorem}[\textbf{Convergence rate of transient term }$\textrm{TT}(i)$]
\label{thm:transient_term}
	Let Assumptions \ref{ass:HessianAssumption}-\ref{ass:noiseModeling} hold. The transient term in \eqref{eq:unrolled_recursion} decays asymptotically according to
	\begin{align} 
	\textrm{TT}(i) = \Theta\left(1/i^{2\lambda_{\min}\mu}\right),\quad\quad \textrm{as\ } i\rightarrow\infty
\end{align}
\end{theorem}
\begin{proof}
See Appendix \ref{sec:Proof_Transient}.
\end{proof}
Actually, in App. \ref{sec:Proof_Transient}, we derive upper and lower bounds for the transient term  (see \eqref{eq:transient_lower_bound} and \eqref{eq:transient_upper_bound}). Examining these bounds, we notice that the transient term will grow initially before it starts to decay. The asymptotic rate of decay of the transient error is on the order of $1/i^{2\lambda_{\min} \mu}$, for any value of $2\lambda_{\min}\mu$. 

We now establish the convergence rate of the remaining high-order terms: $\textrm{HO}_{1,k}(i)$, $\textrm{HO}_2(i)$, and $\textrm{HO}_3(i)$ in the case where $2\lambda_{\min}\mu > 1$:
\begin{theorem}[\textbf{Convergence rate of $\textrm{HO}_{1,k}(i)$}]
\label{thm:HO1}
Let Assumptions \ref{ass:HessianAssumption}--\ref{ass:noiseModeling} hold. It then holds asymptotically that
\begin{align}
\mathrm{HO}_{1,k}(i) = O\left(1/i^{3/2}\right)
\end{align}
when $2\underline{\lambda}\mu > 1$.
\end{theorem}
\begin{proof}
We have from \eqref{eq:S_tilde} that 
\begin{align}
	\|\tilde{\boldsymbol{S}}_{k,i}\| &\leq \int_0^1 \!\!t\!\! \int_0^1\!\! \|\nabla_w^2 J(w^o - s \ t \ \widetilde{\w}_{k,i-1}) - H\|ds dt \nonumber\\
	&\stackrel{(a)}{\leq} \tau_k \|\widetilde{\w}_{k,i-1}\| \int_0^1 t^2 \int_0^1 s\  ds\ dt \nonumber\\
	&= \frac{\tau_k}{6} \|\widetilde{\w}_{k,i-1}\|
\end{align}
where step $(a)$ is due to the global Lipschitz condition \eqref{eq:global_Lipschitz_Hessian}. Therefore,
\begin{align}
	\mathrm{HO}_1(i) &\leq \frac{\tau_k}{6}\E \|\widetilde{\w}_{k,i-1}\|^3 \nonumber\\
	&\stackrel{(b)}{\leq} \frac{\tau_k}{6} \left(\E \|\widetilde{\w}_{k,i-1}\|^4\right)^{3/4} \nonumber\\
	&\stackrel{(c)}{\sim} O(i^{-3/2}),\quad\quad\quad [2\underline{\lambda}\mu > 1]
\end{align}
where step $(b)$ is due to Jensen's inequality, which states that $\E f(\x) \leq f(\E \x)$ for a \emph{concave} function $f(x)$:
\begin{align}
	\E \|\widetilde{\w}_{k,i-1}\|^3 = \E (\|\widetilde{\w}_{k,i-1}\|^4)^{3/4} \leq (\E \|\widetilde{\w}_{k,i-1}\|^4)^{3/4}
\end{align}
and step $(c)$ is due to \eqref{eq:fourth_ind} when $2 \underline{\lambda} \mu > 1$. 
\end{proof}
Observe that Theorem \ref{thm:HO1} establishes that $\textrm{HO}_1(i) = o(1/i)$. This is an important result since the asymptotic term converges at the rate $1/i$ when $2\lambda_{\min}\mu > 1$.
Next, we establish the following lemma for the convergence rate of the second high-order term: $\textrm{HO}_2(i)$:
\begin{theorem}[\textbf{Convergence rate of $\textrm{HO}_{2}(i)$}]
	\label{thm:HO2}
	Let Assumptions \ref{ass:HessianAssumption}-\ref{ass:noiseModeling} hold. It then holds asymptotically that
	\begin{align}
		\textrm{HO}_{2}(i) \sim \begin{cases}
					\Theta(1/i^3) + \Theta(1/i^{2\lambda_{\min} \mu}),& 1 < 2\lambda_{\min}\mu \neq 3\\
					\Theta\left(\frac{\log(i)}{i^3}\right), & 2\lambda_{\min}\mu = 3
   		 \end{cases}
	\end{align}
	when $2\lambda_{\min}\mu > 1$.
\end{theorem}
\begin{proof}
	See Appendix \ref{sec:Proof_HO2}.
\end{proof}
Thus, we have that Theorem \ref{thm:HO2} demonstrates that the second high-order term $\textrm{HO}_2(i) = o(1/i)$, which is again faster than the asymptotic term $\textrm{AT}(i)$. Finally, we examine the final high-order term, $\textrm{HO}_3(i)$:
\begin{theorem}[\textbf{Convergence rate of $\textrm{HO}_{3}(i)$}]
	\label{thm:HO3}
	Let Assumptions \ref{ass:HessianAssumption}-\ref{ass:noiseModeling} hold. It holds asymptotically that
	\begin{equation}
		\textrm{HO}_{3}(i) \sim \begin{cases}
					\Theta(1/i^{3/2}) + \Theta(1/i^{2\lambda_{\min} \mu}),& 1 < 2\lambda_{\min}\mu \neq \frac{3}{2}\\
					\Theta\left(\frac{\log(i)}{i^{3/2}}\right), & 2\lambda_{\min}\mu = \frac{3}{2}
   		 \end{cases}
	\end{equation}
	when $2\lambda_{\min}\mu > 1$.
\end{theorem}
\begin{proof}
	See Appendix \ref{sec:Proof_HO3}.
\end{proof}
We observe from Theorems \ref{thm:transient_term}--\ref{thm:HO3} that $TT(i)$, $\mathrm{HO}_{1,k}(i)$, $\mathrm{HO}_{2}(i)$, and $\mathrm{HO}_{3}(i)$ are higher-order terms in comparison to $\textrm{AT}(i)$ when $2\lambda_{\min}\mu > 1$ since they decay at a rate that is faster than that of $\AT(i)$, $\Theta(1/i)$. Notice that while \eqref{eq:MSD_ind_quadratic}--\eqref{eq:fourth_ind} required $2\underline{\lambda}\mu > 1$, Theorem \ref{thm:asymptotic_term_sub}, which establishes the same asymptotic rate, requires $2\lambda_{\min}\mu > 1$. However, due to \eqref{eq:lambda_min>=bar_lambda}, we know that:
\begin{align}
	2\underline{\lambda} \mu > 1 \Longrightarrow 2 \lambda_{\min} \mu > 1
\end{align}
and thus as long as the condition of \eqref{eq:MSD_ind_quadratic}--\eqref{eq:fourth_ind} is satisfied, the condition on $2\lambda_{\min} \mu > 1$ will automatically be satisfied. We conclude from \eqref{eq:unrolled_recursion} and Theorems \ref{thm:asymptotic_term_sub}--\ref{thm:HO3} that:
\begin{align}
	\textrm{ER}_k(i) \sim \frac{\mu}{2 i} \sum_{m=1}^M \frac{\lambda_m \mu}{2\lambda_m \mu - 1}  \cdot p^\T R_{v,m} p
\end{align}
when $2\lambda_{\min}\mu > 1$, which is our desired result. It is important to observe that it is possible for all terms $\textrm{TT}(i)$, $\textrm{HO}_{1,k}(i)$, $\textrm{HO}_2(i)$, and $\textrm{HO}_3(i)$ to depend on the network topology, just as the asymptotic term $\textrm{AT}(i)$ depends on it through the Perron vector $p$. However, since all of the former terms decay \emph{faster} than $\textrm{AT}(i)$, they do not contribute to the asymptotic excess-risk curve and thus their dependency on the network topology was not examined.

\section{Proof of Theorem \ref{thm:transient_term}}
\label{sec:Proof_Transient}
Introduce the Jordan canonical form of matrix $A$:
	\begin{align}
		A = T D T^{-1} \label{eq:A_jordan2}
	\end{align}
	The relation between the quantities $\{T,D\}$ and $\{Y,D_{N-1},R\}$ in \eqref{eq:A_factorization_split} is as follows:
\begin{align}
T = \left[\!\!\begin{array}{cc}
p & Y
\end{array}\!\!\right], \  D = \left[\!\!\begin{array}{cc}
1\\
 & D_{N-1}
\end{array}\!\!\right], \  T^{-1} &= \left[\!\!\begin{array}{c}
\mathds{1}_N^\T \\
R^\T
\end{array}\!\!\right]
\label{eq:decomposition_D}
\end{align}
Substituting  \eqref{eq:sigma}, \eqref{eq:B_i}, and \eqref{eq:A_jordan2} into \eqref{eq:def_TT}, we observe that the weighting matrix $\Omega_{0,i}$  can be written as:
\begin{align}
	\Omega_{0,i} &= \B_1^\T\cdots\B_{i-1}^\T\Sigma\B_{i-1}\cdots\B_1 \nonumber\\
	&= \frac{1}{2} T D^{i-1} T^{-1}E_{kk}T^{-\T} {D^\T}^{i-1} T^\T \otimes \Phi K_{0,i} \Phi^\T \label{eq:prod_F}
\end{align}
where $K_{j,i}$ is the following diagonal matrix:
\begin{equation}
	K_{j,i} \triangleq \left(\!\prod_{t=j+1}^{i-1}\!(I_M \!-\! \mu(t) \Lambda)\!\right) \!\Lambda\! \left(\!\prod_{t=j+1}^{i-1}\!(I_M \!-\! \mu(t) \Lambda)\!\right) \label{eq:K_i}
\end{equation}	
Now, using \eqref{eq:decomposition_D}, we conclude that
\begin{align}
	(T D T^{-1})^{i-1}  &= T D^{i-1} T^{-1} \nonumber\\
	&= \left[\begin{array}{cc}
						p & Y
						\end{array} \right] \left[\begin{array}{cc}
						1 &  \\ 
						 & D_{N-1}^{i-1}
						\end{array} \right] \left[\begin{array}{c}
						\mathds{1}_N^\T \\ 
						R^\T
						\end{array} \right] \nonumber\\
						&= p \mathds{1}^\T + Y D_{N-1}^{i-1} R^\T \label{eq:TDT^{i-1}}
\end{align}
so that
\begin{align}
	T &D^{i-1} T^{-1}E_{kk}T^{-\T} {D^\T}^{i-1} T^\T	\nonumber\\
	&\stackrel{(a)}{=} (p \mathds{1}^\T + Y D_{N-1}^{i-1} R^\T) E_{kk} (\mathds{1} p^\T + R {D_{N-1}^{\T {i-1}}} Y^\T) \nonumber\\
	&= p \mathds{1}^\T E_{kk}\mathds{1} p^\T + Y D_{N-1}^{i-1} R^\T E_{kk} \mathds{1} p^\T + \nonumber\\
	&\quad\ p \mathds{1}^\T E_{kk} R {D_{N-1}^{\T {i-1}}} Y^\T \!\!+\!\! Y D_{N-1}^{i-1} R^\T E_{kk} R {D_{N-1}^{\T {i-1}}} Y^\T \nonumber\\
	&\stackrel{(b)}{=} p p^\T + Y D_{N-1}^{i-1} R^\T E_{kk} \mathds{1} p^\T + p \mathds{1}^\T E_{kk} R {D_{N-1}^{\T {i-1}}} Y^\T +\nonumber\\
	&\quad\  Y D_{N-1}^{i-1} R^\T E_{kk} R {D_{N-1}^{\T {i-1}}} Y^\T  \label{eq:Transient_jordan_form}
\end{align}
where step $(a)$ is due to \eqref{eq:TDT^{i-1}} and we used the fact that $\mathds{1}_N^\T E_{kk} \mathds{1}_N = 1$ since $E_{kk}$ contains a single unit element at the $(k,k)$ entry in step $(b)$. The first term of \eqref{eq:Transient_jordan_form} is a constant and does not vary with $i$. On the other hand, all other terms vary with $i$. Using Lemma \ref{lem:matrix_power} (Appendix \ref{app:lemmas}), we can now see that all the time-varying terms in \eqref{eq:Transient_jordan_form} decay to zero at least at an exponential rate:
\begin{align}
	\|Y D_{N-1}^{i-1} R^\T\! E_{kk} \mathds{1} p^\T\!\| &\!\leq\! \|Y D_{N-1}^{i-1} R^\T \| \!\cdot\! \|\! E_{kk} \mathds{1}p^\T\!\|\\
	\|p \mathds{1}^\T E_{kk} R {D^\T}_{N-1}^{i-1} Y^\T\| &\!=\! \|\!Y D_{N-1}^{i-1} R^\T E_{kk} \mathds{1} p^\T\!\| \\
	\|Y \!D_{N-1}^{i-1} R^\T \!E_{kk} R {D^\T}_{N-1}^{i-1} Y^\T\!\| &\!\leq\! \|E_{kk}\| \!\cdot\! \|Y D_{N-1}^{i-1} R^\T \|
\end{align}
where, from Lemma \ref{lem:matrix_power},
\begin{align}
	\|Y D_{N-1}^{i-1} R^\T \| &\leq c \cdot \|T\| \!\cdot \|T^{-1}\| \cdot \left(\frac{\rho(D_{N-1})+1}{2}\right)^{i-1}
\end{align}
And since the spectral radius of $D_{N-1}$ is strictly less than $1$, we find that all terms in \eqref{eq:Transient_jordan_form}, with the exception of $p p^\T$, will decay to zero at an exponential rate. We can ignore the decaying terms since the convergence rate will be dominated by a slower term that decays at the rate $i^{-2\lambda_{\min}\mu}$ in $K_{0,i}$.
To see this, we first write down the transient term as:
\begin{align}
	\E \|\widetilde{\sw}_{0}\|^2_{\Omega_{0,i}} 
	&\sim  \frac{1}{2} \E\left\{\widetilde{\sw}_0^\T \left(p p^\T \otimes \Phi K_{0,i} \Phi^\T\right) \widetilde{\sw}_0\right\} \nonumber\\
	&= \frac{1}{2} \E\left\{\widetilde{\sw}_0^\T (p \otimes \Phi) (I_N \otimes K_{0,i}) (p \otimes \Phi)^\T  \widetilde{\sw}_0\right\}
	\label{eq:transient_second_to_last}
\end{align} 
The notation $a(i)\sim b(i)$ is defined after \eqref{eq:asymptotic_theorem_result_sub}.  We now introduce the linear transformation, $\widetilde{\sw}_0'$, of the initial error vector $\widetilde{\sw}_0$:
\begin{align}
\widetilde{\sw}_0' \triangleq (p \otimes \Phi)^\T \widetilde{\sw}_0
\end{align}
This transformation allows us to rewrite  \eqref{eq:transient_second_to_last} as
\begin{align}
\E \|\widetilde{\sw}_0\|^2_{\Omega_{0,i}} &\sim \frac{1}{2} \E\left\{\widetilde{\sw}_0'^\T (I_N \otimes K_{0,i}) \widetilde{\sw}_0'\right\} \nonumber\\
&= \frac{1}{2} \sum_{k=1}^N \E \{\widetilde{\w}_{0,k}'^\T K_{0,i} \widetilde{\w}_{0,k}'\} \label{eq:final_transient}
\end{align}
where $\widetilde{\w}_{0,k}'$ is the $k$-th $M\times 1$ block of the block-vector $\widetilde{\sw}_{0}'$, where $\widetilde{\sw}_{0}' = \mathrm{col}\{\widetilde{\w}_{0,1}',\ldots,\widetilde{\w}_{0,N}'\}$.

The only remaining dependence on the iteration is now embedded in $K_{0,i}$ as defined in \eqref{eq:K_i}. Examining this diagonal matrix we have that:
\begin{align}
K_{0,i}	&=\! \mathrm{diag}\!\left\{\!\!\lambda_{1}\prod_{j=1}^{i-1}(1\!-\!\mu(j)\lambda_{1})^{2}, \ldots, \lambda_{M}\prod_{j=1}^{i-1}(1\!-\!\mu(j)\lambda_{M})^{2}\!\!\right\} \label{eq:K_0,i_diag}
\end{align}
where $\lambda_m$ is the $m$-th eigenvalue of the matrix $H$. Substituting \eqref{eq:K_0,i_diag} into \eqref{eq:final_transient}, we obtain:
\begin{align}
	\E \|\widetilde{\sw}_0\|^2_{\Omega_{0,i}} &\sim  \frac{1}{2} \sum_{m=1}^M \left[\!\lambda_{m}\prod_{j=1}^{i-1}(1\!-\!\mu(j)\lambda_{m})^{2}\!\right]\!  \sum_{k=1}^N\E \|\widetilde{\w}_{0,k,m}'\|^2 \label{eq:final_transient2}
\end{align}
where $\widetilde{\w}_{0,k,m}'$ denotes the $m$-th element of the vector $\widetilde{\w}_{0,k}'$, so that $\widetilde{\w}_{0,k}' = \mathrm{col}\{\widetilde{\w}_{0,k,1}', \ldots, \widetilde{\w}_{0,k,M}'\}$. In order to determine the rate of convergence of the quantity $\lambda_{m}\prod_{j=1}^{i-1}(1\!-\!\mu(j)\lambda_{m})^{2}$, we appeal to Lemma \ref{lem:bounds} (Appendix \ref{app:lemmas}). We use \eqref{eq:scalar_product_transient} in conjunction with \eqref{eq:final_transient2} to conclude that the transient term can be lower- and upper-bounded by:
\begin{align}
	\E \|\widetilde{\sw}_0\|^2_{\Omega_{0,i}} &\geq \frac{1}{2} \sum_{m=1}^M \lambda_m s_m \underline{c}_{m,i} \sum_{k=1}^N \E\|\widetilde{\w}_{0,k,m}'\|^2 \nonumber\\
											  &= \frac{1}{2} \sum_{m=1}^M  \frac{e^{2\sum_{j=1}^{\left\lceil \lambda_m \mu \right\rceil+1}\log\left(1\!-\!\frac{\lambda_m \mu}{j}\right)} \cdot \left(1-\frac{\lambda_m\mu}{i}\right)^{2i}}{(i\!-\!\lambda_m\mu)^{2\lambda_m\mu}} \cdot\nonumber\\
											  &\quad\quad\ \frac{\lambda_m(\left\lceil\lambda_m\mu\right\rceil\!-\!\lambda_m\mu\!+\!1)^{2\lambda_m\mu}}{\left(1-\frac{\lambda_m\mu}{\left\lceil\lambda_m\mu\right\rceil+1}\right)^{2(\left\lceil\lambda_m\mu\right\rceil+1)}} \sum_{k=1}^N \E\|\widetilde{\w}_{0,k,m}'\|^2 \label{eq:transient_lower_bound1}\\
											  &\stackrel{(a)}{\sim} \frac{1}{2} \sum_{m=1}^M  \frac{e^{2\sum_{j=1}^{\left\lceil \lambda_m \mu \right\rceil+1}\log\left(1-\frac{\lambda_m \mu}{j}\right)} \cdot e^{-2\lambda_m\mu}}{(i-\lambda_m\mu)^{2\lambda_m\mu}} \cdot \nonumber\\
											  &\quad\quad\!\frac{\lambda_m(\left\lceil\lambda_m\mu\right\rceil-\lambda_m\mu+1)^{2\lambda_m\mu}}{\left(1-\frac{\lambda_m\mu}{\left\lceil\lambda_m\mu\right\rceil+1}\right)^{2(\left\lceil\lambda_m\mu\right\rceil+1)}} \!\sum_{k=1}^N \E\|\widetilde{\w}_{0,k,m}'\|^2 \label{eq:transient_lower_bound} \\
	\E \|\widetilde{\sw}_0\|^2_{\Omega_{0,i}} &\leq \frac{1}{2} \sum_{m=1}^M \lambda_m s_m \bar{c}_{m,i} \sum_{k=1}^N \E\|\widetilde{\w}_{0,k,m}'\|^2 \nonumber\\
	&= \frac{1}{2} \!\!\sum_{m=1}^M  \!\!\frac{e^{2\sum_{j=1}^{\left\lceil \lambda_m \mu \right\rceil+1}\log\left(1-\frac{\lambda_m \mu}{j}\right)}\!\!\left(1\!-\!\frac{\lambda_m\mu}{i+1}\right)^{\!2(i+1)}}{(i+1-\lambda_m\mu)^{2\lambda_m\mu}} \cdot \nonumber\\
	&\quad\quad\ \frac{\lambda_m(\left\lceil\lambda_m\mu\right\rceil\!-\!\lambda_m\mu\!+\!2)^{2\lambda_m\mu}}{\left(1\!-\!\frac{\lambda_m\mu}{\left\lceil\lambda_m\mu\right\rceil+2}\right)^{2(\left\lceil\lambda_m\mu\right\rceil+2)}} \sum_{k=1}^N \E\|\widetilde{\w}_{0,k,m}'\|^2 \label{eq:transient_upper_bound1}\\
	&\stackrel{(b)}{\sim} \frac{1}{2} \sum_{m=1}^M  \frac{e^{2\sum_{j=1}^{\left\lceil \lambda_m \mu \right\rceil+1}\log\left(1-\frac{\lambda_m \mu}{j}\right)}e^{-2\lambda_m\mu}}{(i+1-\lambda_m\mu)^{2\lambda_m\mu}} \cdot \nonumber\\
	&\quad\quad\!\frac{\lambda_m(\left\lceil\lambda_m\mu\right\rceil-\lambda_m\mu+2)^{2\lambda_m\mu}}{\left(1-\frac{\lambda_m\mu}{\left\lceil\lambda_m\mu\right\rceil+2}\right)^{2(\left\lceil\lambda_m\mu\right\rceil+2)}} \!\sum_{k=1}^N \!\E\|\widetilde{\w}_{0,k,m}'\|^2 \label{eq:transient_upper_bound}
\end{align}
where steps $(a)$ and $(b)$ are due to the expressions $\left(1-\lambda_m\mu/i\right)^{2i}$ and $\left(1-\lambda_m\mu/(i+1)\right)^{2i+2}$ asymptotically converging to $e^{-2\lambda_m\mu}$ \cite[p.~26]{IntegralsSeries}, which is independent of $i$. In fact, the numerators in \eqref{eq:transient_lower_bound1} and \eqref{eq:transient_upper_bound1} account for the increase in the excess-risk at the beginning of the iterations. Eventually, however, the denominator terms of the form $(i-\lambda_m\mu)^{2\lambda_m\mu}$ and $(i+1-\lambda_m\mu)^{2\lambda_m\mu}$ will overtake the increase in $\left(1-\lambda_m\mu/i\right)^{2i}$ and $\left(1-\lambda_m\mu/(i+1)\right)^{2i+2}$ and the excess-risk will begin to decay from that point onwards. Furthermore, examination of \eqref{eq:transient_lower_bound} and \eqref{eq:transient_upper_bound} shows that the $m$-th term decays at the rate $O(i^{-2\lambda_m\mu})$, making the slowest term vanish at $O(i^{-2\lambda_{\min}\mu})$, where $\lambda_{\min}$ is the smallest eigenvalue of the matrix $H$.

\section{Proof of Theorem \ref{thm:HO2}}
\label{sec:Proof_HO2}

Since the matrices $\mathcal{R}_{d,i}$ and $\B_{j+1}^\T\cdots\B_{i-1}^\T\Sigma\B_{i-1}\cdots\B_{j+1}$ are positive semi-definite as long as $\Sigma$ is positive semi-definite, we have that
\begin{align}
	\mathrm{HO}_2(i) &\leq \sum_{j=1}^{i-1} \Tr\left(\frac{\mathcal{R}_{d,j}}{\mu^2(j)}\right) \Tr\left(\mu^2(j) \Omega_{j,i}\right) &\nonumber\\
	&= \sum_{j=1}^{i-1} \E\left\|\frac{\d_{j}}{\mu(j)}\right\|^2 \Tr\left(\mu^2(j)\Omega_{j,i}\right) \label{eq:HO2_sum1}
\end{align}
where in the first inequality we used the property that $0 \leq \Tr(AB) \leq \Tr(A) \Tr(B)$ for nonnegative matrices $A$ and $B$ (Lemma \ref{lem:trace_inequality} in Appendix \ref{app:lemmas}). To evaluate the convergence rate of $\E \|\d_j/\mu(j)\|^2$, we proceed as follows. First, we denote the block diagonal entries of the matrix $\boldsymbol{\D}_{j-1}$ defined by \eqref{eq:bold_script_H} as 
\begin{align}
\H_{k,j-1}\triangleq \int_{0}^{1}
\nabla_w^2\, J_k(w^o-t\widetilde{\w}_{k,j-1})dt
\end{align}
Then, we note that
\begin{align}
\|\d_j/\mu(j)\| &\leq \|\mathcal{A}^\T\|\cdot \left\|\left[\begin{array}{c}
(H - \boldsymbol{H}_{1,j-1})\widetilde{\w}_{1,j-1} \\ 
\vdots \\ 
(H - \boldsymbol{H}_{N,j-1})\widetilde{\w}_{N,j-1}
\end{array} \right]\right\|	\nonumber\\
&\leq \|\mathcal{A}^\T\|\cdot \sum_{k=1}^N \|H - \boldsymbol{H}_{k,j-1}\| \cdot \|\widetilde{\w}_{k,j-1} \| \label{eq:norm_d}
\end{align}
Using the global Lipschitz condition \eqref{eq:global_Lipschitz_Hessian}, we have that 
\begin{align}
	\|H - \boldsymbol{H}_{k,j-1}\| &= \left\|\int_0^1 \!H \!-\! \nabla^2_w J_k(w^o - t \widetilde{\w}_{k,j-1})dt\right\| \nonumber\\
	&\leq \tau_k \|\widetilde{\w}_{k,j-1}\| \int_0^1 \!t \ dt \nonumber\\
	&= \frac{\tau_k}{2} \|\widetilde{\w}_{k,j-1}\| \label{eq:bound_hdiff}
\end{align}
Substituting into \eqref{eq:norm_d}, we obtain
\begin{align}
	\|\d_j/\mu(j)\| &\leq \frac{1}{2} \cdot \|\mathcal{A}^\T\|\cdot \sum_{k=1}^N \tau_k \|\widetilde{\w}_{k,j-1} \|^2 \label{eq:norm_d2}
\end{align}
so that, upon squaring and using the Cauchy-Schwarz inequality,
\begin{align}
	\|\d_j/\mu(j)\|^2 &\leq \frac{1}{4} \cdot \|\mathcal{A}^\T\|^2\cdot \tau' \cdot \sum_{k=1}^N \|\widetilde{\w}_{k,j-1} \|^4 
\end{align}
where 
\begin{align}
	\tau' \triangleq \sum_{k=1}^N \tau_k^2
\end{align}
Taking the expectation, we obtain
\begin{align}
	\E\|\d_j/\mu(j)\|^2 &\leq \frac{1}{4} \cdot \|\mathcal{A}^\T\|^2\cdot \tau' \cdot \sum_{k=1}^N \E \|\widetilde{\w}_{k,j-1} \|^4 \nonumber\\
	&\leq \frac{N \cdot \tau'}{4} \cdot \|\mathcal{A}^\T\|^2\cdot \|\mathcal{W}^4_{j-1}\|_\infty \nonumber\\
	&\stackrel{(a)}{=} O(j^{-2}),\quad\quad 2\underline{\lambda}\mu > 1 \label{eq:HO_w_step_Thm1}
\end{align}
where $\mathcal{W}_{j-1}^4$ is defined in \eqref{eq:W_4} and step $(a)$ is due to \eqref{eq:fourth_ind}. This implies that there exist constants $d > 0$ and $j_o \geq 1$ such that for all $j \geq j_o$, we have that	$\E\|\d_j/\mu(j)\|^2/j^{-2} < d$. Therefore, we can split the sum in \eqref{eq:HO2_sum1} into two sums:
\begin{equation}
	\mathrm{HO}_2(i) \leq d \cdot \sum_{j=j_o}^{i-1} j^{-2} \Tr\left(\mu^2(j)\Omega_{j,i}\right) + \sum_{j=1}^{j_o-1} \E\left\|\d_{j}\right\|^2 \Tr\left(\Omega_{j,i}\right) \label{eq:HO2_sum_two_terms}
\end{equation}
For the first term, we immediately recognize that it fits in the form required by Lemma \ref{lem:series_of_traces} (Appendix \ref{app:lemmas}) with the identifications: $b \leftarrow 2, \mathcal{L} \leftarrow I, q \leftarrow j_o$. So we conclude that the first term asymptotically becomes:
\begin{align}
	d \cdot \sum_{j=j_o}^{i-1} &j^{-2} \Tr\left(\mu(j)^2\Omega_{j,i}\right) \sim \frac{d\cdot \mu^2\cdot \|p\|^2}{2} \sum_{m=1}^M \lambda_m \alpha_{m,2}(i)  \label{eq:HO_w_first_term1}
\end{align}
where
\begin{align}
	\alpha_{m,2}(i) &= \begin{cases}
																										\frac{i^{-3}}{2\lambda_m \mu -3} + \Theta(i^{-2\lambda_{\min}\mu}), & 1 < 2\lambda_{m}\mu \neq 3\\
																										\frac{\log(i)}{i^{3}}, &2\lambda_{m}\mu = 3
																										\end{cases}
\end{align}
and where we require $2\lambda_{\min} \mu > 1$ since step \eqref{eq:HO_w_step_Thm1} required $2\underline{\lambda}\mu > 1$, which automatically guarantees that condition   $2\lambda_{\min} \mu > 1$ is satisfied since $\underline{\lambda} \leq \lambda_{\min}$. Now  \eqref{eq:HO_w_first_term1} will converge at the rate of its slowest term and, hence, we have that
\begin{align}
d \cdot &\sum_{j=j_o}^{i-1} j^{-2} \Tr\left(\mu(j)^2\Omega_{j,i}\right) \nonumber\\
&\sim \begin{cases}
																										\Theta\left(i^{-3}\right) + \Theta(i^{-2\lambda_{\min}\mu}), & 1 < 2\lambda_{\min}\mu \neq 3\\
																										\Theta\left(\frac{\log(i)}{i^{3}}\right), &2\lambda_{\min}\mu = 3
																										\end{cases}\label{eq:HO_w_first_term}
\end{align}
On the other hand, for the second term of \eqref{eq:HO2_sum_two_terms}, we recall \eqref{eq:A_jordan2}--\eqref{eq:K_i} and observe that the trace can be computed as:
\begin{align}
	\Tr(\Omega_{j,i}) &= \Tr(TD^{i-j}T^{-1}E_{kk}T^{-\T}{D^\T}^{i-j}T^\T \otimes \Phi K_{j,i} \Phi^\T)\nonumber\\
	&\stackrel{(a)}{=} \Tr(TD^{i-j}T^{-1}E_{kk}T^{-\T}{D^\T}^{i-j}T^\T) \times\nonumber\\
	&\quad\ \sum_{m=1}^M \lambda_m \frac{\Gamma^2(i-\lambda_m\mu)}{\Gamma^2(i)}\cdot\frac{\Gamma^2(j+1)}{\Gamma^2(j+1-\lambda_m\mu)} \nonumber\\
	&\stackrel{(b)}{\sim} \Tr(TD^{i-j}T^{-1}E_{kk}T^{-\T}{D^\T}^{i-j}T^\T) \times\nonumber\\
	&\quad\ \sum_{m=1}^M i^{-2\lambda_m \mu}\lambda_m \frac{\Gamma^2(j+1)}{\Gamma^2(j+1-\lambda_m\mu)} \nonumber\\
	&\stackrel{(c)}{\sim} \Tr(TE_{11}T^{-1}E_{kk}T^{-\T}E_{11}T^\T) \times\nonumber\\
	&\quad\ \sum_{m=1}^M i^{-2\lambda_m \mu}\lambda_m \frac{\Gamma^2(j+1)}{\Gamma^2(j+1-\lambda_m\mu)} \nonumber\\
	&\stackrel{(d)}{=} \Tr(p \mathds{1}_N^\T E_{kk}\mathds{1}_N p^\T) \sum_{m=1}^M i^{-2\lambda_m \mu} \frac{\lambda_m\cdot \Gamma^2(j+1)}{\Gamma^2(j+1-\lambda_m\mu)} \nonumber\\
	&= \sum_{m=1}^M i^{-2\lambda_m \mu}\lambda_m \frac{\Gamma^2(j+1)}{\Gamma^2(j+1-\lambda_m\mu)} \|p\|^2 \label{eq:Tr_Gamma}
\end{align}
where step $(a)$ is due to Lemma \ref{lem:bounds} (Appendix \ref{app:lemmas}); step $(b)$ is due to Lemma \ref{lem:Gamma} (Appendix \ref{app:lemmas}); step $(c)$ is due to the fact that $D^{i-j} \rightarrow E_{11}$ for large $i$ (since $j$ will only increase up to $j_o$, which will become small in comparison to large $i$ asymptotically); and step $(d)$ is due to the fact that $T E_{11} T^{-1}$ is a rank-$1$ matrix that is spanned by the left- and right-eigenvectors of $A$ corresponding to the eigenvalue $1$. The left eigenvector is $\mathds{1}_N$ since $A$ is left-stochastic. The right eigenvector is  $p$ and is normalized so that the sum of its entries is unity; i.e., $p^\T \mathds{1}_N = 1$ and $A p = p$. Then, $T E_{11} T^{-1} = p \mathds{1}^\T_N$. Substituting \eqref{eq:Tr_Gamma} into the second term of \eqref{eq:HO2_sum_two_terms}, we obtain:
\begin{align}
	\sum_{j=1}^{j_o-1} &\E\left\|\d_{j}\right\|^2 \Tr\left(\Omega_{j,i}\right) \nonumber\\
	&\sim  \sum_{m=1}^M i^{-2\lambda_m \mu} \sum_{j=1}^{j_o-1}  \E\left\|\d_{j}\right\|^2 \lambda_m \frac{\Gamma^2(j+1)}{\Gamma^2(j+1-\lambda_m\mu)} \|p\|^2 \nonumber\\
	&\stackrel{(a)}{=} O(i^{-2\lambda_{\min} \mu}) \label{eq:HO_2_second_term}
\end{align}
where step $(a)$ is due to the fact that the inner sum has finite bounds. Combining \eqref{eq:HO_w_first_term} and \eqref{eq:HO_2_second_term}, we have that:
\begin{align}
\mathrm{HO}_2(i) \!=\! \begin{cases}
																										\Theta(i^{-3}) \!+\! \Theta(i^{-2\lambda_{\min}\mu}), & 1 < 2\lambda_{\min}\mu \neq 3\\
																										\Theta(\frac{\log(i)}{i^{3}}), &2\lambda_{\min}\mu = 3
																										\end{cases}
\end{align}

\section{Proof of Theorem \ref{thm:HO3}}
\label{sec:Proof_HO3}
The study of the final term of \eqref{eq:unrolled_recursion} is slightly more complicated than that of the others since in this case the matrix $\E\{\mathcal{B}_j \widetilde{\sw}_{j-1} \d_{j}^\T\}$ is no longer symmetric, let alone positive semi-definite. Nevertheless, using Lemma \ref{lem:trace_inequality} (Appendix \ref{app:lemmas}), we can obtain:
\begin{align}
	\mathrm{HO}_3(i) &\leq \sum_{j=1}^{i-1} \Tr\!\left(\mu^2(j)\Omega_{j,i}\right) \left\|\frac{\E\{\B_{j} \widetilde{\sw}_{j-1} \d_{j}^\T\}}{\mu(j)^2}\right\| \label{eq:HO_3_orig_sum}
\end{align}
Now, using Jensen's inequality and the submultiplicative norm property of the spectral norm, we have that
\begin{align}
	&\left\|\frac{\E\{\B_{j} \widetilde{\sw}_{j-1} \d_{j}^\T\}}{\mu^2(j)}\right\| \leq \frac{1}{\mu(j)} \E\{\|\B_{j}\| \cdot \|\widetilde{\sw}_{j-1}\| \cdot \|\d_{j}/\mu(j)\|\} \nonumber\\
	&\leq \frac{1}{\mu(j)}  \E\{\|\A^\T\| \|I\!-\!\mu(j) \D\| \cdot \|\widetilde{\sw}_{j-1}\| \cdot \|\d_{j}/\mu(j)\|\}
\end{align}
Using \eqref{eq:norm_d2}, we obtain
\begin{align}
	\left\|\frac{\E\{\B_{j} \widetilde{\sw}_{j-1} \d_{j}^\T\}}{\mu^2(j)}\right\| &\leq \frac{1}{2 \mu(j)} \|\A^\T\|^2\cdot \|I\!-\!\mu(j) \D\| 
	\cdot\nonumber\\
	&\quad\ \E\left\{\!\|\widetilde{\sw}_{j-1}\| \sum_{k=1}^N \tau_k \|\widetilde{\w}_{k,j-1}\|^2\!\right\}
\end{align}
Now, since $\|\widetilde{\sw}_{j-1}\| \leq \sum_{k=1}^N \|\widetilde{\w}_{k,j-1}\|$, we have that:
\begin{align}
\E&\left\{\|\widetilde{\sw}_{j-1}\| \sum_{k=1}^N \tau_k \|\widetilde{\w}_{k,j-1}\|^2\right\} \nonumber\\
&\quad\leq \E\left\{\left(\sum_{k=1}^N \|\widetilde{\w}_{k,j-1}\|\right) \left(\sum_{k=1}^N \tau_k \|\widetilde{\w}_{k,j-1}\|^2\right)\right\}\nonumber\\
&\quad\stackrel{(a)}{\leq}\sqrt{\E\left|\sum_{k=1}^N \|\widetilde{\w}_{k,j-1}\|\right|^2 \cdot \E\left|\sum_{k=1}^N \tau_k \|\widetilde{\w}_{k,j-1}\|^2\right|^2}\nonumber\\
&\quad\stackrel{(b)}{\leq}\sqrt{N \sum_{k=1}^N \E\|\widetilde{\w}_{k,j-1}\|^2 \cdot \tau' \cdot \sum_{k=1}^N \E\|\widetilde{\w}_{k,j-1}\|^4}
\end{align}
where steps $(a)$-$(b)$ are due to the Cauchy-Schwarz inequality and $\tau' \triangleq \sum_{k=1}^N \tau_k^2$. Now, from Theorem \ref{thm:non_asym_convergence_main}, we know that $\E\|\widetilde{\w}_{k,j-1}\|^2 = O(j^{-1})$ and $\E\|\widetilde{\w}_{k,j-1}\|^4 = O(j^{-2})$ when $2\underline{\lambda}\mu > 1$. Therefore, we can conclude that $\E\left\{\|\widetilde{\sw}_{j-1}\| \sum_{k=1}^N \tau_k \|\widetilde{\w}_{k,j-1}\|^2\right\} = O(j^{-3/2})$ and so
\begin{align}
	\left\|\frac{\E\{\B_{j} \widetilde{\sw}_{j-1} \d_{j}^\T\}}{\mu^2(j)}\right\| = O(j^{-1/2})
\end{align}
when $2\underline{\lambda}\mu > 1$. That is, there exist constants $c > 0$ and $j_1 \geq 1$ such that for all $j\geq j_1$, we have that
\begin{align}
	\left\|\frac{\E\{\B_{j} \widetilde{\sw}_{j-1} \d_{j}^\T\}}{\mu^2(j)}\right\| < c \cdot j^{-1/2}
\end{align}
So we can now mirror the technique in App.~\ref{sec:Proof_HO2} and split the sum in \eqref{eq:HO_3_orig_sum} as
\begin{align}
	\mathrm{HO}_3(i) &\leq c \cdot \sum_{j=j_1}^{i-1} \Tr(j^{-1/2} \mu^2(j) \Omega_{j,i}) + \nonumber\\
	&\quad\ \sum_{j=1}^{j_1-1} \Tr\left(\Omega_{j,i}\right) \left\|\E\{\B_{j} \widetilde{\sw}_{j-1} \d_{j}^\T\}\right\|  \label{eq:HO_3_sum2}
\end{align}
For the first term, we use Lemma \ref{lem:series_of_traces} (Appendix \ref{app:lemmas}) with the identifications $b \leftarrow \frac{1}{2}, \mathcal{L} \leftarrow I, q \leftarrow j_1$ to conclude that
\begin{align}
	c \cdot \sum_{j=j_1}^{i-1} \Tr(j^{-1/2} \mu^2(j) \Omega_{j,i}) \sim \frac{c \cdot \mu^2\cdot \|p\|^2}{2} \sum_{m=1}^M \lambda_m \alpha_{m,1/2}(i)  \label{eq:HO_3_first_term1}
\end{align}
where
\begin{align}
	\alpha_{m,2}(i) &= \begin{cases}
																										\frac{i^{-3/2}}{2\lambda_m \mu -3} + \Theta(i^{-2\lambda_m\mu}), & 1 < 2\lambda_{m}\mu \neq 3/2\\
																										\frac{\log(i)}{i^{3/2}}, &2\lambda_{m}\mu = 3/2
																										\end{cases}
\end{align}
Since \eqref{eq:HO_3_first_term1} will converge at the rate of its slowest term, we have that:
\begin{align}
c \cdot &\sum_{j=j_1}^{i-1} j^{-1/2} \Tr\left(\mu(j)^2\Omega_{j,i}\right) \nonumber\\
&\sim \begin{cases}
																										\!\Theta\left(i^{-3/2}\right) \!+\! \Theta(i^{-2\lambda_{\min}\mu}), & 1 < 2\lambda_{\min}\mu \neq 3/2\\
																										\!\Theta\left(\frac{\log(i)}{i^{3/2}}\right), &2\lambda_{\min}\mu = 3/2
																										\end{cases}\label{eq:HO_3_first_term}
\end{align}

For the second term, we substitute \eqref{eq:Tr_Gamma} to obtain
\begin{align}
 &\sum_{j=1}^{j_1-1}\Tr\left(\Omega_{j,i}\right) \left\|\E\{\B_{j} \widetilde{\sw}_{j-1} \d_{j}^\T\}\right\| \nonumber\\
 &\quad\ \sim \sum_{m=1}^M \!i^{-2\lambda_m \mu}\lambda_m \!\!\sum_{j=1}^{j_1-1} \left\|\E\{\B_{j} \widetilde{\sw}_{j-1} \d_{j}^\T\}\right\| \!\cdot\! \frac{\Gamma^2(j\!+\!1)\cdot \|p\|^2}{\Gamma^2(j\!+\!1\!-\!\lambda_m\mu)} \nonumber\\
 &\quad\ \stackrel{(a)}{=} O(i^{-2\lambda_{\min} \mu}) \label{eq:HO_3_second_term}
\end{align}
where step $(a)$ is due to the fact that $\B_{j}$, $\widetilde{\sw}_{j-1}$, and $\d_{j}$ depend on $j$ only and are independent of $i$ so that this finite sum leads to a constant term that is independent of $i$. Combining \eqref{eq:HO_3_first_term} and \eqref{eq:HO_3_second_term}, we have that:
\begin{align}
\mathrm{HO}_3(i) \!=\! \begin{cases}
																										\!\Theta(i^{-3/2}) \!+\! \Theta(i^{-2\lambda_{\min}\mu}), & \!1 \!<\! 2\lambda_{\min}\mu \!\neq\! \frac{3}{2}\\
																										\!\Theta(\frac{\log(i)}{i^{3/2}}), &2\lambda_{\min}\mu \!=\! \frac{3}{2}
																										\end{cases}
\end{align}

\section{Proof of Theorem \ref{thm:consensus}}
\label{sec:Proof_diffusion_consensus}
The main difference between the dynamics of the diffusion and consensus implementations is in the definition of the $\B_i$ matrices in \eqref{eq:B_i} where
\begin{align}
	{\B_{i}^\textrm{diff}} &\triangleq \A^\T (I_{MN} \!-\! \mu(i) \mathcal{H}) \label{eq:B_i_diff}\\
	{\B_{i}^\textrm{cons}} &\triangleq \A^\T - \mu(i) \mathcal{H}  \label{eq:B_i_cons}
\end{align}
where ${\cal A}$ appears in a multiplicative form in \eqref{eq:B_i_diff} and in an additive form in \eqref{eq:B_i_cons}. The apparently small change in the order in which operations take place within the consensus and diffusion strategies actually leads to significant differences in the evolution of the error vectors over the respective networks leading to worse transient and asymptotic performance for consensus strategies; these conclusions are consistent with results reported in \cite{shine_diffusion_consensus,NOW_ML} albeit for constant step-size adaptation. To examine the differences in behavior, we will consider the network excess-risk:
\begin{align}
	\textrm{ER}(i) \triangleq \frac{1}{N} \sum_{k=1}^N \textrm{ER}_k(i)
\end{align}
When $2\lambda_{\min}\mu > 1$, we can average the asymptotic terms from \eqref{eq:unrolled_recursion} to obtain the following expression for the asymptotic excess-risk (which can also be obtained by substituting $\Sigma = I_N \otimes \frac{1}{2} H$ into \eqref{eq:def_AT}): 
\begin{align}
	\textrm{ER}(i) \!= \!\frac{1}{2 N} \sum_{j=1}^{i-1} \mu^2(j) \Tr\left( \mathcal{H} \left(\!\prod_{t=j+1}^{i-1} \!\!\!\!\B_t^\T\!\!\right)^{\!\!\!\!\T} \!\!\G\!\! \left(\!\prod_{t=j+1}^{i-1} \!\!\!\!\B_t^\T\right)\!\!\!\right)
	\label{eq:Trace_diffusion_consensus}
\end{align}
as $i\rightarrow\infty$ where $\B_t$ is either ${\B_{t}^\textrm{diff}}$ or ${\B_{t}^\textrm{cons}}$, depending on which algorithm we wish to examine. Likewise, the matrix $\G$ is either $\G^\textrm{diff}$ or $\G^\textrm{cons}$, depending on the algorithm, where
$\G^\textrm{diff} \triangleq \A^\T \R_g^o \A$ and $\G^\textrm{cons} \triangleq \R_g^o$. Table \ref{tbl:Diffusion_Consensus} lists the parameters for both implementations. 

For simplicity, we assume $A$ is symmetric; more generally, we can consider combination policies $A$ that are close-to-symmetric and employ arguments similar to \cite{shine_diffusion_consensus}. The final conclusion will be similar to the arguments given here.  Now, since $A$ is primitive, the matrix $D$ in the Jordan canonical form of $A = T D T^{-1}$ will be diagonal with a single eigenvalue at one and with all other eigenvalues strictly inside the unit circle. We let the vectors $r_k$ and  $y_{k}$ for $k=\{1,\ldots,N\}$ represent the right and left eigenvectors, respectively, of the matrix $A^\T$ corresponding to the eigenvalue $D_{kk}$ (the $k$-th diagonal element of $D$), i.e.,
$A^\T r_k = D_{kk} r_k, y_k^\T A^\T = D_{kk} y_k^\T$, where we normalize the vectors so that $r_k^{\T} y_{l}=\delta_{k,l}$. In fact, since $A$ is symmetric, the eigenvectors $\{r_k\}$ are themselves orthonormal, as well as the eigenvectors $\{y_k\}$. When $A p = p$, then $D_{11}=1$, $r_1 = \mathds{1}_N$, and $y_1 = p$. Furthermore, let $\{s_m, m=1,\ldots,M\}$ denote the eigenvectors of the matrix $H$, i.e.,
$H s_m = \lambda_m s_m$,
where $\lambda_m$ is the $m$-th eigenvalue of $H$ and the eigenvectors $s_m$ are normalized so that $\|s_m\|^2 = 1$. We observe that the matrices ${\B_{i}^\textrm{diff}}$ and ${\B_{i}^\textrm{cons}}$ share the same eigenvectors but have different eigenvalues. We denote the eigenvectors of $\B_i$ by $y_{l,m}^b$. They can be found to be 
$r_{l,m}^b = r_l \otimes s_m$ and $y_{l,m}^b = y_l \otimes s_m$ \cite{shine_diffusion_consensus,NOW_ML}.
\begin{table}
	\renewcommand{\arraystretch}{2.2}
	\caption{Variables in the Diffusion and Consensus Implementations.}
	\centering
	\begin{tabular}{c||c|c}
	\hline \hline
	\rowcolor[gray]{0.9}{\small{Algorithm}} & {\small{Diffusion \eqref{eq:A}-\eqref{eq:C2}}} & {\small{Consensus \eqref{eq:consensus}}} \\ 
	\hline
	$\B_i$ & $\A^\T (I_{MN} - \mu(i) \D)$ & $\A^\T - \mu(i) \D$ \\ 
	\hline 
	$\lambda_{k,m}(\B_i)$ & $D_{kk} (1-\mu(i) \lambda_m)$ & $D_{kk}-\mu(i) \lambda_m$ \\ 
	\hline 
	$\G$ & $\A^\T \R_g^o \A$ & $\R_g^o$ \\ 
	\hline 
	${y_{k,m}^{b \T}} \G y_{k,m}^b$ & $D_{kk}^2 {y_{k,m}^{b \T}} \mathcal{R}_g^o y_{k,m}^b$  & ${y_{k,m}^{b \T}} \mathcal{R}_g^o y_{k,m}^b$  \\ 
	\hline \hline
	\end{tabular} 
	\label{tbl:Diffusion_Consensus}
	\vspace*{-1\baselineskip}
\end{table}
We now introduce the eigendecomposition 
\begin{align}
	\B_i = \sum_{k=1}^N \sum_{m=1}^M r_{k,m}^b {y_{k,m}^{b \T}} \lambda_{k,m}(\B_i)
\end{align}
where the expressions for the eigenvalues $\lambda_{k,m}(\B_i)$ are listed in Table \ref{tbl:Diffusion_Consensus} for the diffusion and consensus strategies. Now, since the eigenvectors are orthonormal, we have that the finite product of $\B_t$ matrices is:
\begin{align}
	\prod_{t=j+1}^{i-1} \B_t = \sum_{k=1}^N \sum_{m=1}^M r_{k,m}^b {y_{k,m}^{b \T}} \left(\prod_{t=j+1}^{i-1} \lambda_{k,m}(\B_t)\right)
	\label{eq:product_B}
\end{align}
We substitute \eqref{eq:product_B} into \eqref{eq:Trace_diffusion_consensus} to get
\begin{align}
	\textrm{ER}(i) &= \frac{\mu^2}{2N} \sum_{j=1}^{i-1} \!\frac{1}{j^2} \!\Tr\Bigg(\!\!\mathcal{H} \sum_{k=1}^N \!\sum_{\ell=1}^N \!\sum_{m_1=1}^M \!\sum_{m_2=1}^M y_{k,m_1}^{b \T} \mathcal{G} y_{\ell,m_2}^{b} \times \nonumber\\
	&\quad\quad\ r_{k,m_1}^b  r_{\ell,m_2}^{b \T} \bigg(\prod_{t=j+1}^{i-1} \!\!\lambda_{k,m_1}(\B_t)\!\!\bigg)\!\!\bigg(\prod_{t=j+1}^{i-1} \!\!\lambda_{\ell,m_2}(\B_t)\!\!\bigg)\!\!\Bigg)\nonumber\\
	&\stackrel{(a)}{=} \frac{\mu^2}{2N} \sum_{j=1}^{i-1} \frac{1}{j^2} \Tr\Bigg( \sum_{k=1}^N \sum_{\ell=1}^N \sum_{m_1=1}^M \sum_{m_2=1}^M  {y_{k,m_1}^{b \T}} \mathcal{G} {y_{\ell,m_2}^{b}} \times\nonumber\\
	&\quad\quad\ (r_k^\T r_\ell \otimes s_{m_1}^\T H s_{m_2})\! \bigg(\!\prod_{t=j+1}^{i-1} \lambda_{k,m_1}(\B_t)\!\bigg)\times\nonumber\\
	&\quad\quad\quad\ \bigg(\!\prod_{t=j+1}^{i-1} \lambda_{\ell,m_2}(\B_t)\!\bigg)\!\!\Bigg)\nonumber\\
	&\stackrel{(b)}{=} \frac{\mu^2}{2N} \sum_{j=1}^{i-1} \frac{1}{j^2} \Tr\Bigg( \sum_{k=1}^N \sum_{\ell=1}^N (\|r_k\|^2 \otimes \lambda_m) {y_{k,m}^{b \T}} \mathcal{G} {y_{k,m}^{b}} \times\nonumber\\
	&\quad\quad\quad\quad\quad\quad\quad\quad\quad\quad\quad\ \bigg(\prod_{t=j+1}^{i-1} \lambda_{k,m}^2(\B_t)\bigg)\!\!\Bigg) \nonumber\\
	&= \frac{\mu^2}{2 N} \sum_{k=1}^N \sum_{m=1}^M  \lambda_m \cdot \|r_k\|^2 \cdot y_{k,m}^{b\T} \G y_{k,m}^b \times\nonumber\\
	&\quad\quad\quad\quad\quad\quad\quad\sum_{j=1}^{i-1}\frac{1}{j^2}\prod_{t=j+1}^{i-1} \!\!\!\lambda_{k,m}^2(\B_t)  
\end{align} 
where step $(a)$ is due to $\mathcal{H} = I_N \otimes H$ and $r_{k,m} = r_k \otimes s_m$ and step $(b)$ is due to the fact that $D$ is assumed to be diagonalizable (and $H$ is diagonalizable since it is symmetric), which implies that $r_k^\T r_\ell = \|r_k\|^2 \delta_{k,\ell}$ and $s_{m_1}^\T H s_{m_2} = \lambda_k \delta_{m_1,m_2}$, where $\delta_{i,j} = 1$ only when $i=j$ and is zero otherwise. Using Lemma \ref{lem:bounds} (Appendix \ref{app:lemmas}) and the asymptotic expansion in \eqref{eq:asym_upper_gamma}, we can write:
\begin{align}
	\textrm{ER}^\textrm{diff}(i) &\stackrel{(a)}{=} \frac{\mu^2}{2 N} \sum_{k=1}^N \sum_{m=1}^M   \lambda_m D_{kk}^2 \|r_k\|^2 \cdot  y_{k,m}^{b\T}\mathcal{R}_g^oy_{k,m}^{b} \times\nonumber\\
	&\quad\quad\quad\quad\quad\quad\quad\ \sum_{j=1}^{i-1}\frac{1}{j^2}\prod_{t=j+1}^{i-1}\lambda_{k,m}^2(\B_t)\nonumber\\
	&\stackrel{(b)}{=} \frac{\mu^2}{2 N} \sum_{k=1}^N \sum_{m=1}^M   \!\!\lambda_m D_{kk}^2 \|r_k\|^2 \cdot  y_{k,m}^{b\T}\mathcal{R}_g^oy_{k,m}^{b} \times\nonumber\\
	&\quad\quad\quad\quad\quad\  \sum_{j=1}^{i-1} D_{kk}^{2(i-j-1)} j^{-2} \!\!\prod_{t=j+1}^{i-1} (1-\mu(j)\lambda_m)^2 \nonumber\\
	&\stackrel{(c)}{=} \frac{\mu^2}{2 N} \sum_{k=1}^N \sum_{m=1}^M   \!\!\lambda_m \|r_k\|^2 \cdot  y_{k,m}^{b\T}\mathcal{R}_g^oy_{k,m}^{b} \times\nonumber\\
	&\quad\ \sum_{j=1}^{i-1} D_{kk}^{2(i-j)} j^{-2} \frac{\Gamma^2(i\!+\!1\!-\!\lambda_m \mu)}{\Gamma^2(i+1)} \!\cdot\! \frac{\Gamma^2(j+1)}{\Gamma^2(j\!+\!1\!-\!\lambda_m\mu)} \nonumber\\
	&\stackrel{(d)}{=} \frac{\mu^2}{2 N} \sum_{k=1}^N \sum_{m=1}^M   \lambda_m \|r_k\|^2 \cdot  y_{k,m}^{b\T}\mathcal{R}_g^oy_{k,m}^{b} i^{-2\lambda_m\mu} \times\nonumber\\
	&\quad\quad\quad\quad\quad\quad\ \sum_{j=1}^{i-1} D_{kk}^{2(i-j)}  \frac{\Gamma^2(j)}{\Gamma^2(j+1-\lambda_m\mu)} \label{eq:ER_diff}
\end{align}
where step $(a)$ is due to the last line of Table \ref{tbl:Diffusion_Consensus}; step $(b)$ is due to the second line of Table \ref{tbl:Diffusion_Consensus}; step $(c)$ is due to \eqref{eq:SS_finite_product_scalar}; and step $(d)$ is due to \eqref{eq:asym_frac_Gamma} and the property $x\cdot \Gamma(x) = \Gamma(x+1)$.

For the consensus algorithm, we have:
\begin{align}
\textrm{ER}^\textrm{cons}(i) &\stackrel{(a)}{=} \frac{\mu^2}{2 N} \sum_{k=1}^N \sum_{m=1}^M   \lambda_m \|r_k\|^2 \cdot  y_{k,m}^{b\T}\mathcal{R}_g^oy_{k,m}^{b} \times\nonumber\\
&\quad\quad\quad\quad\quad\quad\quad\ \sum_{j=1}^{i-1}\frac{1}{j^2}\prod_{t=j+1}^{i-1}\lambda_{k,m}^2(\B_t)\nonumber\\
	&\stackrel{(b)}{=} \frac{\mu^2}{2 N} \sum_{k=1}^N \sum_{m=1}^M   \lambda_m  \|r_k\|^2 \cdot  y_{k,m}^{b\T}\mathcal{R}_g^oy_{k,m}^{b} \times\nonumber\\
	&\quad\quad\quad\quad\quad\quad\quad\ \sum_{j=1}^{i-1} j^{-2} \prod_{t=j+1}^{i-1} (D_{kk}-\mu(j)\lambda_m)^2 \nonumber\\
	&\stackrel{(c)}{=} \frac{\mu^2}{2 N} \sum_{k=1}^N \sum_{m=1}^M   \!\!\lambda_m \|r_k\|^2 \cdot  y_{k,m}^{b\T}\mathcal{R}_g^oy_{k,m}^{b} \times\nonumber\\
	&\quad\quad\quad\ \sum_{j=1}^{i-1} D_{kk}^{2(i-j-1)} j^{-2} \frac{\Gamma^2(i+1-\lambda_m \mu D_{kk}^{-1})}{\Gamma^2(i+1)} \cdot\nonumber\\
	&\quad\quad\quad\quad\quad\quad\quad\quad\quad\quad\quad \frac{\Gamma^2(j+1)}{\Gamma^2(j+1-\lambda_m\mu D_{kk}^{-1})} \nonumber\\
	&\stackrel{(d)}{=} \frac{\mu^2}{2 N} \sum_{k=1}^N \sum_{m=1}^M   \!\!\lambda_m \|r_k\|^2 \cdot  y_{k,m}^{b\T}\mathcal{R}_g^oy_{k,m}^{b}   \times\nonumber\\
	&\quad\quad i^{-2\lambda_m\mu D_{kk}^{-1}}\sum_{j=1}^{i-1}  \frac{D_{kk}^{2(i-j-1)} \cdot \Gamma^2(j)}{\Gamma^2(j+1-\lambda_m\mu D_{kk}^{-1})} \label{eq:ER_cons}
\end{align}
where step $(a)$ is due to the last line of Table \ref{tbl:Diffusion_Consensus}; step $(b)$ is due to the second line of Table \ref{tbl:Diffusion_Consensus}; step $(c)$ is due to \eqref{eq:SS_finite_product_scalar}; and step $(d)$ is due to \eqref{eq:asym_frac_Gamma} and the property $x\cdot \Gamma(x) = \Gamma(x+1)$.
 
By comparing \eqref{eq:ER_diff} and \eqref{eq:ER_cons}, we observe that the equations for the asymptotic expected excess-risk for the diffusion and consensus strategies are identical except for the most inner summation. When $D_{kk}=1$, the summands inside the most inner terms are identical. In fact, this variation  is the key to the performance difference between the two algorithms. Using Lemma \ref{lem:i^{-2}} (Appendix \ref{app:lemmas}) with $b=0$, we can obtain the following \textit{upper}-bound for the inner summation of the diffusion strategy for any $D_{kk} < 1$:
\begin{align}
	i^{-2\lambda_m\mu} \sum_{j=1}^{i-1} D_{kk}^{2(i-j)}  \frac{\Gamma^2(j)}{\Gamma^2(j+1-\lambda_m\mu)} \leq \frac{1}{\log(D_{kk}^{-1})}\cdot \frac{1}{i^2} \label{eq:ER_UB_diff}
\end{align}
Similarly, using Lemma \ref{lem:i^{-2}} with $b=0$, we can obtain the following \textit{lower}-bound for the inner summation of the diffusion strategy for any $D_{kk} < 1$:
\begin{align}
	D_{kk}^{-2} i^{-2\lambda_m\mu D_{kk}^{-1}} \sum_{j=1}^{i-1}   \frac{D_{kk}^{2(i-j)}\cdot \Gamma^2(j)}{\Gamma^2(j+1-\lambda_m\mu D_{kk}^{-1})} \geq \frac{1}{\log(D_{kk}^{-1})}\cdot \frac{1}{i^2} \label{eq:ER_LB_cons}
\end{align}
as $i\rightarrow\infty$. Therefore, 
\begin{align}
	\textrm{ER}^\textrm{cons}(i)-\textrm{ER}^\textrm{diff}(i) &=  \frac{\mu^2}{2 N} \sum_{k=1}^N \sum_{m=1}^M   \lambda_m \|r_k\|^2 \cdot  y_{k,m}^{b\T}\mathcal{R}_g^oy_{k,m}^{b}\times\nonumber\\
	&\!\!\!\!\!\!\!\!\!\!\!\!\!\!\!\!\!\!\!\!\!\!\!\!\!\!\!\!\!\!\!\!\!\!\!\! \left(i^{-2\lambda_m\mu D_{kk}^{-1}} \sum_{j=1}^{i-1} D_{kk}^{2(i-j-1)}  \frac{\Gamma^2(j)}{\Gamma^2(j+1-\lambda_m\mu D_{kk}^{-1})} -\right.\nonumber\\
	&\left.i^{-2\lambda_m\mu} \sum_{j=1}^{i-1} D_{kk}^{2(i-j)}  \frac{\Gamma^2(j)}{\Gamma^2(j+1-\lambda_m\mu)}\right)\nonumber\\
	&\stackrel{(a)}{\geq} \frac{\mu^2}{2 N} \!\sum_{k=1}^N \sum_{m=1}^M   \!\!\lambda_m \|r_k\|^2 \!\cdot\!  y_{k,m}^{b\T}\mathcal{R}_g^oy_{k,m}^{b} \times\nonumber\\
	&\quad\quad\!\left(\frac{1}{\log(D_{kk}^{-1})}\cdot \frac{1}{i^2} - \frac{1}{\log(D_{kk}^{-1})}\cdot \frac{1}{i^2}\right) \nonumber\\
	&= 0
\end{align}
for large $i$, where step $(a)$ is due to \eqref{eq:ER_UB_diff}--\eqref{eq:ER_LB_cons}.

\section{Useful Lemmas}
\label{app:lemmas}
In this appendix, we list several useful results that are used in the later appendices.

\subsection{Matrix Results}
We first introduce the following lemma that relates the norm of a matrix power to the power of its spectral radius.
\begin{lemma}[\textbf{Bound on the norm of a matrix power}]
\label{lem:matrix_power}
Let $A$ denote a matrix whose spectral radius, $\rho(A)$, is strictly less than $1$. Then,
$
	\|A^n\| \leq c\cdot \left(\frac{\rho(A) + 1}{2}\right)^n
$
where $n \in \mathbb{N}$ and $c$ is some positive constant.
\end{lemma}
\begin{proof}
	From \cite[p.~299]{HornJohnsonVol1}, we have
	$\|A^n\| \leq c\cdot(\rho(A) + \epsilon)^n$
	for any $\epsilon > 0$. Now let $\epsilon = (1-\rho(A))/2$.
\end{proof}

We can also bound the trace of the product of two matrices.
\begin{lemma}[\textbf{A trace inequality}]
	\label{lem:trace_inequality}
	Given a positive semi-definite matrix $A \in \mathbb{R}^{R \times R} \geq 0$ and a matrix $B \in \mathbb{R}^{R \times R}$, it holds that:
	\begin{align}
		\Tr(A B) \leq \Tr(A) \cdot \|B\| \label{eq:trace_inequality}
	\end{align}
	where $\|\cdot\|$ is the $2-$induced norm (or maximum singular value). Furthermore, when $B$ is also positive semi-definite, we have that:
	\begin{align}
	0 \leq \Tr(A B) \leq \Tr(A) \cdot \|B\| \leq \Tr(A) \cdot \Tr(B) \label{eq:trace_inquality_PSD}
	\end{align}
\end{lemma}
\begin{proof}
Relation \eqref{eq:trace_inequality} was proved in \cite{general_trace_inequality}. When the matrix $B$ is also positive semi-definte, we have from \cite{Athans,trace_inequality} that:
\begin{align} 
\lambda_{\min}(B) \Tr(A) \leq \Tr(A B) \leq \lambda_{\max}(B)\cdot \Tr(A) = \|B\| \cdot \Tr(A) \nonumber
\end{align}
Since the matrix $B$ is positive semi-definite, we have that $\lambda_{\min}(B) \geq 0$. Furthermore, we may use the inequality $\lambda_{\max}(B) \leq \Tr(B)$ since $B$ is positive semi-definite to obtain \eqref{eq:trace_inquality_PSD}
\end{proof}

\subsection{Convergence of Inequality Recursions}
We list the following lemma from \cite[p.~45]{Polyak}, which demonstrates the convergence of a deterministic recursion.
\begin{lemma}[\textbf{Convergence of a deterministic recursion}]
\label{lem:Polyak}
	Let a sequence $v(i)$ satisfy 
	\begin{align*}
		&v(i) \leq (1-\eta(i-1)) v(i-1) + \zeta(i-1)\\
		&\sum_{i=0}^\infty \eta(i) = \infty, \quad 0 \leq \eta(i) < 1, \quad  \zeta(i) \geq 0, \quad \frac{\zeta(i)}{\eta(i)} \rightarrow 0.
	\end{align*}
	Then, 
	\begin{align}
		\limsup_{i\rightarrow\infty} v(i) \leq 0
	\end{align}
	and if $v(i)\geq 0$, then $v(i)\rightarrow 0$.\hfill\qed
\end{lemma}

In addition to demonstrating convergence of recursions of the above form, we are also interested in the convergence \textit{rate} for special cases of $\eta(i)$. To this end, we establish the following lemma, which extends Lemma 4 \cite[p.~45]{Polyak}.
\begin{lemma}[\textbf{Convergence rate of a deterministic recursion}]
\label{lem:convergence_deterministic_sequence}
Assume that $f > 0, d > 0, p > 0$ with $f > p$ and let the deterministic sequence $v(i) \geq 0$ satisfy
\begin{align}
	v(i) \leq \left(1-\frac{f}{i} + O\bigg(\frac{1}{i^{2}}\bigg)\right) v(i-1) + \frac{d}{i^{p+1}} \label{eq:condition_on_v}
\end{align}
Then,
\begin{align}
	\limsup_{i\rightarrow \infty} \frac{v(i)}{i^{-p}} \leq \frac{d}{f-p} \label{eq:rate_lemma}
\end{align}
\end{lemma}
\begin{proof}
Define the sequence 
\begin{align}
	u(i) \triangleq \frac{v(i)}{i^{-p}} - \frac{d}{f-p} \label{eq:definition_ui}
\end{align}
Then, it holds that:
\begin{align}
	u(i+1) &= (i+1)^p v(i+1) - \frac{d}{f-p} \nonumber\\
			&\stackrel{(a)}{\leq} \!\!(i\!+\!1)^p \!\left(\!\!\bigg(1 \!-\! \frac{f}{i\!+\!1} \!+\! O\bigg(\!\frac{1}{i^{2}}\!\bigg)\!\!\bigg) v(i) \!+\! \frac{d}{(i\!+\!1)^{p+1}}\!\right) \!- \nonumber\\
			&\quad\ \frac{d}{f-p} \nonumber\\
			&\stackrel{(b)}{=} \frac{v(i)}{i^{-p}} \Big(\!1\!+\!\frac{p}{i}\!+\!O\Big(\frac{1}{i^{2}}\Big)\!\Big) \Big(1\!-\!\frac{f}{i\!+\!1}\!+\!O\Big(\frac{1}{i^{2}}\Big)\Big)  + \nonumber\\
			&\quad\ \frac{d}{i+1} - \frac{d}{f-p} \nonumber\\
			&\stackrel{(c)}{=} \frac{v(i)}{i^{-p}} \Big(1-(f-p)/(i+1)+O(1/i^{2})\Big) + \nonumber\\
			&\quad\ \frac{d}{i+1} - \frac{d}{f-p} \nonumber\\
			&= \left(1-\frac{f-p}{i+1}+O\left(\frac{1}{i^2}\right)\right) u(i) + O\left(\frac{1}{i^2}\right)
\end{align}
where step $(a)$ is due to \eqref{eq:condition_on_v}, step $(b)$ is due to the series expansion of $(1+1/i)^p$ about $i=\infty$, and step $(c)$ is due to:
\begin{align}
\Big(1+p/i+&O(1/i^{2})\Big) \Big(1-f/(i+1)+O(1/i^{2})\Big) \nonumber\\
&= 1 - (f - p)/(i+1) + O(1/i^{2})
\end{align}
Recall that we assumed $f > p$ and, therefore,
\begin{align}
	\sum_{i=1}^\infty \left(\frac{f-p}{i+1}+O\bigg(\frac{1}{i^{2}}\bigg)\right) &= \infty
\end{align}
and
\begin{align}
	\lim_{i\rightarrow\infty} \frac{O(1/i^{2})}{\frac{f-p}{i+1}+O(1/i^{2})} &= 0
\end{align}
We now call upon Lemma \ref{lem:Polyak} to deduce that $\limsup_{i\rightarrow\infty} u(i) \leq 0$, which leads to  \eqref{eq:rate_lemma}.
\end{proof}

We also have the following stochastic analogue to Lemma \ref{lem:Polyak} \cite[pp.~49--50]{Polyak}.
\begin{lemma}[\textbf{Convergence of a stochastic recursion}]
\label{lem:Gladyshev}
	Let there be a sequence of random variables $\v_0,$ $\ldots,\v_i \geq 0$, $\E \v_0 < \infty$, satisfying 
	\begin{align}
		\E\{\v(i) | \v(0),\ldots,\v(i\!-\!1)\} \leq (1\!-\!\eta(i\!-\!1)) \v(i-1) \!+\! \zeta(i\!-\!1)
	\end{align}
	with
	\begin{align}
		\sum_{k=0}^\infty \eta(i) \!=\! \infty, \; 0 \!\leq\! \eta(i) < 1, \; \zeta(i) \!\geq\! 0,
		\;
		\frac{\zeta(i)}{\eta(i)} \!\rightarrow\! 0, \; \sum_{i=0}^\infty \zeta(i) \!< \! \infty
	\end{align}
	Then, $\v(i) \rightarrow 0$ almost surely.\hfill\qed
\end{lemma}

\subsection{Finite Products and Behavior of Special Functions}
We now examine the behavior of finite products. We will observe that these products are closely related to different forms of the Gamma function defined in \eqref{eq:def:Gamma} \cite{IntegralsSeries}.
\begin{lemma}[\textbf{Bounds and identities on finite products}]
\label{lem:bounds}
Let $\mu > 0$, $\lambda_m > 0$, and $1 \leq j < i$, it holds that
\begin{align}
	\prod_{t=j+1}^i \left(1-\frac{\lambda \mu}{t}\right)^2 &= \frac{\Gamma^2(i+1-\lambda \mu)}{\Gamma^2(i+1)}\cdot\frac{\Gamma^2(j+1)}{\Gamma^2(j+1-\lambda \mu)}
	\label{eq:SS_finite_product_scalar}
\end{align}
where the Gamma function is defined by:
\begin{align}
	\Gamma(x) \triangleq \int_0^\infty t^{x-1} e^{-t} dt \label{eq:def:Gamma}
\end{align}

Furthermore, when $i$ is large, it holds that:
\begin{align}
		\underline{c}_{m,i} \cdot s_m \leq \prod_{j=1}^i \left(1-\frac{\lambda_m \mu}{j}\right)^2 \leq 
		\bar{c}_{m,i} \cdot s_m \label{eq:scalar_product_transient}
\end{align}
where
\begin{align}
	s_m &\triangleq e^{2\sum_{j=1}^{\left\lceil \lambda_m \mu \right\rceil+1}\log\left(1-\frac{\lambda_m \mu}{j}\right)} \label{eq:s_m}\\
	\underline{c}_{m,i}  &\triangleq \frac{f_{m,i}\cdot (\left\lceil\lambda_m\mu\right\rceil-\lambda_m\mu+1)^{2\lambda_m\mu}}{\left(1-\frac{\lambda_m\mu}{\left\lceil\lambda_m\mu\right\rceil+1}\right)^{2(\left\lceil\lambda_m\mu\right\rceil+1)}}\\
	\bar{c}_{m,i} &\triangleq \frac{f_{m,i+1}\cdot(\left\lceil\lambda_m\mu\right\rceil-\lambda_m\mu+2)^{2\lambda_m\mu}}{\left(1-\frac{\lambda_m\mu}{\left\lceil\lambda_m\mu\right\rceil+2}\right)^{2(\left\lceil\lambda_m\mu\right\rceil+2)}}\\
	f_{m,i} &\triangleq \frac{\left(1-\frac{\lambda_m\mu}{i}\right)^{2i}}{(i-\lambda_m\mu)^{2\lambda_m\mu}}
\end{align}
and $\lceil \cdot \rceil$ indicates the ceiling operator and $\log(\cdot)$ denotes the natural logarithm. 
\end{lemma}
\begin{proof}
For \eqref{eq:SS_finite_product_scalar}, we observe using the property  $\Gamma(x+1) = x \cdot \Gamma(x)$ that
	\begin{align}
		(t-\lambda\mu)^2 &= \frac{\Gamma^2(t+1-\lambda\mu)}{\Gamma^2(t-\lambda\mu)}
	\end{align}
	and, therefore,
	\begin{align}
		\prod_{t=j+1}^i \!\!(t-\lambda\mu)^2 &= \frac{\Gamma^{2}(j+1+1-\lambda\mu)}{\Gamma^{2}(j+1-\lambda\mu)}\!\cdots\!\frac{\Gamma^{2}(i+1-\lambda\mu)}{\Gamma^{2}(i-\lambda\mu)} \nonumber\\
		&= \frac{\Gamma^{2}(i+1-\lambda\mu)}{\Gamma^{2}(j+1-\lambda\mu)} \label{eq:product_expansion}
	\end{align}
	and, in the special case when $\lambda\mu=0$, 
	\begin{align}
		\prod_{t=j+1}^i t^2 &= \frac{\Gamma^{2}(j+1+1)}{\Gamma^{2}(j+1)}\!\cdots\!\frac{\Gamma^{2}(i+1)}{\Gamma^{2}(i)} = \frac{\Gamma^{2}(i+1)}{\Gamma^{2}(j+1)} \label{eq:product_expansion_lambdamu_0}
	\end{align}
	Then,
	\begin{align}
		\prod_{t=j+1}^i \left(1-\frac{\lambda \mu}{t}\right)^2 &= \prod_{t=j+1}^i \frac{(t-\lambda \mu)^2}{t^2} \nonumber\\
		&=  \frac{\prod_{t=j+1}^i (t-\lambda \mu)^2}{\prod_{t=j+1}^i t^2}\nonumber\\ 
		&\stackrel{(a)}{=} \frac{\Gamma^{2}(i+1-\lambda\mu)}{\Gamma^{2}(j+1-\lambda\mu)} \cdot \frac{\Gamma^2(j+1)}{\Gamma^2(i+1)}
	\end{align}
	where step $(a)$ is due to \eqref{eq:product_expansion}--\eqref{eq:product_expansion_lambdamu_0}. For \eqref{eq:scalar_product_transient}, we start from 
	\begin{align}
		\prod_{j=1}^i \left(1-\frac{\lambda_m \mu}{j}\right)^2 &= e^{2 \sum_{j=1}^i \log\left(1-\frac{\lambda_m \mu}{j}\right)} \nonumber\\
		&\stackrel{(a)}{=} s_m \cdot e^{2 \sum_{j=\left\lceil\lambda_m\mu\right\rceil+2}^i \log\left(1-\frac{\lambda_m \mu}{j}\right)} \label{eq:intermediate_1}
	\end{align}
	where step $(a)$ is valid since $i$ is assumed to be large and $s_m$ was defined in \eqref{eq:s_m}. Observe that $s_m$ is constant and does not depend on $i$. On the other hand, the second term in \eqref{eq:intermediate_1} can be bounded by using the following integral bounds for any increasing function $f(x)$:
	\begin{align}
		\int_{a-1}^b f(x) dx \leq \sum_{j=a}^b f(j) \leq \int_a^{b+1} f(x) dx \label{eq:integral_bounds_increasing_function}
	\end{align}
	Applying \eqref{eq:integral_bounds_increasing_function} to the function $f(x) = \log\left(1-\frac{\lambda_m \mu}{x}\right)$, which is an increasing function in $x$ for $x > \lambda_m\mu$, we get
	\begin{align}
	\begin{cases}
		 e^{2 \sum_{j=\left\lceil\lambda_m\mu\right\rceil+2}^i \log\left(1-\frac{\lambda_m \mu}{j}\right)} \geq e^{2 \int_{\left\lceil \lambda_m \mu \right\rceil+1}^i \log\left(1-\frac{\lambda_m \mu}{x}\right) dx} \\
		e^{2 \sum_{j=\left\lceil\lambda_m\mu\right\rceil+2}^i \log\left(1-\frac{\lambda_m \mu}{j}\right)} \leq   e^{2 \int_{\left\lceil \lambda_m \mu \right\rceil+2}^{i+1} \log\left(1-\frac{\lambda_m \mu}{x}\right) dx}
	\end{cases}
		\label{eq:trans_bounds_1}
	\end{align}
	We can evaluate the definite integrals by noting that the indefinite integral of $\log(1-\lambda_m\mu/x)$ is given by 
	\begin{align}
		\int \log\bigg(1-\frac{\lambda_m\mu}{x}\bigg)dx &= x\cdot  \log\bigg(1-\frac{\lambda_m\mu}{x}\bigg) - \nonumber\\
		&\quad\ \lambda_m\mu \log(x-\lambda_m\mu) \label{eq:indef_int}
\end{align}		
	Evaluating the integrals in \eqref{eq:trans_bounds_1} using \eqref{eq:indef_int}, we obtain the bounds
	\begin{align}
	\int_{\left\lceil\lambda_m\mu\right\rceil+1}^i \log\left(1-\frac{\lambda_m \mu}{x}\right) dx &= i \log\left(1\!-\!\frac{\lambda_m\mu}{i}\right) - \nonumber\\
	&\!\!\!\!\!\!\!\!\!\!\!\!\!\!\!\!\!\!\!\!\!\!\!\!\!\!\!\!\!\!\!\!\!\!\!\!\!\!\!\!\!\!\!\!\!\!\!\!\!\!\!\!\!\!\!\!\!\!\!\!\!\!\!\!\!\!\!\!\! \lambda_m\mu\log(i\!-\!\lambda_m\mu) \!-\! (\left\lceil\lambda_m\mu\right\rceil\!+\!1)\log\left(1\!-\!\frac{\lambda_m\mu}{\left\lceil\lambda_m\mu\right\rceil+1}\right)\!+\nonumber\\
	&\!\!\!\!\!\!\!\!\!\!\!\!\!\!\!\!\!\!\!\!\lambda_m\mu\log(\left\lceil\lambda_m\mu\right\rceil-\lambda_m\mu+1) 
	\label{eq:log_trans_lowerbound}
	\end{align}
	and
	\begin{align}
	\int_{\left\lceil\lambda_m\mu\right\rceil+2}^{i+1} \log\left(1-\frac{\lambda_m \mu}{x}\right) dx &= (i+1) \log\left(1\!-\!\frac{\lambda_m\mu}{i+1}\right)\! - \nonumber\\
	&\!\!\!\!\!\!\!\!\!\!\!\!\!\!\!\!\!\!\!\!\!\!\!\!\!\!\!\!\!\!\!\!\!\!\!\!\!\!\!\!\!\!\!\!\!\!\!\!\!\!\!\!\!\!\!\!\!\!\!\!\!\!\!\lambda_m\mu\log(i\!+\!1\!-\!\lambda_m\mu)+\lambda_m\mu\log(\left\lceil\lambda_m\mu\right\rceil-\lambda_m\mu+2)-\nonumber\\
	&\!\!\!\!\!\!\!\!\!\!\!\!\!\!\!\!\!\!\!\!\!\!\!\!\!\!\!\!\!\!\!\!\!\!\!\!\!\!\!\!(\left\lceil\lambda_m\mu\right\rceil+2)\log\left(1\!-\!\frac{\lambda_m\mu}{\left\lceil\lambda_m\mu\right\rceil+2}\right)
	\label{eq:log_trans_upperbound}
	\end{align}
	Substituting \eqref{eq:log_trans_lowerbound}-\eqref{eq:log_trans_upperbound} into \eqref{eq:trans_bounds_1}, we find  that the right-hand-sides of \eqref{eq:trans_bounds_1} reduce to 
	\begin{align}
		e^{2 \int_{\left\lceil \lambda_m \mu \right\rceil+1}^i \log\left(1-\frac{\lambda_m \mu}{x}\right) dx} &= \underline{c}_{m,i} \label{eq:LB_second_term}\\
		e^{2 \int_{\left\lceil \lambda_m \mu \right\rceil+2}^{i+1} \log\left(1-\frac{\lambda_m \mu}{x}\right) dx} &= \bar{c}_{m,i} \label{eq:UB_second_term}
	\end{align}		
	Combining \eqref{eq:LB_second_term}--\eqref{eq:UB_second_term} with \eqref{eq:intermediate_1}, we arrive at \eqref{eq:scalar_product_transient}.
\end{proof}

In \eqref{eq:SS_finite_product_scalar}, we established that a finite product can be interpreted in terms of Gamma functions. We are interested in the behavior of the Gamma function in the asymptotic regime for its argument. For this purpose, we call upon the following lemma from \cite{IntegralsSeries,tricomi1951}, which describes some properties of the Gamma function.
\begin{lemma}[\textbf{Properties of Gamma functions}]
\label{lem:Gamma}
The Gamma function satisfies: 
\begin{align}
	\lim_{|s| \rightarrow \infty} \frac{\Gamma^2(s+a)}{\Gamma^2(s+c)} s^{-2(a-c)} &= 1
	\label{eq:asym_frac_Gamma}
\end{align}
for $|\arg(s+a)| < \pi$ and
\begin{equation}
	\Gamma(s,x) = x^{s-1} e^{-x} \left[\sum_{m=0}^{M-1} \frac{(-1)^m \Gamma(1-s+\!m)}{x^m \Gamma(1-s)} + O\!\left(|x|^{-M}\right)\!\right]
	\label{eq:asym_upper_gamma}
\end{equation}
for $|x|\!\rightarrow\!\infty$, $-3\pi/2 < \arg(x) < 3 \pi/2$, and any integer $M\geq 1$. Here, the notation $\Gamma(s,x)$ denotes the upper incomplete Gamma function:
\begin{align}
\Gamma(s,x) \triangleq \int_x^\infty t^{s-1} e^{-t} dt
\end{align} \hfill\QED
\end{lemma}

\noindent Using the above lemma, we can now examine the convergence rate of a \textit{series} of Gamma functions. In Lemma \ref{lem:asym_constant_term}, we observe that a series of a fraction of Gamma functions will converge at different rates, $i^{-\nu}$ or $\log(i)/i$. And in Lemma \ref{lem:i^{-2}} we observe that when each term in the series is weighted exponentially, the convergence rate will become faster.
\begin{lemma}[\textbf{Series of Gamma functions}]
\label{lem:asym_constant_term}
Let $b \geq 0$, $q\geq 1$. Then,
\begin{align}
	\frac{1}{i^{2\lambda_m\mu}} &\sum_{j=q}^{i-1} j^{-b} \frac{\Gamma^2(j)}{\Gamma^2(j+1-\lambda_m\mu)} \nonumber\\
	&\sim \begin{cases}
																										\frac{i^{-b-1}}{2\lambda_m \mu -b - 1} + \Theta(1/i^{2\lambda_m\mu}), & 2\lambda_{m}\mu \neq 1+b\\
																										\frac{\log(i)}{i^{b+1}}, &2\lambda_{m}\mu = 1+b
																										\end{cases} \label{eq:result_thm:asym_const_term_higher_order}
\end{align}
\end{lemma}
\begin{proof}
	Observe that \eqref{eq:asym_frac_Gamma} states that $\Gamma^2(s+a)/\Gamma^2(s+c) \sim s^{2(a-c)}$ for large $s$. Thus, for any $\epsilon_2 > 0$, there exists $q' \geq 2$ such that for all $s\geq q'$, the difference between unity and the ratio of $\Gamma^2(s+a)/\Gamma^2(s+c)$ to its asymptotic function $s^{2(a-c)}$ is small:
		\begin{align}
			\left|\frac{\Gamma^2(s+a)}{\Gamma^2(s+c)}\cdot \frac{1}{s^{2(a-c)}} - 1\right| < \epsilon_2
		\end{align}
		which implies that
		\begin{align}
			(1-\epsilon_2) s^{2(a-c)} < \frac{\Gamma^2(s+a)}{\Gamma^2(s+c)} < (1+\epsilon_2) s^{2(a-c)} \label{eq:bound_using_s_a}
		\end{align}
		Setting $s = j, a = 0, c =  1-\lambda_m\mu$ into \eqref{eq:bound_using_s_a}, we obtain the following bound for the summand in \eqref{eq:result_thm:asym_const_term_higher_order}:		
	\begin{align}
		(1\!-\!\epsilon_2) j^{2\lambda_m\mu - 2-b} < \frac{j^{-b}\cdot \Gamma^2(j)}{\Gamma^2(j+1-\lambda_m\mu)} < (1\!+\!\epsilon_2) j^{2\lambda_m\mu - 2-b} \label{eq:epsilon_upper_lower_series}
	\end{align}
	Therefore, we can split the sum into two terms, one covering values from $q$ to $q'-1$, while the other summing over values from $q'$ to $i-1$:
	\begin{align}
		&\frac{1}{i^{2\lambda_m\mu}} \sum_{j=q}^{i-1}  \frac{j^{-b}\cdot \Gamma^2(j)}{\Gamma^2(j\!+\!1\!-\!\lambda_m\mu)} \nonumber\\
		&= \frac{1}{i^{2\lambda_m\mu}} \sum_{j=q}^{q'-1} \frac{j^{-b} \cdot \Gamma^2(j)}{\Gamma^2(j\!+\!1\!-\!\lambda_m\mu)} \!+\! \frac{1}{i^{2\lambda_m\mu}} \sum_{j=q'}^{i-1}  \frac{j^{-b}\cdot \Gamma^2(j)}{\Gamma^2(j\!+\!1\!-\!\lambda_m\mu)} \nonumber\\
		&\stackrel{(a)}{=} O(i^{-2\lambda_m\mu}) + \frac{1}{i^{2\lambda_m\mu}} \sum_{j=q'}^{i-1} \frac{j^{-b} \cdot \Gamma^2(j)}{\Gamma^2(j+1-\lambda_m\mu)} \label{eq:sum_of_terms_before_bound}
	\end{align}
	where step $(a)$ is due to the fact that the first term is proportional to $i^{-2\lambda_m\mu}$ (with no other dependence on $i$). We would like to obtain the convergence rate of the second term in \eqref{eq:sum_of_terms_before_bound}. For this, using \eqref{eq:epsilon_upper_lower_series}, we have
\begin{align}
	\frac{1}{i^{2\lambda_m\mu}} \sum_{j=q'}^{i-1} \frac{j^{-b} \cdot \Gamma^2(j)}{\Gamma^2(j+1-\lambda_m\mu)} < \frac{1+\epsilon_2}{i^{2\lambda_m\mu}} \sum_{j=q'}^{i-1} j^{2\lambda_m\mu - 2-b} \label{eq:second_term_gamma_frac_UB}
\end{align}	
and
\begin{align}
	\frac{1}{i^{2\lambda_m\mu}} \sum_{j=q'}^{i-1} \frac{j^{-b} \cdot \Gamma^2(j)}{\Gamma^2(j+1-\lambda_m\mu)} > \frac{1-\epsilon_2}{i^{2\lambda_m\mu}} \sum_{j=q'}^{i-1} j^{2\lambda_m\mu - 2-b} \label{eq:second_term_gamma_frac_LB}
\end{align}	
Therefore, we only need to study the behavior of $i^{-2\lambda_m\mu} \sum_{j=q'}^{i-1} j^{2\lambda_m\mu - 2-b}$. We rely on the following integral bounds, which are valid for any positive monotonic function $f(x)$:
	\begin{align}
		\int_{a}^{b} f(x) dx \leq \sum_{i=a}^b f(i) \leq \int_{a-1}^{b+1} f(x) dx \label{eq:integral_bound_general}
	\end{align}
	and use the following indefinite integral for the case where $2\lambda_m\mu - 2 - b \neq 1$:
	\begin{align}
		\int x^{2\lambda_m\mu - 2 -b} dx = \frac{x^{2\lambda_m\mu - 1 - b}}{2\lambda_m\mu -1 -b} + \mathrm{constant} \label{eq:integral_not_zero}
	\end{align}
	to obtain the following bounds 
	\begin{align}
		\frac{1}{i^{2\lambda_m\mu}} \sum_{j=q'}^{i-1} j^{2\lambda_m\mu - 2 - b} &\leq \frac{1}{i^{2\lambda_m\mu}} \int_{q'-1}^{i} x^{2\lambda_m \mu - 2-b} dx \nonumber\\
		&= \frac{i^{-b-1}}{2\lambda_m \mu -b - 1} + O(i^{-2\lambda_m\mu}) \label{eq:not_zero_UB}
	\end{align}
	Substituting \eqref{eq:not_zero_UB} into \eqref{eq:second_term_gamma_frac_UB}, we obtain the following upper-bound for the second term of \eqref{eq:sum_of_terms_before_bound}:
	\begin{align}
		\frac{1}{i^{2\lambda_m\mu}} \!\sum_{j=q'}^{i-1} \!\frac{j^{-b} \cdot \Gamma^2(j)}{\Gamma^2(j+1-\lambda_m\mu)} < \frac{(1\!+\!\epsilon_2)\cdot i^{-b-1}}{2\lambda_m \mu -b - 1} \!+\! O(i^{-2\lambda_m\mu}) \label{eq:second_term_final_UB}
	\end{align}
	Similarly, we can use the integral bound \eqref{eq:integral_bound_general} to obtain the lower-bound:
	\begin{align}
		\frac{1}{i^{2\lambda_m\mu}} \sum_{j=q'}^{i-1} j^{2\lambda_m\mu - 2 - b} &\geq \frac{1}{i^{2\lambda_m\mu}} \int_{q'}^{i-1} x^{2\lambda_m \mu - 2-b} dx \nonumber\\
		&\!\!\!\!\!\!\!\!\!\!\!\!\!\!\!\!\!\!\!\!\!\!\!\!\!= 	\frac{1}{i^{2\lambda_m\mu}} \left(\frac{(i-1)^{2\lambda_m\mu - 1 - b}}{2\lambda_m\mu -1 -b} - \frac{(q')^{2\lambda_m\mu - 1 - b}}{2\lambda_m\mu -1 -b}\right)	\nonumber\\
		&\!\!\!\!\!\!\!\!\!\!\!\!\!\!\!\!\!\!\!\!\!\!\!\!\!\sim \frac{i^{-b-1}}{2\lambda_m \mu -b - 1} + O(i^{-2\lambda_m\mu}) \label{eq:not_zero_LB}
	\end{align}
	Substituting \eqref{eq:not_zero_LB} into \eqref{eq:second_term_gamma_frac_LB}, we obtain the following lower-bound for the second term of \eqref{eq:sum_of_terms_before_bound}:
	\begin{align}
		\frac{1}{i^{2\lambda_m\mu}} \!\sum_{j=q'}^{i-1} \!\frac{j^{-b} \cdot \Gamma^2(j)}{\Gamma^2(j+1-\lambda_m\mu)} > \frac{(1\!-\!\epsilon_2)\cdot i^{-b-1}}{2\lambda_m \mu -b - 1} \!+\! O(i^{-2\lambda_m\mu})\label{eq:second_term_final_LB}
	\end{align}

	We conclude from \eqref{eq:second_term_final_UB} and \eqref{eq:second_term_final_LB} that for large $i$, when 	$2\lambda_m \mu -b -1 \neq 0$, and for any $\epsilon_2 > 0$, 
	\begin{align}
		\left|\sum_{j=q'}^{i-1} \!\frac{i^{-2\lambda_m\mu} \cdot \Gamma^2(j)}{\Gamma^2(j+1-\lambda_m\mu)} -\frac{i^{-b-1}}{2\lambda_m \mu -b - 1} \!-\! O(i^{-2\lambda_m\mu})\right| < \epsilon_2
	\end{align}
	which implies that, as $i\rightarrow\infty$,
	\begin{align}
		\sum_{j=q'}^{i-1} \!\frac{i^{-2\lambda_m\mu} \cdot \Gamma^2(j)}{\Gamma^2(j+1-\lambda_m\mu)} \sim \frac{i^{-b-1}}{2\lambda_m \mu -b - 1} \!+\! O(i^{-2\lambda_m\mu}) \label{eq:first_line_b}
	\end{align}
	when  $2\lambda_m \mu -b -1 \neq 0$, which is the first line of \eqref{eq:result_thm:asym_const_term_higher_order}.
	
	On the other hand, when $2\lambda_m \mu -b -1 = 0$, we may use the following indefinite integral in place of \eqref{eq:integral_not_zero}:
	\begin{align}
		\int x^{2\lambda_m\mu - 2 -b} dx = \int \frac{1}{x} dx = \log(x) + \mathrm{constant} \label{eq:integral_zero}
	\end{align}
	Then, we have that
	\begin{align}
		\frac{1}{i^{1+b}} \sum_{j=q'}^{i-1} j^{-1} &= \frac{1}{i^{1+b}} \sum_{j=q'-1}^{i} \frac{1}{j}  \nonumber\\
		&\leq \frac{1}{i^{1+b}} \int_{q'-1}^{i} \frac{1}{x} dx \nonumber\\ 
		&= \frac{\log(i)}{i^{b+1}} + O(i^{-1-b}) \label{eq:zero_UB}
	\end{align}
		Substituting \eqref{eq:zero_UB} into \eqref{eq:second_term_gamma_frac_UB}, we obtain the following upper-bound for the second term of \eqref{eq:sum_of_terms_before_bound}:
	\begin{align}
		\frac{1}{i^{1+b}} \!\sum_{j=q'}^{i-1} \!\frac{j^{-b} \cdot \Gamma^2(j)}{\Gamma^2(j+1-\lambda_m\mu)} < (1+\epsilon_2) \cdot  \frac{\log(i)}{i^{b+1}} + O(i^{-1-b})  \label{eq:second_term_final_UB_zero}
	\end{align}
	Similarly, we can use the integral bound \eqref{eq:integral_bound_general} to obtain the lower-bound:
	\begin{align}
		\frac{1}{i^{1+b}} \sum_{j=q'}^{i-1} j^{-1} &\geq \frac{1}{i^{1+b}} \int_{q'}^{i-1} \frac{1}{x} dx 
		\sim\frac{\log(i)}{i^{b+1}} + O(i^{-1-b}) \label{eq:zero_LB}
	\end{align}	
	Substituting \eqref{eq:zero_LB} into \eqref{eq:second_term_gamma_frac_LB}, we obtain the following lower-bound for the second term of \eqref{eq:sum_of_terms_before_bound}:
	\begin{align}
		\frac{1}{i^{1+b}} \!\sum_{j=q'}^{i-1} \!\frac{j^{-b} \cdot \Gamma^2(j)}{\Gamma^2(j+1-\lambda_m\mu)} > (1-\epsilon_2) \frac{\log(i)}{i^{b+1}} \!+\! O(i^{-1-b})\label{eq:second_term_final_LB_zero}
	\end{align}
	We conclude from \eqref{eq:second_term_final_UB} and \eqref{eq:second_term_final_LB} that when 	$2\lambda_m \mu -b -1 = 0$, and for any $\epsilon_2 > 0$, 
	\begin{align}
		\left|\sum_{j=q'}^{i-1} \!\frac{i^{-2\lambda_m\mu} \cdot \Gamma^2(j)}{\Gamma^2(j+1-\lambda_m\mu)} -\frac{\log(i)}{i^{b+1}} \!-\! O(i^{-b-1})\right| < \epsilon_2
	\end{align}
	which implies that, as $i\rightarrow\infty$,
	\begin{align}
		\sum_{j=q'}^{i-1} \!\frac{i^{-b-1} \cdot \Gamma^2(j)}{\Gamma^2(j+1-\lambda_m\mu)} \sim \frac{\log(i)}{i^{b+1}} \label{eq:second_line_b}
	\end{align}
	when  $2\lambda_m \mu -b -1 = 0$, which is the second line of \eqref{eq:result_thm:asym_const_term_higher_order}.
\end{proof}
Compared with Lemma \ref{lem:asym_constant_term}, we now scale each term in the series by an exponentially decaying weight. We will observe that the convergence rate becomes faster than in Lemma \ref{lem:asym_constant_term}.
\begin{lemma}[\textbf{Series of weighted Gamma functions}]
\label{lem:i^{-2}}
Let $q \geq 1$, $b \geq 0$, and $0 < \rho < 1$. Then we have the following:
\begin{align}
	\frac{1}{i^{2\lambda_m\mu}} \sum_{j=q}^{i-1} \rho^{i-j} \frac{j^{-b} \Gamma^2(j)}{\Gamma^2(j+1-\lambda_m\mu)} &= \Theta(1/i^{2+b}),\quad\quad\!\!\!\!\!\mathrm{as\ }i\rightarrow\infty \label{eq:sum_exponentially_weighted}
\end{align}
and 
\begin{align}
	\frac{\rho}{\log(\rho^{-1})} \leq \lim_{i\rightarrow\infty} \frac{i^{2+b}}{i^{2\lambda_m\mu}} \!\cdot\!  \sum_{j=q}^{i-1}  \frac{\rho^{i-j} j^{-b} \Gamma^2(j)}{\Gamma^2(j+1-\lambda_m\mu)} \leq \frac{1}{\log(\rho^{-1})} \label{eq:UB_LB_exp_weighted}
\end{align}
\end{lemma}
\begin{proof}
	Notice that due to \eqref{eq:asym_frac_Gamma}, we have that for any $\epsilon_2 > 0$, there exists integer $q' \geq 2$ such that for all $j \geq q'$, \eqref{eq:epsilon_upper_lower_series} is satisfied. Therefore, we may divide the sum in \eqref{eq:sum_exponentially_weighted} into two parts, $j < q'$, and $j\geq q'$:
	\begin{align}
		&\frac{1}{i^{2\lambda_m\mu}} \sum_{j=q}^{i-1} \frac{\rho^{i-j} j^{-b} \cdot \Gamma^2(j)}{\Gamma^2(j+1-\lambda_m\mu)} \nonumber\\
		&= \frac{1}{i^{2\lambda_m\mu}} \!\left(\!\!\rho^i\!\sum_{j=1}^{q'-1} \frac{\rho^{-j}\cdot j^{-b}\cdot \Gamma^2(j)}{\Gamma^2(j+1-\lambda_m\mu)} \!+\!\! \sum_{j=q'}^{i-1}  \frac{\rho^{i-j}\cdot j^{-b}\cdot\Gamma^2(j)}{\Gamma^2(j+1-\lambda_m\mu)}\!\!\right) \nonumber\\
		&= O(i^{-2\lambda_m\mu} \rho^{i}) + \frac{1}{i^{2\lambda_m\mu}}  \sum_{j=q'}^{i-1}  \frac{\rho^{i-j}\cdot j^{-b}\cdot\Gamma^2(j)}{\Gamma^2(j+1-\lambda_m\mu)}\label{eq:split_exp_weight}
	\end{align}
Using the right-most bound in \eqref{eq:epsilon_upper_lower_series} for any $\epsilon_2 >0$ and large enough $q'$:
\begin{align}
&\frac{1}{i^{2\lambda_m\mu}} \sum_{j=1}^{i-1} \frac{\rho^{i-j}\cdot \Gamma^2(j)}{\Gamma^2(j+1-\lambda_m\mu)} \nonumber\\
&\leq O(i^{-2\lambda_m\mu} \rho^{i}) + \frac{(1+\epsilon_2)\rho^{i}}{i^{2\lambda_m\mu}} \sum_{j=q'}^{i-1} \rho^{-j} j^{2\lambda_m\mu-2-b}\nonumber\\
	&\stackrel{(a)}{\leq} O(i^{-2\lambda_m\mu} \rho^{i}) + \frac{(1+\epsilon_2)\rho^{i}}{i^{2\lambda_m\mu}} \int_{q'-1}^{i} e^{-x\log(\rho)} x^{2\lambda_m\mu-2-b} dx \label{eq:UB_exp_weight_def_int}
\end{align}
where step $(a)$ is due to \eqref{eq:integral_bound_general}.
Now, in order to compute the definite integral on the right-hand-side of  \eqref{eq:UB_exp_weight_def_int}, we use the following indefinite integral \cite[p.~108]{IntegralsSeries}:
\begin{align}
	\int x^m e^{-\beta x^n} dx = -\frac{\Gamma(\gamma,\beta x^n)}{n \beta^\gamma}+ \mathrm{constant}
\end{align}
where $\gamma \triangleq (m+1)/n$, $\beta \neq 0$, and $n\neq 0$. Making the identifications: $m \leftarrow 2\lambda_m\mu-2-b, \quad n\leftarrow 1, \quad \beta \leftarrow \log(\rho)$, we get
\begin{align}
	\int e^{-x\log(\rho)} x^{2\lambda_m\mu-2-b} dx &= -\frac{\Gamma(2\lambda_m\mu-1-b,\log(\rho) x)}{(\log(\rho))^{2\lambda_m\mu-1-b}} + \nonumber\\
	&\quad\  \mathrm{constant} \label{eq:indef_int2}
\end{align}
Using \eqref{eq:indef_int2} in conjunction with \eqref{eq:UB_exp_weight_def_int}, we obtain the following:
\begin{align}
	&\frac{1}{i^{2\lambda_m\mu}} \sum_{j=1}^{i-1} \frac{\rho^{i-j}\cdot \Gamma^2(j)}{\Gamma^2(j+1-\lambda_m\mu)} \leq O(i^{-2\lambda_m\mu} \rho^{i}) + \nonumber\\ 
	&\frac{(1+\epsilon_2)\rho^{i}}{i^{2\lambda_m\mu}} \!\!\Bigg[\frac{
	\Gamma(2\lambda_m\mu-1-b,(q'\!-\!1)\log(\rho))}{\left(\log(\rho)\right)^{2\lambda_m\mu-1-b}}- \nonumber\\
	&\quad\quad\quad\quad\quad\!\!\frac{
	\Gamma(2\lambda_m\mu-1-b,i\log(\rho))}{\left(\log(\rho)\right)^{2\lambda_m\mu-1-b}}\Bigg] \label{eq:UB_exp_weight_def_int2}
\end{align}
Now, observe that the first term inside the bracket of \eqref{eq:UB_exp_weight_def_int2} is proportional to $i^{-2\lambda_m\mu} \rho^i$, so we may combine it with the first term to obtain:
\begin{align}
\frac{1}{i^{2\lambda_m\mu}} \sum_{j=1}^{i-1} \frac{\rho^{i-j}\cdot \Gamma^2(j)}{\Gamma^2(j+1-\lambda_m\mu)} &\leq O(i^{-2\lambda_m\mu} \rho^{i}) - \nonumber\\
&\!\!\!\!\!\!\!\!\!\!\!\!\!\!\!\!\!\!\!\!\!\!\!\!\!\!\!\!\!\!\!\!\!\!\!\!\!\!\!\frac{(1+\epsilon_2)\rho^{i}}{i^{2\lambda_m\mu}} \cdot \frac{
	\Gamma(2\lambda_m\mu-1-b,i\log(\rho))}{\left(\log(\rho)\right)^{2\lambda_m\mu-1-b}} \label{eq:second_to_last_weighted_UB}
\end{align}
What remains to be done is to characterize the convergence rate of the second term in \eqref{eq:second_to_last_weighted_UB}. To do this, we utilize the asymptotic expansion of the upper incomplete Gamma function listed in \eqref{eq:asym_upper_gamma} with the identifications $s \leftarrow 2\lambda_m\mu-1-b, x \leftarrow i\log(\rho), M\leftarrow 1$ and write
\begin{align}	
\Gamma(&2\lambda_m\mu-1-b,i\log(\rho)) \nonumber\\ 
&= (i\log(\rho))^{2\lambda_m\mu-2-b} \cdot e^{-i\log(\rho)} \left[1 + O\!\left(|i\log(\rho))|^{-1}\right)\right] \nonumber\\
&= i^{2\lambda_m\mu-2-b} (\log(\rho))^{2\lambda_m\mu-2-b} \rho^{-i} + O(\rho^{-i} i^{2\lambda_m\mu-3-b}) \label{eq:inc_gamma_expansion}
\end{align}
Substituting \eqref{eq:inc_gamma_expansion} into \eqref{eq:second_to_last_weighted_UB}, we obtain
\begin{align}
&\frac{1}{i^{2\lambda_m\mu}} \sum_{j=1}^{i-1} \frac{\rho^{i-j}\cdot \Gamma^2(j)}{\Gamma^2(j+1-\lambda_m\mu)} \nonumber\\
&\leq O(i^{-2\lambda_m\mu} \rho^{i}) - \frac{(1+\epsilon_2)\rho^{i}}{i^{2\lambda_m\mu}} \frac{
	1}{(\log(\rho))^{2\lambda_m\mu-1-b}} \times\nonumber\\
	&\quad\  \left[i^{2\lambda_m\mu-2-b} (\log(\rho))^{2\lambda_m\mu-2-b} \rho^{-i} + O(\rho^{-i} i^{2\lambda_m\mu-3-b})\right]\nonumber\\
	&= O(i^{-2\lambda_m\mu} \rho^{i}) + (1+\epsilon_2) \frac{i^{-2-b}}{\log(\rho^{-1})}  + O(i^{-3-b}) \label{eq:final_LB_exp_weight}
\end{align}
The lower bound is derived in a similar way starting with \eqref{eq:split_exp_weight}, except for taking the integral limits for the lower-bound to be $\{q,i-1\}$:
\begin{align}
	\frac{1}{i^{2\lambda_m\mu}} \sum_{j=1}^{i-1} \frac{\rho^{i-j}\cdot \Gamma(j)^2}{\Gamma(j+1-\lambda_m\mu)^2} &\geq O(i^{-2\lambda_m\mu} \rho^{i}) + O(i^{-3-b})\nonumber\\
	&\quad\ \frac{(1-\epsilon_2)\rho\cdot  i^{-2-b}}{\log(\rho^{-1})}\label{eq:final_UB_exp_weight}
\end{align}
as $i\rightarrow\infty$. Since the upper and lower-bounds \eqref{eq:final_LB_exp_weight}--\eqref{eq:final_UB_exp_weight} hold for any $\epsilon_2>0$ and large $i$, we obtain \eqref{eq:UB_LB_exp_weighted}.
\end{proof}

The final intermediate result we will establish concerns the convergence rate of a sequence of traces.
\begin{lemma}[\textbf{Convergence rate of a series of traces}]
\label{lem:series_of_traces}
Assume that $\Sigma = \frac{1}{2} E_{kk} \!\otimes\! H$, $b\geq 0$, $q\geq 1$, and let ${\cal L}$ denote a block diagonal matrix: 
\begin{align}
	\mathcal{L} \triangleq \textrm{blockdiag}\{L_{1},\ldots,L_N\} \label{eq:L_blockdiag}
\end{align}
where each $L_k \in \mathbb{R}^{M\times M}$. Then,
\begin{align}
	C(i) &\triangleq \sum_{j=q}^i j^{-b} \mu^2(j) \Tr(\mathcal{L} \Omega_{j,i}) \sim \frac{\mu^2}{2} \sum_{m=1}^M \lambda_m \alpha_{m,b}(i) p^\T L_m' p \label{eq:desired_C}
\end{align}
where $\mu(i) \triangleq \mu/i$, $\Omega_{j,i}$ is defined by \eqref{eq:def_Gamma},  
\begin{align}
	\alpha_{m,b}(i) \!\!\triangleq\!\! \begin{cases}
																										\frac{i^{-b-1}}{2\lambda_m \mu -b - 1} + \Theta(i^{-2\lambda_m\mu}), & 2\lambda_{m}\mu \neq 1+b\\
																										\frac{\log(i)}{i^{b+1}}, &2\lambda_{m}\mu = 1+b
																										\end{cases}
	\label{eq:alpha_m}
\end{align}
and
\begin{align}
L_m' \triangleq \mathrm{diag}\{(\Phi^\T L_{1} \Phi)_{mm}, (\Phi^\T L_{N} \Phi)_{mm}\} \label{eq:L_prime}
\end{align}
and $\Phi$ is the eigen-basis of the matrix $H$ as defined in Assumption \ref{ass:HessianAssumption} and \eqref{eq:eigendecomp_H}.
\end{lemma}
\begin{proof}
	Introduce the Jordan canonical form of matrix $A$:
	\begin{align}
		A = T D T^{-1} \label{eq:A_jordan1}
	\end{align}
	Now, recall the definition of the weighting matrix $\Omega_{j,i}$ from \eqref{eq:def_Gamma}, reproduced here for convenience:
	\begin{align}
	\Omega_{j,i} &\triangleq \left(\prod_{t=j+1}^{i-1} \B_t^\T\right) \!\!\Sigma \!\left(\prod_{t=j+1}^{i-1} \B_t^\T\right)^{\!\!\!\!\!\T} &&
	\end{align}
	Substituting  \eqref{eq:sigma}, \eqref{eq:B_i}, and \eqref{eq:A_jordan1} into $\Omega_{j,i}$, we obtain:
\begin{align}
	\Omega_{j,i} &= \frac{1}{2} TD^{i-j}T^{-1}E_{kk}T^{-\T}{D^\T}^{i-j}T^\T \otimes \Phi K_{j,i} \Phi^\T \label{eq:prod_F_j_i}
\end{align}
where $K_{j,i}$ was defined in \eqref{eq:K_i}, repeated here for convenience:
\begin{equation}
	K_{j,i} \triangleq \left(\!\prod_{t=j+1}^{i-1}\!(I_M \!-\! \mu(t) \Lambda)\!\right) \!\Lambda\! \left(\!\prod_{t=j+1}^{i-1}\!(I_M \!-\! \mu(t) \Lambda)\!\right)
\end{equation}	
Thus, we can re-write $C(i)$ as:
	\begin{align}
	C(i) &= \frac{1}{2} \sum_{j=q}^{i-1} j^{-b} \Tr\bigg(\mathcal{L} (TD^{i-j}T^{-1}E_{kk}T^{-\T}{D^\T}^{i-j}T^\T \otimes \nonumber\\
	&\quad\quad\quad\quad\quad\quad\quad\ \mu^2(j) \Phi K_{j,i} \Phi^\T)\bigg) \label{eq:asym_Tr_0}
	\end{align}
	We denote the entries of the matrix below by $[b_{mn}]$:
	\begin{align}
	TD^{i-j}T^{-1}E_{kk}T^{-\T}{D^\T}^{i-j}T^\T = \left[\begin{array}{cccc}
	b_{11} & b_{12} & \cdots & b_{1N} \\ 
	b_{21} & b_{22} & \cdots & b_{2,N} \\ 
	\vdots & \vdots & \ddots & \vdots \\
	b_{N1} & b_{N2} & \cdots & b_{NN}
	\end{array} \right] \label{eq:5}
\end{align}
so that
\begin{align}
	T&D^{i-j}T^{-1}E_{kk}T^{-\T}{D^\T}^{i-j}T^\T \otimes \mu^2(j) \Phi K_{j,i} \Phi^\T \nonumber\\
	&=\mu^2(j) \left[\begin{array}{ccc}
	b_{11}\Phi K_{j,i} \Phi^\T  & \cdots & b_{1N} \Phi K_{j,i} \Phi^\T \\ 
	\vdots & \ddots & \vdots \\ 
	b_{N1}\Phi K_{j,i} \Phi^\T & \cdots & b_{NN} \Phi K_{j,i} \Phi^\T
	\end{array} \right] \label{eq:second_matrix}
\end{align}
Substituting \eqref{eq:L_blockdiag} and \eqref{eq:second_matrix} into \eqref{eq:def_AT}, we obtain
\begin{align}
	C(i) &\triangleq \frac{1}{2} \sum_{j=1}^{i-1} j^{-b} \mu^2(j) \Tr\!\left(\left[\begin{array}{ccc}
	L_1 &  &  \\ 
	 & \ddots &  \\ 
	 &  & L_N
	\end{array} \right] \times\right.\nonumber\\
	&\quad\quad\quad\quad\quad\left.\left[\begin{array}{ccc}
	b_{11}\Phi K_{j,i} \Phi^\T  & \cdots & b_{1N} \Phi K_{j,i} \Phi^\T \\ 
	\vdots & \ddots & \vdots \\ 
	b_{N1}\Phi K_{j,i} \Phi^\T & \cdots & b_{NN} \Phi K_{j,i} \Phi^\T
	\end{array} \right]\right)\ \nonumber\\
	&= \frac{1}{2} \sum_{j=1}^{i-1} j^{-b} \mu^2(j) \sum_{k=1}^N b_{kk} \Tr\left(\Phi^\T L_k\Phi K_{j,i} \right) \label{eq:AT2}
\end{align}
Now, we know that the matrix $K_{j,i}$ is diagonal with diagonal elements:
\begin{align}
	[K_{j,i}]_{mm} = \lambda_m \prod_{t=j+1}^{i-1} (1-\mu(t) \lambda_m)^2 \label{eq:11}
\end{align}
Let us denote the matrix $\Phi^\T L_k\Phi$ by
\begin{align}
	L_k' \triangleq \Phi^\T L_k \Phi \label{eq:12}
\end{align}
so that
\begin{align}
	L_k' K_{j,i} &= \left[
	\begin{array}{cccc}
	\left[L_k'\right]_{11} & \left[L_k'\right]_{12} & \cdots & \left[L_k'\right]_{1M} \\
	\left[L_k'\right]_{21} & \left[L_k'\right]_{22} & \cdots & \left[L_k'\right]_{2M}\\
	\vdots & \vdots & \ddots & \vdots\\
	\left[L_k'\right]_{M1} & \left[L_k'\right]_{M2} & \cdots & \left[L_k'\right]_{MM}\\
	\end{array} 
	\right] \times \nonumber\\
	&\quad\quad\  \left[\begin{array}{cccc}
	[K_{j,i}]_{11} &  &  &  \\ 
	 & [K_{j,i}]_{22} &  &  \\ 
	 &  & \ddots &  \\ 
	 &  &  & [K_{j,i}]_{MM}
	\end{array} \right]
\end{align}
and \eqref{eq:AT2} becomes
\begin{align}
C(i) &= \frac{1}{2} \sum_{j=1}^{i-1} j^{-b} \mu^2(j) \sum_{k=1}^N b_{kk} \times\nonumber\\
&\ \Tr\left(\left[
	\begin{array}{ccc}
	\left[K_{j,i}\right]_{11} \left[L_k'\right]_{11}  & \cdots & \left[K_{j,i}\right]_{MM} \left[L_k'\right]_{1M} \\
	\left[K_{j,i}\right]_{11}\left[L_k'\right]_{21}  & \cdots & \left[K_{j,i}\right]_{MM} \left[L_k'\right]_{2M}\\
	\vdots & \ddots & \vdots\\
	\left[K_{j,i}\right]_{11}\left[L_k'\right]_{M1} & \cdots & \left[K_{j,i}\right]_{MM} \left[L_k'\right]_{MM}
	\end{array} 
	\right] \right) \nonumber\\
	&= \frac{1}{2} \sum_{m=1}^M  \sum_{j=1}^{i-1} \left[K_{j,i}\right]_{mm} j^{-b}  \mu^2(j) \sum_{k=1}^N b_{kk}   \left[L_k'\right]_{mm} \label{eq:17}
\end{align}
Finally, note that 
\begin{align}
	\Tr&\left(\!\left[\!\!\begin{array}{cccc}
	b_{11} & b_{12} & \cdots & b_{1N} \\ 
	b_{21} & b_{22} & \cdots & b_{2,N} \\ 
	\vdots & \vdots & \ddots & \vdots \\
	b_{N1} & b_{N2} & \cdots & b_{NN}
	\end{array} \!\!\right] \!\!\left[\!\!\!\begin{array}{cccc}
	\left[L_1'\right]_{mm} &  &  &  \\ 
	 & \!\!\!\!\!\!\left[L_2'\right]_{mm}\!\!\!\!\!\! &  &  \\ 
	 &  & \!\!\!\!\!\!\ddots\!\!\!\!\!\! &  \\ 
	 &  &  & \!\!\left[L_N'\right]_{mm}\!\!
	\end{array} \!\!\!\right]\! \right)\nonumber\\
	&\quad\quad\quad\quad\quad\quad\quad\quad\quad\quad= \sum_{k=1}^N b_{kk}   \left[L_k'\right]_{mm} \label{eq:18}
\end{align}
Combining \eqref{eq:5}, \eqref{eq:11},\eqref{eq:12}, \eqref{eq:17}, and \eqref{eq:18}, we obtain:
\begin{align}
	C(i) &= \frac{1}{2} \sum_{m=1}^M \sum_{j=q}^{i-1}   j^{-b}  \left[\mu^2(j) \lambda_m \prod_{t=j+1}^{i-1} (1-\mu(t) \lambda_m)^2\right] \times \nonumber\\
	&\quad\quad\quad\ \Tr\left(T D^{i-j}T^{-1}E_{kk}T^{-\T}{D^\T}^{i-j}T^\T L_{m}' \right)
\label{eq:asym_Tr_3}
\end{align}
where $L_m'$ is defined by \eqref{eq:L_prime}. Next we observe that the product $\prod_{t=j+1}^{i-1} (1-\mu(t) \lambda_m)^2$ has the same form as the term in \eqref{eq:SS_finite_product_scalar} when $\mu(t) = \mu/t$ and, hence, it can be described by ratios of Gamma functions:
\begin{align}
	\prod_{t=j+1}^{i-1} (1-\mu(t) \lambda_m)^2 = \frac{\Gamma^2(i-\lambda_m\mu)}{\Gamma^2(i)}\cdot\frac{\Gamma^2(j+1)}{\Gamma^2(j+1-\lambda_m\mu)} \label{eq:gamma_prod_expansion}
\end{align}
We now utilize \eqref{eq:asym_frac_Gamma} with the identifications $s \leftarrow i, a \leftarrow -\lambda_m \mu, c \leftarrow 0$ to deduce that the first fraction in \eqref{eq:gamma_prod_expansion} will converge asymptotically as $i^{-2\lambda_m\mu}$. Therefore, we have
\begin{align}
	\mu^2(j) \lambda_m\!\!\! \prod_{t=j+1}^{i-1} \!\!\!\!(1\!-\!\mu(t) \lambda_m)^2 &\sim \frac{\lambda_m \mu^2 i^{-2\lambda_m\mu} \cdot \Gamma^2(j\!+\!1)}{j^2 \cdot \Gamma^2(j+1-\lambda_m\mu)} \nonumber\\
	&\stackrel{(a)}{=}  \frac{\lambda_m \mu^2 i^{-2\lambda_m\mu} \cdot \Gamma^2(j)}{\Gamma^2(j+1-\lambda_m\mu)}
	\label{eq:asym_finite_product}
\end{align}
where step $(a)$ is due to the fact that $x\cdot \Gamma(x) = \Gamma(x+1)$. Substituting \eqref{eq:asym_finite_product} into \eqref{eq:asym_Tr_3}, we get as $i\rightarrow \infty$:
\begin{align}
C(i) &= \frac{\mu^2}{2} \sum_{m=1}^M \frac{\lambda_m}{i^{2\lambda_m\mu}} \sum_{j=q}^{i-1} j^{-b} \frac{\Gamma^2(j)}{\Gamma^2(j+1-\lambda_m\mu)}\times\nonumber\\
&\quad\quad\quad\ \Tr\left(T D^{i-j}T^{-1}E_{kk}T^{-\T}{D^\T}^{i-j}T^\T L_m'\right)\label{eq:asym_Tr_4}
\end{align}
We use $\Tr(A^\T B C D^\T) = (\mathrm{vec}(A))^\T (D \otimes B) \mathrm{vec}(C)$ to expand the trace as:
\begin{align}
	&C(i) = \frac{\mu^2}{2} \!\sum_{m=1}^M (\textrm{vec}(T^\T L_m' T))^\T \lambda_m \times\nonumber\\
	&\left[\frac{1}{i^{2\lambda_m\mu}}\sum_{j=q}^{i-1} \frac{j^{-b} \cdot \Gamma^2(j)}{\Gamma^2(j\!+\!1\!-\!\lambda_m\mu)} (D\! \otimes\! D)^{i-j}\right]\!\!\textrm{vec}(T^{-1}E_{kk}T^{-\T})\label{eq:asym_Tr_5}
\end{align}
First, we will examine the following sum:
\begin{align}
\frac{1}{i^{2\lambda_m\mu}} \sum_{j=q}^{i-1} \frac{j^{-b} \cdot \Gamma^2(j)}{\Gamma^2(j+1-\lambda_m\mu)} (D \!\otimes\! D)^{i-j} \label{eq:orig}
\end{align}
Using \eqref{eq:decomposition_D} we have
\begin{align}
	(D \otimes D)^{i-j} &= D^{i-j} \otimes D^{i-j} \nonumber\\
	&= \left[\begin{array}{cc}
	1 &  \\ 
	 & D_{N-1}^{i-j}
	\end{array} \right] \otimes \left[\begin{array}{cc}
	1 &  \\ 
	 & D_{N-1}^{i-1}
	\end{array} \right] \nonumber\\
	&= \mathrm{diag}\{1, D_{N-1}^{i-j}, D_{N-1}^{i-j}, D_{N-1}^{i-1} \otimes D_{N-1}^{i-1}\} \label{eq:DoD}
\end{align}
Substituting \eqref{eq:DoD} into \eqref{eq:orig}, we obtain:
\begin{align}
	\frac{1}{i^{2\lambda_m\mu}} &\sum_{j=q}^{i-1} \frac{j^{-b} \cdot \Gamma^2(j)}{\Gamma^2(j+1-\lambda_m\mu)} (D \!\otimes\! D)^{i-j} \nonumber\\
	&= \frac{1}{i^{2\lambda_m\mu}} \sum_{j=q}^{i-1} \frac{j^{-b} \cdot \Gamma^2(j)}{\Gamma^2(j+1-\lambda_m\mu)} \times\nonumber\\
	&\quad\ \mathrm{diag}\{1, D_{N-1}^{i-j} , D_{N-1}^{i-j}, D_{N-1}^{i-1} \otimes D_{N-1}^{i-1}\}\nonumber\\
	&= \mathrm{diag}\{\tau(i), A_{1}(i), A_1(i), A_2(i)\}   \label{eq:asym_largest_expansion}
\end{align}
where
\begin{align}
	\tau(i) &\triangleq i^{-2\lambda_m\mu} \sum_{j=q}^{i-1} \frac{j^{-b} \cdot \Gamma^2(j)}{\Gamma^2(j+1-\lambda_m\mu)}\\
	A_1(i) &\triangleq i^{-2\lambda_m\mu}\sum_{j=q}^{i-1} \frac{j^{-b} \cdot \Gamma^2(j)}{\Gamma^2(j+1-\lambda_m\mu)} D_{N-1}^{i-j}\\
	A_2(i) &\triangleq i^{-2\lambda_m\mu}\sum_{j=q}^{i-1} \frac{j^{-b} \cdot \Gamma^2(j)}{\Gamma^2(j+1-\lambda_m\mu)} D_{N-1}^{i-j} \otimes D_{N-1}^{i-j}
\end{align}
From Lemma \ref{lem:asym_constant_term} we conclude that $\tau(i)$ converges asymptotically according to \eqref{eq:result_thm:asym_const_term_higher_order}:
\begin{align}
	\tau(i) &\sim \begin{cases}
																										\frac{i^{-b-1}}{2\lambda_m \mu -b - 1} + \Theta(i^{-2\lambda_m\mu}), & 2\lambda_{m}\mu \neq 1+b\\
																										\frac{\log(i)}{i^{b+1}}, &2\lambda_{m}\mu = 1+b
																										\end{cases} \nonumber\\
																										&=  \alpha_{m,b}(i)
\end{align}
where $\alpha_{m,b}(i)$ was defined in \eqref{eq:alpha_m}. We now examine the convergence rate of the other three sub-matrices in  \eqref{eq:asym_largest_expansion}. The second and third sub-matrices can be shown to converge to zero at a faster rate via:
\begin{align}
	\left\|A_1(i)\right\| &\leq \sum_{j=1}^{i-1} \frac{i^{-2\lambda_m\mu}\cdot j^{-b}\cdot\Gamma^2(j)}{\Gamma^2(j+1-\lambda_m\mu)} \|D_{N-1}^{i-j}\| \nonumber\\
	&\stackrel{(a)}{\leq} c\cdot \sum_{j=1}^{i-1} \frac{i^{-2\lambda_m\mu}\cdot j^{-b}\cdot\Gamma^2(j)}{\Gamma^2(j+1-\lambda_m\mu)} \left(\frac{\rho(D_{N-1})+1}{2}\right)^{i-j} \label{eq:bound_matrix}
\end{align}
\noindent where step $(a)$ is due to Lemma \ref{lem:matrix_power}. In order to evaluate the rate of decay of the above terms, we appeal to Lemma \ref{lem:i^{-2}} by making the  identification $\rho \leftarrow \frac{\rho(D_{N-1})+1}{2}$ where $\rho < 1$ since the matrix $A$ is primitive. We conclude that the matrix $A_1(i)$ will converge to the zero matrix at a faster rate than $\tau(i)$; i.e.,
\begin{align}
	A_1(i) = \Theta(1/i^{2+b}) 
\end{align}
In a similar manner, we observe that
\begin{equation}
	\left\|A_2(i)\right\| \leq c\cdot \sum_{j=1}^{i-1} \frac{i^{-2\lambda_m\mu}\cdot j^{-b}\cdot\Gamma^2(j)}{\Gamma^2(j+1-\lambda_m\mu)} \!\left(\!\frac{\rho(D_{N-1})+1}{2}\!\right)^{2(i-j)}
\end{equation}
Using Lemma \ref{lem:i^{-2}} with the following identification:
\begin{align}
	\rho \leftarrow \left(\frac{\rho(D_{N-1})+1}{2}\right)^{2}
\end{align}
we also conclude that $A_2(i) = \Theta(1/i^{2+b})$. Since the matrices $A_1(i)$ and $A_2(i)$ in \eqref{eq:asym_largest_expansion} will converge to zero at a relatively high rate, we can now deduce the following asymptotic relationship for \eqref{eq:asym_largest_expansion}:
\begin{align}
	\frac{1}{i^{2\lambda_m\mu}} \sum_{j=1}^{i-1} \frac{j^{-b}\cdot\Gamma^2(j)}{\Gamma^2(j+1-\lambda_m\mu)} (D \!\otimes\! D)^{i-j} \!\sim\! \alpha_{m,b}(i) E_{11} \!\otimes\! E_{11}
	\label{eq:inner_matrix}
\end{align}
where $\alpha_{m,b}(i)$ was defined in \eqref{eq:alpha_m} and $E_{11}$ is an $N\times N$ matrix with a single one at the $(1,1)$-th element, and all other elements are zero. Now, substituting \eqref{eq:inner_matrix} back into \eqref{eq:asym_Tr_5} and using the identity $\Tr(A^\T B C D^\T) = (\mathrm{vec}(A))^\T (D \otimes B) \mathrm{vec}(C)$ again, we get
\begin{align}
C(i) \!&=\! \frac{\mu^2}{2} \sum_{m=1}^M \lambda_m \alpha_{m,b}(i) \Tr\left(L_m' T E_{11} T^{-1}E_{kk}T^{-\T} E_{11}T^\T\right) \label{eq:asym_Tr_6}
\end{align}
We observe that $T E_{11} T^{-1}$ is a rank-$1$ matrix that is spanned by the left- and right-eigenvectors of $A$ corresponding to the eigenvalue at one. The left eigenvector is $\mathds{1}_N$ since $A$ is left-stochastic. The right eigenvector is the Perron vector, $p$, which is normalized so that the sum of its entries is equal to one, i.e., $p^\T \mathds{1}_N = 1$, and $A p = p$. Then, $T E_{11} T^{-1} = p \mathds{1}^\T_N$. Substituting into \eqref{eq:asym_Tr_6}, we arrive at the desired expression \eqref{eq:desired_C}. 
\end{proof}

\vspace{\baselineskip}
\bibliographystyle{IEEEbib}
\bibliography{refs}
 \newpage

\begin{IEEEbiography}[{\includegraphics[width=1in,height=1.25in,clip,keepaspectratio]{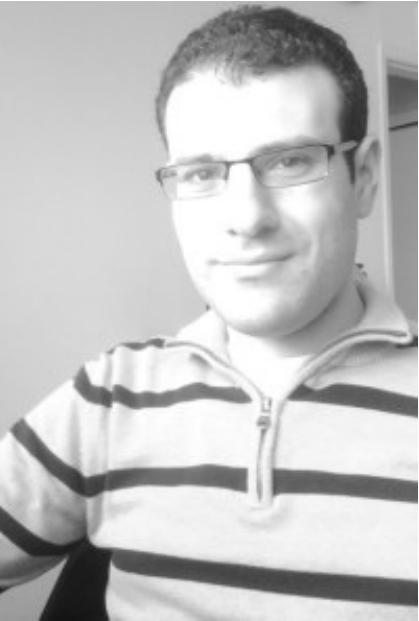}}]{Zaid J. Towfic}

(S'06--M'15) received his B.S. degrees in electrical engineering, computer science, and mathematics from the University of Iowa in 2007. He received his M.S. and Ph.D. degrees in electrical engineering from the University of California, Los Angeles (UCLA) in 2009 and 2014, respectively. Between 2008 and 2014, he was a member of the Adaptive Systems Laboratory at UCLA. From 2007 to 2008, he worked for Rockwell Collins' Advanced Technology Center (Cedar Rapids, IA) working on MIMO, embedded sensing, and digital modulation classification. He is currently a Technical Staff member at MIT Lincoln Laboratory (Lexington, MA). His research interests include machine learning, stochastic optimization, signal processing, adaptive systems, and information theory. 
\end{IEEEbiography}

\vspace{-8\baselineskip}
\begin{IEEEbiography}[{\includegraphics[width=1in,height=1.25in,clip,keepaspectratio]{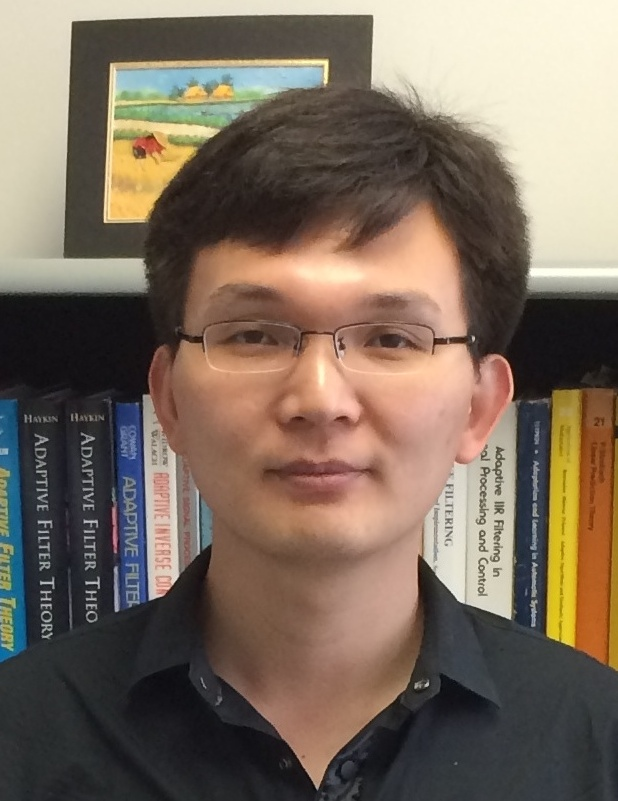}}]{Jianshu Chen}
(S'10--M'15) received his B.S. and M.S. degrees in Electronic Engineering from Harbin Institute of Technology in 2005 and Tsinghua University in 2009, respectively. He received his Ph.D degree in Electrical Engineering from University of California, Los Angeles (UCLA) in 2014. Between 2009 and 2014, he was a member of the Adaptive Systems Laboratory at UCLA. And between 2014 and 2015, he was a postdoctoral researcher at Microsoft Research, Redmond, WA. He is currently a researcher at Microsoft Research, Redmond, WA. His research interests include statistical signal processing, stochastic approximation, distributed optimization, machine learning, reinforcement learning, and smart grids. 
\end{IEEEbiography}

\vspace{-8\baselineskip}
\begin{IEEEbiography}[{\includegraphics[width=1in,height=1.25in,clip,keepaspectratio]{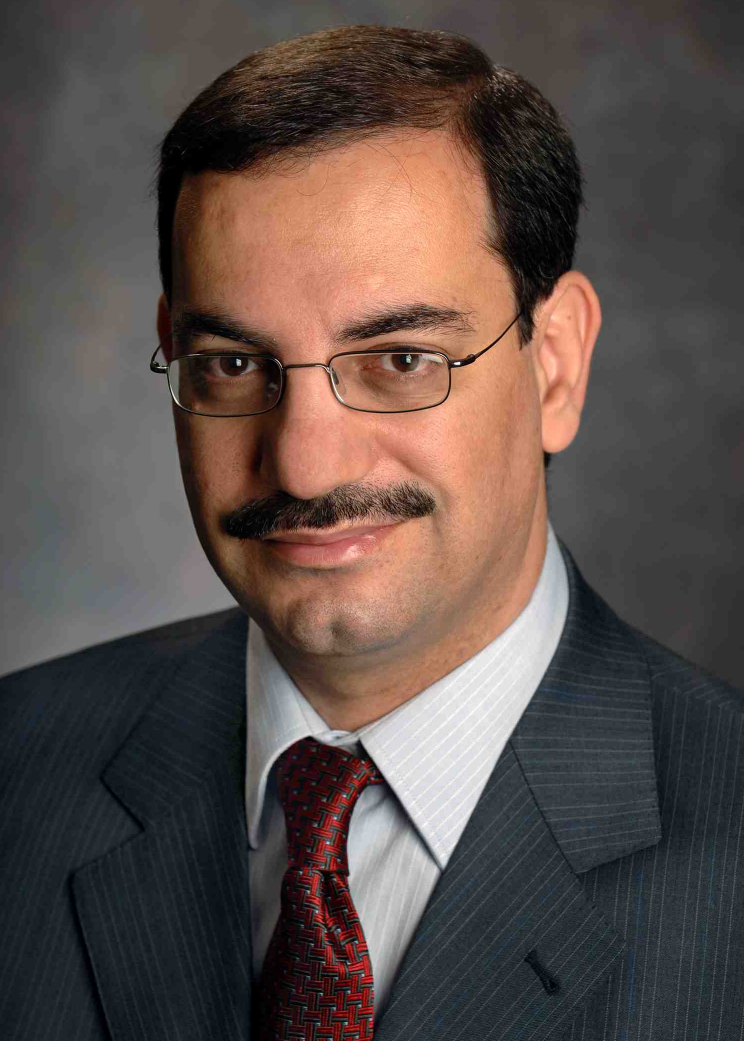}}]{Ali H. Sayed}
(S'90--M'92--SM'99--F'01) received the M.S. degree in electrical engineering from the University of Sao Paulo, Brazil, in 1989 and the Ph.D. degree in electrical engineering from Stanford University, Stanford, CA, in 1992. He is professor and former chairman of electrical engineering at the University of California, Los Angeles, USA, where he directs the UCLA Adaptive Systems Laboratory. An author of more than 480 scholarly publications and six books, his research involves several areas including adaptation and learning, statistical signal processing, distributed processing, network science, and biologically inspired designs. Dr. Sayed has received several awards including the 2015 Education Award from the IEEE Signal Processing Society, the 2014 Athanasios Papoulis Award from the European Association for Signal Processing, the 2013 Meritorious Service Award, and the 2012 Technical Achievement Award from the IEEE Signal Processing Society. Also, the 2005 Terman Award from the American Society for Engineering Education, the 2003 Kuwait Prize, and the 1996 IEEE Donald G. Fink Prize. He served as Distinguished Lecturer for the IEEE Signal Processing Society in 2005 and as Editor-in Chief of the \textsc{IEEE Transactions on Signal Processing} (2003-2005). His articles received several Best Paper Awards from the IEEE Signal Processing Society (2002, 2005, 2012, 2014). He is a Fellow of the American Association for the Advancement of Science (AAAS). He is recognized as a Highly Cited
Researcher by Thomson Reuters.
\end{IEEEbiography}

\end{document}